\documentclass[reqno, 11pt, a4paper]{amsart}
\usepackage{geometry}
\usepackage[utf8]{inputenc}
\usepackage[activeacute,english]{babel}
\usepackage[T1]{fontenc}
\usepackage{lmodern}

\usepackage{amsmath}
\usepackage{amsthm}
\usepackage{enumitem}
\usepackage{amsfonts,amssymb}
\usepackage{graphicx} 
\usepackage{mathtools}
\usepackage[abbrev]{amsrefs}
\usepackage[foot]{amsaddr}

\usepackage[pdftex,pdfpagelabels,bookmarks,hyperindex,hyperfigures]{hyperref}

\newtheorem{theorem}{Theorem}
\theoremstyle{plain}

\newtheorem{claim}[theorem]{Claim}
\newtheorem*{claim*}{Claim}

\newtheorem{construction}[theorem]{Construction}

\newtheorem{conjecture}[theorem]{Conjecture}

\newtheorem{definition}[theorem]{Definition}

\newtheorem{lemma}[theorem]{Lemma}

\newtheorem{proposition}[theorem]{Proposition}
\newtheorem{question}[theorem]{Question}
\newtheorem*{question*}{Question}

\numberwithin{equation}{section}
\numberwithin{theorem}{section}
\numberwithin{case}{section}

\numberwithin{subcase}{case}

\def\A{\mathcal{A}}

\def\E{\mathcal{E}}
\def\F{\mathcal{F}}

\def\K{\mathcal{K}}
\def\J{\mathcal{J}}

\def\eps{\varepsilon}

\def\cP{\mathcal{P}}
\def\pp{\mathcal{P}}

\def\T{\mathcal{T}}
\def\bfd{\mathbf{d}}

\def\h{H}

\def\Y{\mathcal{Y}}

\DeclareMathOperator{\ex}{ex}

\def\defined{=}

\newenvironment{proofclaim}[1][Proof of the claim]{\begin{proof}[#1]}{\end{proof}}

\def\COMMENT#1{}
\def\TASK#1{}
\let\TASK=\footnote% COMMENT OUT for clean output
\let\COMMENT=\footnote% COMMENT OUT for clean output

\usepackage{tikz}
\usetikzlibrary{calc}

\allowdisplaybreaks

\begin{document}

\title[Covering and tiling hypergraphs with tight cycles]{Covering and tiling hypergraphs with tight cycles}
\author{Jie Han}
\address{Department of Mathematics, University of Rhode Island, Kingston, RI, USA, 02881}
\email{jie\_han@uri.edu}
\author{Allan Lo}
\thanks{The research leading to these results was partially supported by FAPESP (Proc. 2013/03447-6, 2014/18641-5, 2015/07869-8) (J.~Han) EPSRC, grant no. EP/P002420/1 (A.~Lo) and the Becas Chile scholarship scheme from CONICYT (N.~Sanhueza-Matamala).}
\author{Nicol\'as Sanhueza-Matamala}
\address{School of Mathematics, University of Birmingham, Edgbaston, Birmingham, B15 2TT, UK}
\email{s.a.lo@bham.ac.uk, NIS564@bham.ac.uk}
\date{\today}

\begin{abstract} 
	A $k$-uniform tight cycle $C^k_s$ is a hypergraph on $s > k$ vertices with a cyclic ordering such that every $k$ consecutive vertices under this ordering form an edge.
	The pair $(k,s)$ is admissible if $\gcd(k,s) = 1$ or $k/\gcd(k,s)$ is even.
	We prove that if $s \ge 2k^2$ and $H$ is a $k$-uniform hypergraph with minimum codegree at least $(1/2 + o(1))|V(H)|$, then every vertex is covered by a copy of~$C^k_s$.
	The bound is asymptotically sharp if $(k,s)$ is admissible.
	Our main tool allows us to arbitrarily rearrange the order of which a tight path wraps around a complete $k$-partite $k$-uniform hypergraph, which may be of independent interest.% to other researchers.
	
	For hypergraphs $F$ and $H$, a perfect $F$-tiling in~$H$ is a spanning collection of vertex-disjoint copies of~$F$.
	For $k \ge 3$, there are currently only a handful of known $F$-tiling results when $F$ is $k$-uniform but not $k$-partite.
	If $s \not \equiv 0 \bmod{k}$, then $C^k_s$ is not $k$-partite.
	Here we prove an $F$-tiling result for a family of non $k$-partite $k$-uniform hypergraphs~$F$.
	Namely, for $s \ge 5k^2$, every $k$-uniform hypergraph $H$ with minimum codegree at least $(1/2 + 1/(2s) + o(1))|V(H)|$ has a perfect $C^k_s$-tiling.
	Moreover, the bound is asymptotically sharp if $k$ is even and $(k,s)$ is admissible.
\end{abstract}

\maketitle

\section{Introduction}

Let $\h$ and $F$ be graphs.
An \emph{$F$-tiling} in $H$ is a set of vertex-disjoint copies of~$F$.
An $F$-tiling is \emph{perfect} if it spans the vertex set of~$\h$.
Note that a perfect $F$-tiling is also known as an \emph{$F$-factor} or a \emph{perfect $F$-matching}.
The following question in extremal graph theory has a long and rich story: given $F$ and $n$, what is the maximum $\delta$ such that there exists a graph $H$ on $n$ vertices with minimum degree at least $\delta$ without a perfect $F$-tiling?
We call such $\delta$ the \emph{tiling degree threshold} for $F$ and denote it by $t(n, F)$.
Note that if $n \not\equiv 0 \bmod |V(F)|$ then a perfect $F$-tiling cannot exist, so this case is not interesting.
Hence we will always assume that $n \equiv 0 \bmod |V(F)|$ whenever we discuss~$t(n, F)$.

%When $F$ is a $2$-graph, there is a long and rich story of studying $t(n, F)$.
A first result in the study of tiling thresholds in graphs comes from the celebrated theorem of Dirac~\cite{Dirac1952} on Hamiltonian cycles, which easily shows that $t(n, K_2) = n/2 - 1$.
Corr\'adi and Hajnal~\cite{CorradiHajnal1963} proved that $t(n, K_3) = 2n/3 - 1$,
and Hajnal and Szemer\'edi~\cite{HajnalSzemeredi1970} generalized this result for complete graphs of any size, showing that~$t(n, K_t) = (1 - 1/t)n -1 $.
For a general graph~$F$, K\"uhn and Osthus~\cite{KuehnOsthus2006} determined $t(n, F)$ up to an additive constant depending only on~$F$. This improved previous results due to Alon and Yuster~\cite{AlonYuster1996}, Koml\'os, S\'ark\"ozy and Szemer\'edi~\cite{KomlosSarkoezySzemeredi2001} and Koml\'os~\cite{Komlos2000}.

We study tilings in the setting of $k$-graphs, i.e. hypergraphs where every edge has exactly~$k$ vertices, for some $k \ge 2$.
We focus on tilings using ``tight cycles'', which are $k$-graphs that generalise the usual notion of cycles in graphs.
We also study the related problem of finding $F$-coverings in a hypergraph $\h$, that is, finding copies of $F$, not necessarily vertex-disjoint, which together cover every vertex of $\h$.
After choosing a notion of ``minimum degree'' for $k$-uniform hypergraphs, both tilings and coverings give rise to corresponding questions in extremal hypergraph theory, which generalise the ``tiling thresholds'' in graphs to the setting of hypergraphs.
In what follows, we describe precisely all of the problems under consideration.

\subsection{Tiling thresholds}

A \emph{hypergraph} $\h = (V(\h),E(\h))$ consists of a vertex set $V(\h)$ and an edge set $E(\h)$, where each edge $e \in E(\h)$ is a subset of~$V(\h)$.
We will simply write $V$ and $E$ for $V(H)$ and $E(H)$, respectively, if it is clear from the context.
Given a set $V$ and a positive integer $k$, $\binom{V}{k}$ denotes the set of subsets of $V$ with size exactly~$k$.
We say that $\h$ is a \emph{k-uniform hypergraph} or \emph{$k$-graph}, for short, if $E \subseteq \binom{V}{k}$. %, and we abbreviate `$k$-uniform hypergraphs' to \emph{$k$-graphs}.
Note that $2$-graphs are usually known simply as \emph{graphs}.

Given a hypergraph $\h$ and a set $S \subseteq V$, let the \emph{neighbourhood $N_\h(S)$ of $S$} be the set $\{ T \subseteq V \setminus S : T \cup S \in E \}$ and let $\deg_{\h}(S) = |N_\h (S)|$ denote the number of edges of $\h$ containing~$S$.
If $w \in V$, then we also write $N_\h(w)$ for~$N_\h( \{w\} )$.
We will omit the subscript if $\h$ is clear from the context.
We denote by $\delta_i(\h)$ the \emph{minimum $i$-degree of $\h$}, that is, the minimum of $\deg_{\h}(S)$ over all $i$-element sets $S \in \binom{V}{i}$.
Note that $\delta_0(\h)$ is equal to the number of edges of~$\h$.
Given a $k$-graph $\h$, $\delta_{k-1}(\h)$ and $\delta_1(\h)$ are referred to as the \emph{minimum codegree} and the \emph{minimum vertex degree} of $\h$, respectively.

For $k$-graphs~$\h$ and~$F$, an \emph{$F$-tiling} in $H$ is a set of vertex-disjoint copies of~$F$;
and an $F$-tiling is \emph{perfect} if it spans the vertex set of~$\h$.
For a $k$-graph $F$, define the \emph{codegree tiling threshold} $t(n, F)$ to be the maximum of $\delta_{k-1}(\h)$ over all $k$-graphs~$\h$ on $n$ vertices without a perfect $F$-tiling.
We implicitly assume $n \equiv 0 \bmod |V(F)|$ whenever we discuss~$t(n, F)$.

We describe known results on tiling thresholds for $k$-graphs, when $k \ge 3$.
Let $K^k_t$ denote the complete $k$-graph on $t$ vertices.
For $k \ge 3$, K\"uhn and Osthus~\cite{KuehnOsthus2006} determined $t(n, K^k_k)$ asymptotically;
the exact value was determined by R\"odl, Ruci\'nski and Szemer\'edi~\cite{RoedlRucinskiSzemeredi2009} for sufficiently large~$n$.
Lo and Markstr\"om~\cite{LoMarkstroem2015} determined $t(n, K^3_4)$ asymptotically, and independently, Keevash and Mycroft~\cite{KeevashMycroft2014} determined $t(n, K^3_4)$ exactly for sufficiently large~$n$.

We say that a $k$-graph $H$ is \emph{$t$-partite} (or that $\h$ is a \emph{$(k,t)$-graph}, for short) if $V$ has a partition $\{ V_1, \dotsc, V_t\}$ such that $|e \cap V_i| \leq 1$ for all edges $e \in E$ and all $1 \leq i \leq t$.
A $(k,t)$-graph $H$ is \emph{complete} if $E$ consists of all $k$-sets~$e$ such that $|e \cap V_i| \leq 1$, for all $1 \leq i \leq t$.
Recently, Mycroft~\cite{Mycroft2016} determined the asymptotic value of $t(n, K)$ for all complete $(k,k)$-graphs~$K$.
However, much less is known for non-$k$-partite $k$-graphs. 
For more results on tiling thresholds for $k$-graphs, see the survey of Zhao~\cite{Zhao2016}.

\subsection{Covering thresholds}

Given a $k$-graph $F$, an \emph{$F$-covering in $H$} is a spanning set of copies of $F$.
%a $k$-graph $\h$ has an \emph{$F$-covering} if for all vertices $v \in V(\h)$, $\h$ contains a copy of $F$ containing~$v$.
Similarly, define the \emph{codegree covering threshold} $c(n, F)$ of $F$ to be the maximum of $\delta_{k-1}(\h)$ over all $k$-graphs $\h$ on $n$ vertices not containing an $F$-covering.
%As before, this value also depends on $k$, but this will always be clear from the context.
%We write $c_i(n, F)$ if $k$ is clear from the context.

Trivially, a perfect $F$-tiling is an $F$-covering, and an $F$-covering has a copy of~$F$.
Thus, \[ \ex_{k-1}(n, F) \leq c(n, F) \leq t(n, F), \]
where $\ex_{k-1}(n, F)$ is \emph{codegree Tur\'an threshold}, that is, the maximum of $\delta_{k-1}(\h)$ over all $F$-free $k$-graphs~$H$ on $n$ vertices.
In this sense, the covering problem is an intermediate problem between the Tur\'an and the tiling problems.

As for results on covering thresholds, for any non-empty ($2$-)graph $F$, we have $c(n, F) = \left( \frac{\chi(F) - 2}{\chi(F) - 1} + o(1) \right)n$, see~\cite{HanZangZhao2015},
where $\chi(F)$ is the chromatic number of~$F$.
Han, Zang and Zhao~\cite{HanZangZhao2015} studied the vertex-degree variant of the covering problem, for complete $(3,3)$-graphs~$K$.
Falgas-Ravry and Zhao~\cite{Falgas-RavryZhao2016} studied $c(n, F)$ when~$F$ is $K^3_4$, $K^3_4$ with one edge removed, $K^3_5$ with one edge removed and other $3$-graphs.
	
\subsection{Cycles in hypergraphs}

Given $1 \leq \ell < k$, we say that a $k$-graph on more than $k$ vertices is an \emph{$\ell$-cycle} if every vertex lies in some edge and there is a cyclic ordering of the vertices such that under this ordering, every edge consists of $k$ consecutive vertices and two consecutive edges intersect in exactly $\ell$ vertices.
Note that an $\ell$-cycle on $s$ vertices can exist only if $k - \ell$ divides~$s$.
If $\ell = 1$ we call the cycle \emph{loose}, if $\ell = k-1$ we call the cycle \emph{tight}.
We write $C^k_s$ for the $k$-uniform tight cycle on $s$ vertices.%, or $C_s$ if $k$ is clear from the context.

When $k = 2$, $\ell$-cycles reduce to the usual notion of cycles in graphs.
Corr\'adi and Hajnal~\cite{CorradiHajnal1963} determined $t(n, C^2_3)$ and Wang~\cites{Wang2010, Wang2012} determined $t(n, C^2_4)$ and $t(n, C^2_5)$.
In fact, El-Zahar~\cite{ElZahar1984} gave the following conjecture on cycle tilings.
\begin{conjecture}[El-Zahar \cite{ElZahar1984}]
	Let $G$ be a graph on $n$ vertices and let $n_1, \dotsc, n_r \ge 3$ be integers such that $n_1 + \dotsb + n_r = n$.
	If $\delta(G) \ge \sum_{i=1}^r \lceil n_i / 2 \rceil$, then $G$ contains $r$ vertex-disjoint cycles of lengths $n_1, \dotsc, n_r$ respectively.
\end{conjecture}
The bound on the minimum degree, if true, would be best possible.
In particular, the conjecture would imply that $t(n, C^2_s) =  \lceil s / 2 \rceil n / s - 1$.
The conjecture was verified for $r = 2$ by El-Zahar and a proof (for large $n$) was announced by Abbasi \cite{Abbasi1998} as well as by Abbasi, Khan, S\'ark\"ozy and Szemer\'edi (see \cite{Szemeredi2013}).

Given integers $\ell, k$ such that $1 \leq \ell \leq (k-1)/2$, it is easy to see that a $k$-uniform $\ell$-cycle on $s$ vertices~$C$ satisfies $c(n, C) \leq s+1$ (by constructing $C$ greedily).
If $s \equiv 0 \bmod k$, then the tight cycle~$C^k_s$ is $k$-partite.
For all $t \ge 1$, let $K^k(t)$ denote the complete $(k,k)$-graph whose vertex classes each have size~$t$.
Note that $C^k_s$ is a spanning subgraph of~$K^k(s/k)$.
Erd\H{o}s~\cite{Erdoes1964} proved the following result, which implies an upper bound on the Tur\'an number of~$C^k_s$. 

\begin{theorem}[Erd\H{o}s~\cite{Erdoes1964}] \label{theorem:kovarisosturanhypergraphs}
	For all $k \ge 2$ and $s > 1$, there exists $n_0 = n_0(k,s)$ such that $\ex( n, K^{k}(s) )  < n^{k - 1/s^{k-1}}$ for all $n \ge n_0$.
\end{theorem}

Our first result is a sublinear upper bound for $c(n, C^k_s)$ when $s \equiv 0 \bmod k$.

\begin{proposition} \label{proposition:coveringthresholdmodk}
	For all $2 \leq k \leq s$ with $s \equiv 0 \bmod k$, there exist $n_0(k, s)$ and $c = c(k, s)$ such that $c(n, C_s^k) \leq cn^{1 - 1/s^{k-1}}$ for all $n \ge n_0$.
\end{proposition}

There are some previously known results for tiling problems regarding $\ell$-cycles.
Whenever $C$ is a $3$-uniform loose cycle, $t(n, C)$ was determined exactly by Czygrinow~\cite{Czygrinow2016}.
For general loose cycles $C$ in $k$-graphs, $t(n, C)$ was determined asymptotically by Mycroft~\cite{Mycroft2016} and exactly by Gao, Han and Zhao~\cite{GaoHanZhao2016}.
For tight cycles $C^k_s$ with $s \equiv 0 \bmod k$, Mycroft~\cite{Mycroft2016} proved that $t(n, C^k_s) = (1/2 + o(1))n$.
Notice that all mentioned cycle tiling results correspond to cases where the cycles are $k$-partite (since $k$-uniform loose cycles are $k$-partite for $k \ge 3$).

We now focus on the covering and tiling problems for the tight cycle $C^k_s$, for all integers $k,s$ which do not necessarily make $C^k_s$ a $(k,k)$-graph.
We show that a minimum codegree of $(1/2+o(1))n$ suffices to find a $C^k_s$-covering.

\begin{theorem} \label{theorem:coveringthresholdnotmodk}
	Let $k, s \in \mathbb{N}$ with $k \ge 3$ and $s \ge 2k^2$.
	For all $\gamma > 0$, there exists $n_0 = n_0(k, s, \gamma)$ such that for all $n \ge n_0$, $ c(n, C^k_s) \leq (1/2 + \gamma)n$.
\end{theorem}

Moreover, this result is asymptotically tight if $k$ and $s$ satisfy the following divisibility conditions.
Let $2 \leq k < s$ and let $d = \gcd(k,s)$. We say that the pair $(k,s)$ is \emph{admissible} if $d = 1$ or $k/d$ is even.
Note that an admissible pair~$(k,s)$ satisfies $s \not\equiv 0 \bmod k$.

\begin{proposition} \label{proposition:lowerboundscovering}
	Let $3 \leq k < s$ be such that $(k,s)$ is admissible.
	Then $c(n, C^k_s) \ge \lfloor n / 2 \rfloor - k + 1$.
	Moreover, if $k$ is even, then $\ex_{k-1}(n, C^k_s) \ge \lfloor n / 2 \rfloor - k + 1$.
\end{proposition}

Notice that if $(k,s)$ is admissible, $k \ge 3$ is even and $s \ge 2k^2$, then Theorem~\ref{theorem:coveringthresholdnotmodk} and Proposition~\ref{proposition:lowerboundscovering} imply that $ \ex_{k-1}(n, C^k_s) = (1/2 + o(1))n$.

We also study the tiling problem corresponding to $C^k_s$.
We give some lower bounds on~$t(n, C^k_s)$.
Notice that the bound is significantly higher if $(k,s)$ is admissible.

\begin{proposition} \label{proposition:lowerboundstiling}
	Let $2 \leq k < s \leq n$ with $n$ divisible by $s$.
	Then $t(n, C^k_s) \ge \lfloor n/2 \rfloor - k$.
	Moreover, if $(k,s)$ is admissible, then
	\[ t(n, C_s^k) \ge  \begin{dcases}
	\left\lfloor  \left(\frac{1}{2} + \frac{1}{2s} \right) n \right\rfloor - k & \text{if $k$ is even,} \\ \left\lfloor \left( \frac{1}{2} + \frac{k}{4s(k-1) + 2k} \right)  n \right\rfloor- k & \text{if $k$ is odd.} \end{dcases} \]	
\end{proposition}

On the other hand, recall that the case $s \equiv 0 \bmod k$ was solved asymptotically by Mycroft~\cite{Mycroft2016}, thus we study the complementary case.
We prove an upper bound on $t(n, C^k_s)$ which is valid whenever $s \not\equiv 0 \bmod k$ and $s \ge 5k^2$.
Note that the bound is asymptotically sharp if $k$ is even and $(k,s)$ is admissible.

\begin{theorem} \label{theorem:tilingthreshold}
	Let $3 \leq k < s$ be such that $s \ge 5k^2$ and $s \not\equiv 0 \bmod k$.		
	Then, for all $\gamma > 0$, there exists $n_0 = n_0(k, s, \gamma)$ such that for all $n \ge n_0$ with $n \equiv 0 \bmod s$, \[ t(n, C_s^k) \leq \left( \frac{1}{2} + \frac{1}{2s} + \gamma \right)n. \]
\end{theorem}

\subsection{Organisation of the paper}

In Section~\ref{section:notation} we set up basic notation and give sketches of the proofs of our main results, Theorems~\ref{theorem:coveringthresholdnotmodk} and~\ref{theorem:tilingthreshold}.

In Section~\ref{section:extremalgraphs} we give constructions which imply lower bounds for the Tur\'an numbers and covering and tiling thresholds of tight cycles, thus proving Propositions~\ref{proposition:lowerboundscovering} and~\ref{proposition:lowerboundstiling}.

In the next two sections we study the covering problem.
In Section~\ref{section:gadgets} we describe a family of gadgets which will be useful during the proofs of Proposition~\ref{proposition:coveringthresholdmodk} and Theorem~\ref{theorem:coveringthresholdnotmodk}.
Those proofs are done in Section~\ref{section:upperboundcoveringthresholds}.

Sections \ref{section:absorption}--\ref{section:tilings} are dedicated to investigating the tiling problem.
Our aim is the proof of Theorem~\ref{theorem:tilingthreshold}, i.e. bounding $t(n, C^k_s)$ from above.
In Section~\ref{section:absorption}, we review the absorption technique for tilings, which we use in Section~\ref{section:tilingthresholds} to prove Theorem~\ref{theorem:tilingthreshold} under the assumption that we can find an almost perfect $C^k_s$-tiling (Lemma~\ref{lemma:almostcstiling}).
We prove Lemma~\ref{lemma:almostcstiling} in the next two sections: in Section~\ref{section:regularity} we review tools of hypergraph regularity and in Section~\ref{section:tilings} we introduce various auxiliary tilings that we use to finish the proof.

We conclude with some remarks and open problems in Section~\ref{section:remarks}.

\section{Notation and sketchs of proofs} \label{section:notation}

For a hypergraph $\h$ and $S \subseteq V$, we denote $\h[S]$ to be the subgraph of $\h$ induced on $S$, that is, $V(\h[S]) = S$ and $E(\h[S]) = \{ e \in E : e \subseteq S \}$.
Let $\h \setminus S = H[V \setminus S]$.
For hypergraphs $\h$ and $G$, let $\h - G$ be the subgraph of~$\h$ obtained by removing all edges in $E(\h) \cap E(G)$.

Given $a, b, c$ reals with $c > 0$, by $a = b \pm c$ we mean that $b - c \leq a \leq b + c$.
We write $x \ll y$ to mean that for all $y \in (0,1]$ there exists an $x_0 \in (0,1)$ such that for all $x \leq x_0$ the subsequent statement holds.
Hierarchies with more constants are defined in a similar way and are to be read from the right to the left.
We will always assume that the constants in our hierarchies are reals in~$(0,1]$.
Moreover, if $1/x$ appears in a hierarchy, this implicitly means that $x$ is a natural number.

For all $k$-graphs $\h$ and all $x \in V$, define the \emph{link $(k-1)$-graph $\h(x)$ of $x$ in $\h$} to be the $(k-1)$-graph with $V(\h(x)) = V \setminus \{ x\}$ and~$E(\h(x)) = N_{\h}(x)$.
Given integers $a_1, \dotsc, a_t \ge 1$, let $K^{k}(a_1, \dotsc, a_t)$ denote the complete $(k,t)$-graph with vertex partition $V_1, \dotsc, V_t$ such that $|V_i| = a_i$ for all $1 \leq i \leq t$.

For a family $\mathcal{F}$ of $k$-graphs, an \emph{$\mathcal{F}$-tiling} is a set of vertex-disjoint copies of (not necessarily identical) members of $\mathcal{F}$.

For a sequence of distinct vertices $v_1, \dotsc, v_s$ in a $k$-graph $H$, we say $P=v_1 \dotsb v_s$ is a \emph{tight path} if all $k$ consecutive vertices form an edge.
Note that all tight paths have an associated ordering of vertices.
Hence, $v_1 \dotsb v_s$ and $v_s \dotsb v_1$ are assumed to be different tight paths, even if the corresponding subgraphs they define are the same.

Suppose that $P_1 = v_1 \dotsb v_s$ and $P_2 = w_1 \dotsb w_{s'}$ are two vertex-disjoint tight paths in a $k$-graph $H$.
If it happens that $v_1 \dotsb v_s w_1 \dotsb w_{s'}$ is also a tight path in $H$, then we will denote it by~$P_1 P_2$.
We sometimes refer to $P_1 P_2$ as the \emph{concatenation of $P_1$ and $P_2$}.
Note that $P_1P_2$ has more edges than $P_1 \cup P_2$. 
We naturally extend this definition (whenever it makes sense) to the concatenation of a sequence of paths $P_1$, $\dotsc$, $P_r$, and we denote the resulting path by $P_1 \dotsb P_r$.
For two tight paths $P_1$ and $P_2$, we say that $P_2$ \emph{extends} $P_1$, if $P_2 = P_1 P'$ for some tight path~$P'$ (where we may have $|V(P')| < k$, that is, $P'$ contains no edge).
Also, we may define a tight cycle~$C$ by writing $C = v_1 \dotsb v_s$, whenever $v_{i} \dotsb v_s v_1 \dotsb v_{i-1}$ is a tight path for all $1 \leq i \leq s$.

For all $k \in \mathbb{N}$,  let $[k] = \{1, \dotsc, k \}$. Let $S_k$ be the symmetric group of all permutations of the set $[k]$, with the composition of functions as the group operation. Let $\operatorname{id} \in S_k$ be the \emph{identity function} that fixes all elements in~$[k]$. Given distinct $i_1, \dotsc, i_r \in [k]$, the \emph{cyclic permutation} $(i_1 i_2 \dotsb i_r) \in S_k$ is the permutation that maps $i_{j}$ to $i_{j+1}$ for all $1 \leq j < r$ and $i_r$ to $i_1$, and fixes all the other elements; we say that such a cyclic permutation has \emph{length}~$r$. All permutations $\sigma \in S_k$ can be written as a composition of cyclic permutations $\sigma_1 \dotsb \sigma_t$ such that these cyclic permutations are \emph{disjoint}, meaning that there are no common elements between all pairs of these different cyclic permutations.

Let $H$ be a $k$-graph, $V_1, \dotsc, V_k$ be disjoint vertex sets of $V$ and let $\sigma \in S_k$.
We say that a tight path $P = v_1 \dotsb v_\ell$ in $H$ has \emph{end-type~$\sigma$ with respect to $V_1, \dotsc, V_k$} if for all $2 \leq i \leq k$, $v_{\ell - k + i} \in V_{\sigma(i)}$. Similarly, we say $P$ has \emph{start-type~$\sigma$ with respect to $V_1, \dotsc, V_k$} if $v_{i} \in V_{\sigma(i)}$ for all $1 \leq i \leq k-1$. If $H$ and $V_1, \dotsc, V_k$ are clear from the context, we simply say that $P$ has \emph{end-type $\sigma$} and \emph{start-type $\sigma$}, respectively.
Note that one could define start-type and end-type in terms of $(k-1)$-tuples in~$[k]$ instead. 
However, for our purposes, it is more convenient to define it in terms of permutations of~$[k]$.

\subsection{Sketches of proofs of Theorems~\ref{theorem:coveringthresholdnotmodk} and~\ref{theorem:tilingthreshold}}
We now sketch the proof of Theorem~\ref{theorem:coveringthresholdnotmodk}.
Let $\h$ be a $k$-graph on $n$ vertices with $\delta_{k-1}(\h) \ge (1/2 + \gamma)n$.
Consider any vertex $x \in V(\h)$.
We can show that, for some appropriate value of $t$, $x$ is contained in some copy $K$ of $K^k_k(t)$ with vertex classes $V_1, \dotsc, V_k$. % (this is done in Proposition~\ref{proposition:kkscovering}).
Suppose that $s \equiv r \not\equiv 0 \bmod k$ with $1 \leq r < k$.
Suppose $P = v_1 \dotsb v_k$ is a tight path in $K$ such that $v_i \in V_i$ for all $1 \le i \le k$ and $v_1 = x$.
By wrapping around~$K$, we may find a tight path $P_2 = v_1 \dotsb v_{\ell}$ which extends $P_1$, but if we only use vertices and edges of~$K$, then we have $v_j \in V_{\ell}$ where $j \equiv \ell \bmod k$, for all $j \in [\ell]$.
To break this pattern, we will use some gadgets (see Section~\ref{section:gadgets} for a formal definition).
Roughly speaking, a gadget is a $k$-graph on $V(K)$ and some extra vertices of $H$.
Using these gadgets we can extend $P$ to a tight path $P'$ with end-type $\sigma$, for an arbitrary $\sigma \in S_k$ (see Lemma~\ref{lemma:gadgetsigma}).
Having done that (and choosing $\sigma$ appropriately), then it is easy to extend $P'$ into a copy of $C_s^k$ by wrapping around $V_1, \dotsc, V_k$.
%The gadgets we introduce are more general, allowing us to have $v_{(a-1)k+i} \in V_{\sigma (i)}$ for any permutation~$\sigma$ on $[k]$.% (see Section~\ref{section:gadgets}).

The proof of Theorem~\ref{theorem:tilingthreshold} uses the absorbing method, introduced by R\"odl, Ruci\'nski and Szemer\'edi \cite{RoedlRucinskiSzemeredi2009}.
We first find a small vertex set $U \subseteq V(H)$ such that $H[U \cup W]$ has a perfect $C_s^k$-tiling for all small sets $W$ with $|U| + |W| \equiv 0 \bmod s$. % (see Lemma~\ref{lemma:theabsorbinglemma}).
Thus the problem of finding a perfect $C_s^k$-tiling is reduced to finding a $C_s^k$-tiling in $H \setminus U$ covering almost all of the remaining vertices.
However, we do not find such $C_s^k$-tiling directly.
First we show that there exists a $k$-graph~$F_s$ on $s$ vertices containing a~$C_s^k$ which has a particularly useful structure: it is obtained from a complete $(k,k)$-graph by adding a few extra vertices.
So finding an almost perfect $F_s$-tiling suffices.
Instead, we show that there exists an $\{ F_s, E_s \}$-tiling $\T$ for some suitable $k$-graph~$E_s$, subject to the minimisation of some objective function~$\phi(\T)$.
We do so by considering its fractional relaxation, which we call a weighted fractional $\{ F^\ast_s, K^\ast_s \}$-tiling (see Section~\ref{subsection:weightedfractionaltilings}).
Further, we use the hypergraph regularity lemma in the form of `regular slice lemma' of Allen, B\"ottcher, Cooley and Mycroft~\cite{AllenBottcherCooleyMycroft2017}.

\section{Lower bounds} \label{section:extremalgraphs}

In this section, we construct $k$-graphs which give lower bounds for the codegree Tur\'an numbers and covering and tiling thresholds for tight cycles.
These constructions will imply Proposition~\ref{proposition:lowerboundscovering} and Proposition~\ref{proposition:lowerboundstiling}.
%We remark that the bounds that we obtain are not best possible for all values of $n$, $s$ and~$k$.
%Some improvements can be made by considering the same examples but being more careful with the calculations of the minimum degree of each $k$-graph.
%Most of the times, this can be done by separating the analysis in cases depending on divisibility conditions of $n$, $s$ and~$k$ (compare, for instance, with the extremal examples for perfect matchings in $k$-graphs~\cite[Construction 1.1]{Zhao2016}).
%We did not pursue this direction to simplify the presentation, since our main interest is in the asymptotics of each threshold function.
We remark that the bounds obtained here can be improved by an additive constant via careful calculations and case distinctions, which we omit for the sake of giving a clear presentation.

Let $A$ and $B$ be disjoint vertex sets.
Define $\h^k_0 = \h^k_0(A,B)$ to be the $k$-graph on $A \cup B$ such that the edges of $\h^k_0$ are exactly the $k$-sets $e$ of vertices that satisfy $|e \cap B| \equiv 1 \bmod 2$.
Note that $\delta_{k-1}(\h^k_0) \ge \min\{ |A|,|B| \} - k + 1$.
%\begin{align*}
%\begin{cases}
%\min\{ |A|, |B| - 1 \} - k + 2 & \text{if $k$ odd,} \\
%\min\{ |A|, |B| \} - k + 2 & \text{if $k$ even.} \\
%\end{cases}
%\end{align*}

\begin{proposition} \label{proposition:cyclicgcd}
	Let $3 \leq k \leq s$ and $d = \gcd(k,s)$.
	Let $A$ and $B$ be disjoint vertex sets.
	Suppose that $H^k_0(A,B)$ contains a tight cycle $C^k_s$ on $s$ vertices with $V(C^k_s) \cap A \neq \emptyset$.
	Then $|V(C^k_s) \cap A| \equiv 0 \mod s/d$ and $(k,s)$ is not an admissible pair.
\end{proposition}

\begin{proof}
	Let $C_s^k = v_1 \dotsb v_s$.
	For all $1 \leq i \leq s$, let $\phi_i \in \{ A, B \}$ be such that $v_i \in \phi_i$ and let $\phi_{s + i} = \phi_i$.
	If two edges $e$ and $e'$ in $E(H^k_0(A,B))$ satisfy $|e \cap e'| = k-1$, then $|e \cap A| = |e' \cap A|$ by construction.
	Thus $\phi_{i + k} = \phi_i$ for all $1 \leq i \leq s$.
	Therefore, $\phi_{i+d} = \phi_i$ for all $1 \leq i \leq s$.
	Hence, $|V(C_s^k) \cap A| \equiv 0 \bmod s/d$.
	
	Let $r = |\{ v_1, \dotsc, v_k \} \cap A| = |\{ i: 1 \leq i \leq k, \phi_i = A \}|$.
	Note that $r > 0$ and $r \in \{ k/d, 2k/d, \dotsc, k \}$.
	Since $\{ v_1, \dotsc, v_k \}$ is an edge in $H^k_0(A,B)$, it follows that $k-r \equiv 1 \bmod 2$ and so, $r \not\equiv k \bmod 2$.
	This implies $d \ge 2$ and $k/d$ is odd, i.e. $(k,s)$ is not an admissible pair.
\end{proof}

Now we use Proposition~\ref{proposition:cyclicgcd} to prove Propositions~\ref{proposition:lowerboundscovering} and~\ref{proposition:lowerboundstiling}.

\begin{proof}[Proof of Proposition~\ref{proposition:lowerboundscovering}]
	Let $A$ and $B$ be disjoint vertex sets of sizes $|A| = \lfloor n/2 \rfloor$ and $|B| = \lceil n/2 \rceil$.
	Consider the $k$-graph $\h_0 = \h^k_0(A,B)$.
	By Proposition~\ref{proposition:cyclicgcd}, no vertex of $A$ can be covered with a copy of~$C^k_s$.
	Then $c(n, C^k_s) \ge \delta_{k-1}(H_0) \ge \lfloor n / 2 \rfloor - k + 1$.
	
	Moreover, if $k$ is even, then $H^k_0(A,B) = H^k_0(B,A)$.
	So no vertex of $B$ can be covered by a copy of~$C^k_s$.
	Hence $\h_0$ is $C^k_s$-free.
	Therefore, $\ex_{k-1}(n, C^k_s) \ge \delta_{k-1}(H_0) \ge \lfloor n / 2 \rfloor - k + 1$.
\end{proof}

\begin{proof}[Proof of Proposition~\ref{proposition:lowerboundstiling}]
	To see the first part of the statement, let $d := \gcd(k,s)$ and $s' := s/d$. % and $k' = k/d$.
	Note that $d \leq k < s$, thus $s' > 1$.
	Let $A$ and $B$ be disjoint vertex sets chosen such that $|A| + |B| = n$, $\left| |A| - |B| \right| \leq 2$ and $|A| \not\equiv 0 \bmod s'$.
	Consider the $k$-graph $H_0 = H^k_0(A,B)$ and note that $\delta_{k-1}(H_0) \ge \min\{ |A|,|B| \} - k + 1 \ge \lfloor n/2 \rfloor - k$.
	Proposition~\ref{proposition:cyclicgcd} implies that all copies~$C$ of $C^k_s$ in $\h_0$ satisfy $|V(C) \cap A| \equiv 0 \bmod s'$.
%	It follows that $\delta_{k-1}(\h_0) \ge \lfloor n/2 \rfloor - k$.
%	Since all copies of $C_s^k$ in $\h_0$ satisfy $|V(C_s^k) \cap A| \equiv 0 \bmod s'$,
	Since $|A| \not\equiv 0 \bmod s'$, it is impossible to cover all vertices in $A$ with vertex-disjoint copies of~$C_s^k$.
	This proves that $t(n, C^k_s) \ge \delta_{k-1}(\h_0) \ge \lfloor n/2 \rfloor - k$ as desired.
	
	Now suppose that $(k,s)$ is an admissible pair.
	Let $\h$ be the $k$-graph on $n$ vertices with a vertex partition $\{A, B, T\}$ with $|A| = \lceil (n - |T|) / 2 \rceil$ and $|B| = \lfloor (n - |T|)/2 \rfloor$, where $|T|$ will be specified later.
	The edge set of $\h$ consists of all $k$-sets $e$ such that $|e \cap B| \equiv 1 \bmod 2$ or $e \cap T \neq \emptyset$.
	Note that $\delta_{k-1}(\h) \ge \min\{|A|, |B| \} + |T| - (k-1) \ge \lfloor (n + |T|)/2 \rfloor - k + 1$. We separate the analysis into two cases depending on the parity of~$k$.
	\medskip
	
	\noindent\textbf{Case 1: $k$ even.}
	Since $\h[A \cup B] = \h^k_0(A, B) = \h^k_0(B,A)$, by Proposition~\ref{proposition:cyclicgcd}, $\h[A \cup B]$ is $C^k_s$-free.
	Thus, all copies of $C^k_s$ in $\h$ must intersect $T$ in at least one vertex.
	Hence, all $C^k_s$-tilings have at most $|T|$ vertex-disjoint copies of~$C^k_s$.
	Taking $|T| = n/s - 1$ assures that $\h$ does not contain a perfect $C^k_s$-tiling.
	This implies that $t(n, C^k_s) \ge \left\lfloor \left(1/2 + 1/(2s) \right) n \right\rfloor - k$.
	\medskip
	
	\noindent\textbf{Case 2: $k$ odd.}
	Since $\h[A \cup B] = \h^k_0(A,B)$, by Proposition~\ref{proposition:cyclicgcd} no vertex in $A$ can be covered by a copy of~$C^k_s$.
	Hence, all copies of $C^k_s$ in $\h$ with non-empty intersection with~$A$ must also have non-empty intersection with~$T$.
	Moreover, all edges in~$\h$ intersect~$A$ in at most $k-1$ vertices, so all copies of~$C^k_s$ in~$H$ intersect~$A$ in at most $s(k-1)/k$ vertices.
	Thus a perfect $C^k_s$-tiling would contain at most $|T|$ and at least $k|A|/(s(k-1))$ cycles intersecting~$A$.
	Let $|T| = \lceil nk / (2s(k-1) + k) \rceil - 1$.
	Since $|T| < nk/(2s(k-1)+k)$ and $|A| \ge (n-|T|)/2$, \begin{align*}
	\frac{k|A|}{s(k-1)} \ge \frac{k(n-|T|)}{2s(k-1)} > \frac{nk}{2s(k-1)}\left( 1 - \frac{k}{2s(k-1) + k} \right) > |T|,
	\end{align*}
	%Then $|T| < k |A|/ (s(k-1))$.
	and thus a perfect $C^k_s$-tiling in $\h$ cannot exist.
	This implies
	\begin{align*}
	t(n, C^k_s) \ge \delta_{k-1}(\h) \ge \left\lfloor \frac{n+|T|}{2} \right\rfloor - k + 1 \ge \left\lfloor \left( \frac{1}{2} + \frac{k}{4s(k-1) + 2k} \right)  n \right\rfloor - k,
	\end{align*}
	as desired.
\end{proof}

\section{$G$-gadgets} \label{section:gadgets}

Throughout this section, let $\tau \defined (123 \dotsb k) \in S_k$.
Let $H$ be a $k$-graph, and let $K$ be a complete $(k,k)$-graph in $H$ with its natural vertex partition $\{ V_1, \dotsc, V_k\}$. %be vertex-disjoint sets of $V$ which form the natural vertex partition of a copy~$K$ of a complete $(k,k)$-graph in~$H$.
Knowing the end-types and start-types of paths with respect to $V_1, \dotsc, V_k$ will help us to concatenate them and form longer paths which contains them both.
For instance, if $P_1$ and $P_2$ are vertex-disjoint tight paths, $P_1$ has end-type $\pi$ and $P_2$ has start-type $\pi$, then we can concatenate the paths and obtain $P_1 P_2$.

Let $P$ be a tight path in $H$ with end-type $\pi \in S_k$.
For $x \in V_{\pi(1)} \setminus V(P)$, $Px$ is a tight path of~$\h$ with end-type~$\pi \tau$.
We call such an extension a \emph{simple extension of~$P$}.
By repeatedly applying $r$ simple extensions (which is possible as long as there are available vertices), we may obtain an extension $P x_1 \dotsb x_r$ of $P$ with end-type $\pi \tau^r$, using $r$ extra vertices and edges in~$K$.

In the same spirit, observe that if $P_1$ has end-type $\pi$ and $P_2$ has start-type $\pi \tau$, then the sequence of ordered clusters corresponding to the last $k-1$ vertices of $P_1$ coincides with the corresponding sequence of the first $k-1$ vertices of $P_2$.
Thus, by using one extra vertex $x \in V_{\pi(1)} \setminus (V(P_1) \cup V(P_2))$ and setting $P_1 x P_2$, we can join these paths.

If $P$ is a path with end-type~$\pi$, we would like to find a path $P'$ that extends $P$ such that $|V(P')| \equiv |V(P)| \bmod k$ and $P'$ has end-type $\sigma$, for arbitrary $\sigma \in S_k$.
The goal of this section is to define and study `$G$-gadgets', a tool which will allow us to do precisely that.

Let $G$ be a $2$-graph on~$[k]$ and $S \subseteq V(\h)$.
%Let $\h$ be a $k$-graph and let $K$ be a complete $(k,k)$-graph in $\h$ with vertex partition $V_1, \dotsc, V_k$.
We say $W_G \subseteq V(\h)$ is a \emph{$G$-gadget for $K$ avoiding $S$} if there exists a family of pairwise-disjoint sets $\{ W_{ij} : ij \in E(G) \}$ such that $W_G = \bigcup_{ij \in E(G)} W_{ij}$, and for all $ij \in E(G)$, 
\begin{enumerate}[label = {\rm (W\arabic*)}]
	\item $|W_{ij}| = 2k - 1$,
	\item $|W_{ij} \setminus V(K)| = 1$, $W_{ij} \cap S = \emptyset$ and, for all $1 \le i' \le k$,
	\begin{align*}
	|W_{ij} \cap V_{i'}| = 
	\begin{cases}
	1	& \text{if $ i' \in \{i,j\}$,}\\
	2	& \text{otherwise,}
	\end{cases}
	\end{align*}
%	\item $W_{ij}$ are pairwise disjoint, and
	\item \label{item:propertyW3} for all $\sigma \in S_k$ with $\sigma(1) \in \{i, j\}$, $\h[W_{ij}]$ contains a spanning tight path with start-type $\sigma \tau$ and end-type $(ij) \sigma$.
\end{enumerate}
If $K$ is clear from the context, we will just say ``a $G$-gadget avoiding $S$''.
For all edges $ij \in E(G)$, we write $w_{ij}$ for the unique vertex in $W_{ij} \setminus V(K)$.

We emphasize that~\ref{item:propertyW3} is the key property that allows us to obtain an extension of a path at the same time we perform a change in the end-type.
In words,~\ref{item:propertyW3} says that given any $k-1$ ordered clusters that miss~$V_i$, there exists a tight path with vertex set~$W_{ij}$, which start with the same ordered $k-1$ clusters and ends with the same ordered $k-1$ clusters but with~$V_j$ replaced by~$V_i$. 
In other words, $W_{ij}$ allows us to ``switch'' the type of a path by replacing~$i$ by $j$.
See Figure~\ref{figure:ggadget} for an example.

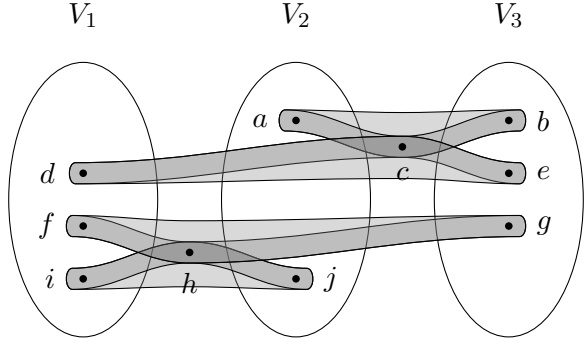
\begin{figure}
	\begin{tikzpicture}
	
	\def\x{0.70} % vertical distance
	\def\y{2.8} %horizontal distance
	
	\node at (0,-1*\x) {$V_1$};
	\node at (1*\y,-1*\x) {$V_2$};
	\node at (2*\y,-1*\x) {$V_3$};
	
	\draw (0*\y,-4.5*\x) ellipse (0.35*\y cm and 2.6*\x cm);
	\draw (1*\y,-4.5*\x) ellipse (0.35*\y cm and 2.6*\x cm);
	\draw (2*\y,-4.5*\x) ellipse (0.35*\y cm and 2.6*\x cm);
	
	\filldraw[fill=gray, fill opacity=0.3]
	($(0.95*\y, -3.2*\x)$)
	to[out=0,in=180] ($(1.5*\y, -3.7*\x)$)
	to[out=0,in=180] ($(2.05*\y, -3.2*\x)$)
	to[out=0,in=0] ($(2.05*\y, -2.8*\x)$)
	to[out=180,in=0] ($(1.5*\y,-2.85*\x)$)
	to[out=180,in=0] ($(0.95*\y,-2.8*\x)$)
	to[out=180,in=180] ($(0.95*\y,-3.2*\x)$);
	
	\filldraw[fill=gray, fill opacity=0.3]
	($(0.95*\y, -3.2*\x)$)
	to[out=0,in=180] ($(1.5*\y, -3.7*\x)$)
	to[out=0,in=180] ($(2.05*\y, -4.2*\x)$)
	to[out=0,in=0] ($(2.05*\y, -3.8*\x)$)
	to[out=180,in=0] ($(1.5*\y,-3.3*\x)$)
	to[out=180,in=0] ($(0.95*\y,-2.8*\x)$)
	to[out=180,in=180] ($(0.95*\y,-3.2*\x)$);
	
	\filldraw[fill=gray, fill opacity=0.3]
	($(-0.1, -4.2*\x)$)
	to[out=0,in=180] ($(1.5*\y, -3.7*\x)$)
	to[out=0,in=180] ($(2.05*\y, -3.2*\x)$)
	to[out=0,in=0] ($(2.05*\y, -2.8*\x)$)
	to[out=180,in=0] ($(1.5*\y,-3.3*\x)$)
	to[out=180,in=0] ($(-0.1,-3.8*\x)$)
	to[out=180,in=180] ($(-0.1,-4.2*\x)$);
	
	\filldraw[fill=gray, fill opacity=0.3]
	($(-0.1, -4.2*\x)$)
	to[out=0,in=180] ($(1.5*\y, -4.1*\x)$)
	to[out=0,in=180] ($(2.05*\y, -4.2*\x)$)
	to[out=0,in=0] ($(2.05*\y, -3.8*\x)$)
	to[out=180,in=0] ($(1.5*\y,-3.3*\x)$)
	to[out=180,in=0] ($(-0.1,-3.8*\x)$)
	to[out=180,in=180] ($(-0.1,-4.2*\x)$);
	
	\filldraw[fill=gray, fill opacity=0.3]
	($(1.05*\y, -5.8*\x)$)
	to[out=180,in=0] ($(0.5*\y, -5.3*\x)$)
	to[out=180,in=0] ($(-0.05*\y, -5.8*\x)$)
	to[out=180,in=180] ($(-0.05*\y, -6.2*\x)$)
	to[out=0,in=180] ($(0.5*\y,-6.15*\x)$)
	to[out=0,in=180] ($(1.05*\y,-6.2*\x)$)
	to[out=0,in=0] ($(1.05*\y, -5.8*\x)$);

	\filldraw[fill=gray, fill opacity=0.3]
	($(1.05*\y, -5.8*\x)$)
	to[out=180,in=0] ($(0.5*\y, -5.3*\x)$)
	to[out=180,in=0] ($(-0.05*\y, -4.8*\x)$)
	to[out=180,in=180] ($(-0.05*\y, -5.2*\x)$)
	to[out=0,in=180] ($(0.5*\y,-5.7*\x)$)
	to[out=0,in=180] ($(1.05*\y,-6.2*\x)$)
	to[out=0,in=0] ($(1.05*\y, -5.8*\x)$);
	
	\filldraw[fill=gray, fill opacity=0.3]
	($(2.05*\y, -4.8*\x)$)
	to[out=180,in=0] ($(0.5*\y, -5.3*\x)$)
	to[out=180,in=0] ($(-0.05*\y, -5.8*\x)$)
	to[out=180,in=180] ($(-0.05*\y, -6.2*\x)$)
	to[out=0,in=180] ($(0.5*\y,-5.7*\x)$)
	to[out=0,in=180] ($(2.05*\y,-5.2*\x)$)
	to[out=0,in=0] ($(2.05*\y, -4.8*\x)$);
	
	\filldraw[fill=gray, fill opacity=0.3]
	($(2.05*\y, -4.8*\x)$)
	to[out=180,in=0] ($(0.5*\y, -4.9*\x)$)
	to[out=180,in=0] ($(-0.05*\y, -4.8*\x)$)
	to[out=180,in=180] ($(-0.05*\y, -5.2*\x)$)
	to[out=0,in=180] ($(0.5*\y,-5.7*\x)$)
	to[out=0,in=180] ($(2.05*\y,-5.2*\x)$)
	to[out=0,in=0] ($(2.05*\y, -4.8*\x)$);
	
%	\fill (0*\y,-3*\x) circle (0.05)	;
	\fill (1*\y,-3*\x) circle (0.05)	node[label={[label distance=\y]180:$a$}]{};
	\fill (2*\y,-3*\x) circle (0.05)	node[label={[label distance=\y]0:$b$}]{};
	
	\fill (1.5*\y,-3.5*\x) circle (0.05) node[label=below:$c$]{};
	
	\fill (0*\y,-4*\x) circle (0.05)	node[label={[label distance=\y]180:$d$}]{};
	\fill (2*\y,-4*\x) circle (0.05)	node[label={[label distance=\y]0:$e$}]{};
	
	\fill (0*\y,-5*\x) circle (0.05)	node[label={[label distance=\y]180:$f$}]{};
	\fill (2*\y,-5*\x) circle (0.05)	node[label={[label distance=\y]0:$g$}]{};
	
	\fill (0.5*\y,-5.5*\x) circle (0.05)	node[label=below:$h$]{};
	
	\fill (0*\y,-6*\x) circle (0.05)	node[label={[label distance=\y]180:$i$}]{};
	\fill (1*\y,-6*\x) circle (0.05)	node[label={[label distance=\y]0:$j$}]{};
%	\fill (2*\y,-6*\x) circle (0.05)	;
	
	\end{tikzpicture}
	
	\caption{
		An example of a $G$-gadget in a $3$-graph $H$.
		Let $G$ be the graph on $[3]$ consisting of the edges $12$ and $23$. %, i.e. a path with two edges.
		$K$ is the complete $(3,3)$-graph with vertex partition $V_1, V_2, V_3$ (edges not shown) and $W_G$ consists of the union of $W_{12} = \{ a,b,c,d,e \}$ and $W_{23} = \{ f,g,h,i,j \}$, including the edges $\{ abc, bcd, cde, ace, fgh, ghi, hij, fhj \}$.
		The coloured edges are in $H \setminus K$. \newline
		We show an example of~\ref{item:propertyW3}.
		Let $\sigma = \operatorname{id}$, and note that $\sigma(1) = 1$.
		In~$H[W_{12}]$ we find the tight path $abcde$ on $5 = 2k-1$  vertices, whose start-type is $\sigma (123) = (123)$ and its end-type is $(12) \sigma = (12)$.
		This means the first two vertices of $abcde$ are in clusters $V_2, V_3$, and its last two vertices are in clusters $V_1, V_3$, respectively.
	}
	\label{figure:ggadget}
\end{figure}

Suppose $P$ is a tight path with end-type $\pi$ and $\sigma$ is a cyclic permutation. In the next lemma, we show how to extend $P$ into a tight path with end-type $\sigma \pi$ using a $G$-gadget, where $G$ is a path.

\begin{lemma} \label{lemma:gadgetcycle}
	Let $k \ge 3$ and $r \ge 2$.
	Let $\sigma =  (i_1 i_2 \dotsb i_r) \in S_k$ be a cyclic permutation. 
	Let $G$ be a $2$-graph on $[k]$  containing the path $Q = i_1 i_2 \dotsb i_r$.
	Let $\h$ be a $k$-graph containing a complete $(k,k)$-graph $K$ with vertex partition $V_1, \dotsc, V_k$.
	Suppose that $P$ is a tight path in $\h$ with end-type $\pi \in S_k$ such that $\pi(1) = i_r$.
	Suppose $W_G$ is a $G$-gadget avoiding $V(P)$ and $|V_{i_j} \setminus V(P)| \ge 2|E(G)|$ for all $1 \leq j \leq r$.
	Then there exists an extension $P'$ of $P$ with end-type $\sigma \pi$ such that
	\begin{enumerate}[label={\rm (\roman*)}]
		\item $|V(P')| = |V(P)| + 2k(r-1)$, 
		\item for all $1 \le i \le k$, 
		\begin{align*}
		| V_i \cap ( V(P') \setminus V(P) )| = 
		\begin{cases}
		2 (r-1) - 1 & \text{if $i \in \{i_1, i_2, \dotsc, i_{r-1} \}$,}\\
		2 (r-1)	& \text{otherwise,}
		\end{cases}
		\end{align*}
		\item there exists a $(G - Q)$-gadget $W_{G - Q}$ for $K$ avoiding $V(P')$ and
		\item $V(P') \setminus V(P \cup K) = \{ w_{i_j i_{j+1}} : 1 \leq j < r \}$.
	\end{enumerate}
\end{lemma}

\begin{proof}
	We proceed by induction on~$r$.
	First suppose that $r = 2$ and so $\sigma = (i_1 i_2)$.
	Consider a $G$-gadget $W_G$ avoiding~$V(P)$.
	Since $i_1 i_2 \in E(G)$, there exists a set~$W_{i_1 i_2} \subseteq W_{G}$ disjoint from $V(P)$ such that $|W_{i_1 i_2}| = 2k - 1$ and $\h[W_{i_1 i_2}]$ contains a spanning tight path~$P''$ with start-type~$\pi \tau$ and end-type~$(i_1 i_2) \pi = \sigma \pi$. 
	Note that $|V_{i_2} \cap W_G| \le 2 |E(G)| - 1$, as~$|V_{i_2} \cap  W_{i_1i_2}| =1$.
	Hence $V_{i_2} \setminus \left( V(P) \cup W_G \right) \ne \emptyset$.
	Take an arbitrary vertex~$x_{i_2} \in V_{i_2} \setminus \left( V(P) \cup W_G \right)$ and set $P' = P x_{i_2} P''$.
	Since $\pi( 1 ) = i_2$, it follows that $P'$ is a tight path with end-type~$ \sigma \pi$, and $P'$ satisfies properties (i), (ii) and (iv).
	Set $W_{G - i_1 i_2} = W_G \setminus W_{i_1 i_2}$. Then $W_{G - i_1 i_2}$ is a $(G - i_1 i_2)$-gadget for $K$ avoiding $V(P')$, so $P'$ satisfies property (iii), as desired.
	
	Next, suppose $r > 2$. Define $\sigma' = (i_2 i_3 \dotsb i_{r})$ and note that $\sigma = (i_1 i_2) \sigma'$. Then $\sigma'$ is a cyclic permutation of length $r-1$,  with $\pi(1) = i_r$ and the path~$Q' = i_2 \dotsb i_{r-1} i_r$ is a subgraph of~$G$. By the induction hypothesis, there exists an extension~$P''$ of $P$ with end-type~$\sigma' \pi$ such that $|V(P'')| = |V(P)| + 2k(r-2)$ and,
	for all $1 \le i \le k$, 
	\begin{align*}
	| V_i \cap ( V(P'') \setminus V(P) )| = 
	\begin{cases}
	2 (r-2) - 1 & \text{if $i \in \{i_2, i_3, \dotsc,  i_{r-1} \}$},\\
	2 (r-2)	& \text{otherwise.}
	\end{cases}
	\end{align*}
	Moreover, there exists a $(G - Q')$-gadget $W_{G - Q'}$ avoiding $V(P'')$ and $V(P'') \setminus V(P \cup K) = \{ w_{i_j i_{j+1}} : 2 \leq j < r \}$.
	
	Note that $\sigma' \pi (1) = \sigma'( i_r ) = i_2$ and $i_1 i_2 \in E(G - Q')$.
	For all $1 \le i \le r$, $| V_i \setminus V(P') | \ge 2 |E(G - Q')| $.
	Again by the induction hypothesis, there exists an extension $P'$ of $P''$ with end-type $(i_1 i_2) \sigma' \pi = \sigma \pi$ such that $|V(P')| = |V(P'')| + 2k = |V(P)| + 2k(r-1)$ and,
	for all $1 \le i \le k$, 
	\begin{align*}
	| V_i \cap ( V(P') \setminus V(P'') )| = 
	\begin{cases}
	1 & \text{if $i = i_1$},\\
	2	& \text{otherwise.}
	\end{cases}
	\end{align*} and $V(P') \setminus (V(P'' \cup K)) = \{ w_{i_1 i_2} \}$, so $P'$ satisfies properties (i), (ii) and (iv). Furthermore, set $W_{G - Q} = W_G - \bigcup_{j = 1}^{r-1} W_{i_j i_{j+1}}$. Then $W_{G - Q}$ is a $(G - Q)$-gadget for $K$ avoiding $V(P')$, so $P'$ satisfies property (iii) as well.
\end{proof}

In the next lemma, we show how to extend a path with end-type $\operatorname{id}$ to one with an arbitrary end-type.
We will need the following definitions.
Consider an arbitrary $\sigma \in S_k \setminus \{ \operatorname{id} \}$.
Write $\sigma$ in its cyclic decomposition 
\begin{align*}
\sigma =  (i_{1,1} i_{1,2} \dotsb i_{1, r_1} )  (i_{2,1} i_{2,2} \dotsb i_{2, r_2})
\dotsb
( i_{t,1} i_{t,2} \dotsb i_{t, r_t} )   ,
\end{align*}
where $\sigma$ is a product of $t = t(\sigma)$ disjoint cyclic permutations of respective lengths $r_1, \dotsc, r_t$ so that $r_j \ge 2$ and $i_{j,r_j} = \min \{ i_{j, r'} : 1 \leq r' \leq r_j \}$ for all $1 \leq j \leq t$; and $i_{1,r_1} < i_{2,r_2} < \dotsb < i_{t,r_t}$.
Define $m(\sigma) = i_{t,r_t}$. On the other hand, if $\sigma = \operatorname{id}$, then define $t(\sigma) = 0$ and $m(\sigma) = 1$.
Define $G_{\sigma}$ to be the $2$-graph on $[k]$ consisting precisely of the (vertex-disjoint) paths $Q_j = i_{j,1} i_{j,2} \dotsb i_{j,r_j}$ for all $1 \leq j \leq t(\sigma)$.
So $G_{\operatorname{id}}$ is an empty $2$-graph.
Note that for all $\sigma$,
\begin{align}
2| E(G_{\sigma})| + t(\sigma) = 2 \sum_{j=1}^{t(\sigma)} r_j - t(\sigma) \le 2k-1. \label{eqn:sumr_j}
\end{align}
For $1 \leq i \leq k$ and $\sigma \in S_k \setminus \{ \operatorname{id} \}$, set $X_{i, \sigma} = 1$ if $i \in \{ i_{t', 1}, \dotsc, i_{t', r_{t'} - 1} \}$ for some $1 \leq t' \leq t$, and $X_{i, \sigma} = 0$ otherwise.
Also, for $1 \leq i \leq k$, set $Y_{i, \sigma} = 1$ if $i \in \{ \sigma(j) : 1 \leq j < m(\sigma) \}$ and $Y_{i, \sigma} = 0$ otherwise. If $\sigma = \operatorname{id}$, then define $X_{i, \sigma} = Y_{i, \sigma} = 0$ for all $1 \leq i \leq k$.

\begin{lemma} \label{lemma:gadgetsigma}
	Let $k \ge 3$. Let $\h$ be a $k$-graph containing a complete $(k,k)$-graph $K$ with vertex partition $V_1, \dotsc, V_k$ and a tight path $P$ with end-type~$\operatorname{id}$.
	Let $\sigma \in S_k$ and let $G$ be a $2$-graph on $[k]$ containing~$G_{\sigma}$.
	Suppose that $K$ has a $G$-gadget $W_G$ avoiding $V(P)$, and $|V_i \setminus V(P)| \ge 2 | E(G) | + 2$.
	Then there exists an extension $P'$ of $P$ with end-type $\sigma \tau^{m(\sigma) - 1}$ such that
	\begin{enumerate}[label = {\rm (\roman*)}]
		\item $|V(P')| = |V(P)|+ 2 k |E(G_{\sigma})| + m(\sigma) - 1$,
		\item for all $1 \le i \le k$, $|V_i \cap ( V(P') \setminus V(P) )| = 2|E(G_\sigma)| - X_{i, \sigma} + Y_{i, \sigma}$,
		\item $K$ has a $(G - G_{\sigma})$-gadget avoiding $V(P')$ and
		\item $V(P') \setminus V(P \cup K) = \{ w_{ij} : ij \in E(G_\sigma) \}$.
	\end{enumerate}
\end{lemma}

\begin{proof}
	Let
	\begin{align*}
	\sigma =  (i_{1,1} i_{1,2} \dotsb i_{1, r_1} )  (i_{2,1} i_{2,2} \dotsb i_{2, r_2})
	\dotsb
	( i_{t,1} i_{t,2} \dotsb i_{t, r_t} )  
	\end{align*}
	as defined above. We proceed by induction on $t = t(\sigma)$. If $t = 0$, then $\sigma = \operatorname{id}$ and $m(\sigma) = 1$, so the lemma holds by setting $P' = P$. Now suppose that $t \ge 1$ and the lemma is true for all $\sigma' \in S_k$ with $t(\sigma') <t$. Let 
	\begin{align*}
	\sigma_1 = (i_{1,1} i_{1,2} \dotsb i_{1, r_1} ) (i_{2,1} i_{2,2} \dotsb i_{2, r_2}) \dotsb ( i_{t-1,1} i_{t-1,2} \dotsb i_{t-1, r_{t-1}} )
	\end{align*}
	and $\sigma_2 = ( i_{t,1} i_{t,2} \dotsb i_{t, r_t} ) $,  so $\sigma_1 \sigma_2 = \sigma_2 \sigma_1 = \sigma$. 
	For $1 \le i \le 2$, let $G_i = G_{\sigma_i}$ and $m_i = m(\sigma_i)$.
	Note that $G_{\sigma} = G_1 \cup G_2$.
	Let $G' = G - G_1$.
	Since $t(\sigma_1)= t-1$, by the induction hypothesis, there exists a path $P_1$ that extends $P$ with end-type $\sigma_1 \tau^{ m_1- 1 }$ such that
	\begin{enumerate}[label = {\rm (\roman*$'$)}]
		\item $|V(P_1)| = |V(P)| + 2 k |E(G_1)|   + m_1 - 1$,
		\item \label{itemlist:lemdos} for all $1 \le i \le k$, $|V_i \cap ( V(P_1) \setminus V (P) )| = 2 |E(G_1)| - X_{i, \sigma_1} + Y_{i, \sigma_1}$,
		\item $K$ has a $G'$-gadget $W_{G'}$ avoiding $V(P_1)$ and
		\item $V(P_1) \setminus V(P \cup K) = \{ w_{ij} : ij \in E(G_1) \}$.
	\end{enumerate}
	Note that for all $1 \le i \le k$,
	\begin{align*}
	\left| V_i \setminus \left( V(P_1) \cup W_{G'} \right) \right| \ge 2 |E(G)| + 2  - (2 |E(G_1)| + 1 ) - 2 |E(G')| =1.
	\end{align*}
	We extend $P_1$ using $m_2 - m_1 >0 $ simple extensions, avoiding the set $V(P_1) \cup W_{G'}$ in each step, to obtain an extension $P_2$ of $P_1$ with end-type~$\sigma_1 \tau^{m_1- 1} \tau^{m_2 - m_1} = \sigma_1 \tau^{m_2 - 1}$ such that  
	\begin{align*}
	|V(P_2)|  = |V(P_1)| + m_2 - m_1 = |V(P)| + 2 k |E(G_1)| + m_2 - 1
	\end{align*}
	and $W_{G'}$ is a $G'$-gadget for $K$ that avoids~$V(P_2)$.
	As $P_1$ has end-type~$\sigma_1 \tau^{m_1 - 1}$, $V(P_2) \setminus V(P_1)$ contains precisely one vertex in $V_i$ for all $i \in \{ \sigma_1 \tau^{m_1 - 1}(j) : 1 \leq j \leq m_2 - m_1 \} = \{ \sigma_1 ( m_1 ), \dotsc, \sigma_1(m_2 - 1)  \}$.
	Since $\sigma_1(i) = \sigma(i)$ for all $m_1 \leq i < m_2$ and $m_2 = i_{t, r_t}$, together with \ref{itemlist:lemdos} we deduce that \begin{align}
	|V_i \cap ( V(P_2) \setminus V (P) )| = 2 |E(G_1)| - X_{i, \sigma_1} + Y_{i, \sigma}. \label{eq:gadgetlemmaaftersimpleextension}
	\end{align}		
	
	Note that $\sigma_1 \tau^{m_2 - 1}(1) = \sigma_1 ( m_2 ) = \sigma_1 ( i_{t, r_t} ) = i_{t, r_t}$. 
	Since $G'$ contains $G_{2}$, by Lemma~\ref{lemma:gadgetcycle} there exists an extension $P'$ of $P_2$ with $|V(P')| = |V(P_2)| + 2k|E(G_2)|$ and $P'$ has end-type $\sigma_2 \sigma_1  \tau^{m_2 - 1} = \sigma \tau^{m(\sigma) - 1}$, as $m_2 = m(\sigma)$.
	Moreover, as $G' - G_2 = G - G_{\sigma}$, $K$ has a $(G - G_{\sigma})$-gadget avoiding $V(P')$, implying (iii).
	Similarly, (iv) holds.
	Note that 
	\begin{align*}
	|V(P')| = |V(P_2)| + 2k|E(G_2)| = |V(P)| + 2k |E(G_{\sigma})| + m(\sigma) - 1
	\end{align*}
	implying (i). Finally, for  all $1 \le i \le k$, we have \[ |V_i \cap (V(P') \setminus V(P_2) )| = \begin{cases}
	2 |E(G_2)| - 1 & \text{if $i \in \{i_{t,1}, \dotsc, i_{t, r_{t} - 1} \}$},\\
	2 |E(G_2)|	& \text{otherwise.}
	\end{cases} \]
	So $|V_i \cap (V(P') \setminus V(P_2))| = 2|E(G_2)| - X_{i, \sigma_2}$.
	Note that $X_{i, \sigma} = X_{i, \sigma_1} + X_{i, \sigma_2}$ because $\sigma_1$ and $\sigma_2$ are disjoint.	
	Thus, together with \eqref{eq:gadgetlemmaaftersimpleextension}, (ii) holds.
 \end{proof}

Now we want to use the previous lemmas to find tight cycles of a given length.
%Let $K$ be a complete $(k,k)$-graph with classes $V_1, \dotsc, V_k$.
Let $P$ be a tight path with start-type~$\sigma$ and end-type~$\pi$. If $\pi = \sigma$, then there exists a tight cycle $C$ containing $P$ with $V(C) = V(P)$. Similarly if $\pi = \sigma \tau^{-r}$, then (by using~$r$ simple extensions) there exists a tight cycle $C$ on $|V(P)| + r$ vertices containing~$P$.
In general, in order to extend $P$ into a tight cycle we use Lemma~\ref{lemma:gadgetsigma} to first extend $P$ to a path $P'$ with end-type $\sigma \tau^{- r}$ for some suitable~$r$, using the edges of $K$ and a suitable $G$-gadget.
The next lemma formalises the aforementioned construction of the tight cycle $C$ containing $P$ and gives us precise bounds on the sizes of $V_i \cap (V(C) \setminus V(P))$ in the case where $\sigma = \pi$, which will be useful during Section~\ref{section:tilings}.

\begin{lemma} \label{lemma:gadgettightcycle}
	Let $k \ge 3$. Let $\sigma, \pi \in S_k$ and $0 \le r < k$.
	Then there exists a $2$-graph $G:= G(\sigma, \pi,r)$ on $[k]$ consisting of a vertex-disjoint union of paths such that the following holds for all $s \ge k(2k-1)$ with $s \equiv r \bmod k$:
	let $\h$ be a $k$-graph containing a complete $(k,k)$-graph $K$ with vertex partition $V_1, \dotsc, V_k$, and let $P$ be a tight path with start-type $\sigma$ and end-type~$\pi$.
	Suppose $W_G$ is a $G$-gadget for $K$ avoiding $V(P)$ and $| V_i \setminus V(P) | \ge \lfloor s /k \rfloor + 1$.
	Then there exists a tight cycle $C$ on $|V(P)| + s$ vertices containing $P$, such that \[ V(C) \setminus (V(P \cup K)) = \{ w_{ij} : ij \in E(G) \}. \]
	Moreover, if $\sigma = \pi$, then for all $1 \le i,j \le k$, \[ \big| | V_i \cap ( V(C) \setminus V(P)) | - | V_j \cap ( V(C) \setminus V(P)) | \big| \le 1.\]
\end{lemma}

\begin{proof}
	Without loss of generality, we may assume that $\pi = \operatorname{id}$.
	Define $\sigma' = \sigma \tau^{-r}\in S_k$.
	Let $G = G_{\sigma'}$.
	Note that $|E(G)| \le k-1$, $t(\sigma')\le k/2$ and $2|E(G)| + t(\sigma') \le 2k - 1$ by~\eqref{eqn:sumr_j}.
	Let $\h,K,P$ be as defined in the lemma.
	By Lemma~\ref{lemma:gadgetsigma}, there exists an extension $P'$ of $P$ with end-type $\sigma' \tau^{m(\sigma') - 1}$ such that 
	$|V(P')| = |V(P)|+ 2 k |E(G)| + m(\sigma') - 1$, for all $1 \le i \le k$,
	\begin{align*}
	|V_i \cap ( V(P') \setminus V (P) )| = 2 |E(G)| - X_{i, \sigma'} + Y_{i, \sigma'}
	\end{align*} and $V(P') \setminus (V(P \cup K)) = \{ w_{ij} : ij \in E(G) \}$.
	We use $k - m(\sigma') + 1$ simple extensions to get an extension $P''$ of $P'$ of order
	\begin{align*}
	|V(P'')| = |V(P')| + (k - m(\sigma') + 1)  = |V(P)| + 2 k |E(G)|+k.
	\end{align*}
	Note that $V(P'') \setminus V(P')$ uses precisely one vertex in each of the clusters $V_i$ for all $i \in \{ \sigma' \tau^{m(\sigma') - 1}(j) : 1 \leq j \leq k - m(\sigma') + 1  \} = \{ \sigma'(j) : m(\sigma') \leq j \leq k \} = \{ j : Y_{j, \sigma'} = 0 \}$. It follows that for all $1 \le i \le k$,
	\begin{align*}
	|V_i \cap ( V(P'') \setminus V (P) )| = 2 |E(G)| + 1 - X_{i, \sigma'}.
	\end{align*} Note that $P''$ has end-type $\sigma ' \tau^{m(\sigma') - 1} \tau^{k - m(\sigma') + 1} = \sigma' = \sigma \tau^{-r}$.
	For all $1 \leq i \leq k$ and $0 \leq r < k$, set $Z_{i, \sigma, r} = 1$ if $i \in \{ \sigma(j) : k-r+1 \leq j \leq k \}$, and set $Z_{i, \sigma, r} = 0$ otherwise. We use $r$ more simple extensions to get an extension $P'''$ of $P$ with end-type $\sigma \tau^{-r} \tau^r = \sigma$ of order
	\begin{align*}
	|V(P''')| = |V(P'')| + r  = |V(P)| + 2 k |E(G)|+k + r
	\end{align*} such that, for all $1 \le i \le k$,
	\begin{align*}
	|V_i \cap ( V(P''') \setminus V (P) )| = 2 |E(G)| + 1 + Z_{i, \sigma, r} - X_{i, \sigma'}.
	\end{align*}
	
	Since $|E(G)| \leq k-1$ and $s \equiv r \bmod k$, it follows that $|V(P''')| \leq  |V(P)| + s$. Also, $|V(P''') \setminus V(P)| \equiv s \bmod{k}$.  For all $1 \le i \le k$,
	\begin{align*}
	|V_i \setminus V(P''') | 
	& \ge |V_i \setminus V(P) | - 2 |E(G)| - 1 + X_{i, \sigma'} - Z_{i, \sigma, r} \\
	& \ge \lfloor s / k \rfloor - 2 |E(G)| - 1 = \frac{1}{k} \left( k \lfloor s / k \rfloor - 2 k |E(G)| - k \right) \\
	& = \frac{1}{k} \left( s - r - 2 k |E(G)| - k \right) = \frac{1}{k} \left(  s - (|V(P''')| - |V(P)|) \right).
	\end{align*}
	Since $P'''$ has start-type $\sigma$ and end-type $\sigma$, then we can easily extend $P'''$ (using simple extensions) into a tight cycle $C$ on $|V(P)| + s$ vertices.
	Note that $V(C) \setminus (V(P \cup K)) = \{ w_{ij} : ij \in E(G) \}$, as desired.
	
	Moreover, for all $1 \le i,j \le k$,
	\begin{align*}
	& \big| | V_i \cap ( V(C) \setminus V(P)) | - | V_j \cap ( V(C) \setminus V(P)) | \big| \\
	& = \big| | V_i \cap ( V(P''') \setminus V(P)) | - | V_j \cap ( V(P''') \setminus V(P)) | \big| \\
	& = | ( Z_{i, \sigma, r} - X_{i, \sigma'} ) - ( Z_{j, \sigma, r} - X_{j, \sigma'} ) |.
	\end{align*}	
	Suppose now that $\sigma = \pi = \operatorname{id}$. We will show that $-1 \leq Z_{i, \sigma, r} - X_{i, \sigma'} \leq 0$ for all $1 \leq i \leq k$, implying that for all $1 \leq i, j \leq k$, $\big| | V_i \cap ( V(C) \setminus V(P)) | - | V_j \cap ( V(C) \setminus V(P)) | \big| \leq 1$.
	It suffices to show that if $Z_{i, \sigma, r} = 1$, then $X_{i, \sigma'} = 1$.
	If $r = 0$ then it is obvious, so suppose that $1 \leq r < k$.
	Let $1 \le i \le k$ such that $Z_{i, \sigma, r} = 1$.
	Since $\sigma = \pi = \operatorname{id}$, then $\sigma' = \tau^{-r}$.
	So if $Z_{i, \sigma, r} = 1$, then $k - r + 1 \leq i \leq k$.
	To show that $X_{i, \tau^{-r}} = 1$, we need to show that $i$ is not the minimal element in the cycle that it belongs in the cyclic decomposition of $\tau^{-r}$, that is, there exists $m < i$ such that $i$ is in the orbit of $m$ under~$\tau^{-r}$.
	Let $d = \gcd(r, k)$.
	Choose $1 \leq m \leq d$ such that $m \equiv i \bmod d$.
	The order of $\tau^{-r}$ is exactly $k/d$ and the orbit of $m$ has exactly $k/d$ elements.
	There are exactly $k/d$ elements $i'$ satisfying $1 \leq i' \leq k$ and $i' \equiv m \bmod d$, and all elements $i'$ in the orbit of $m$ also satisfy $i' \equiv m \bmod d$, so it follows that $i$ is in the orbit of $m$ under~$\tau^{-r}$.
	Finally, $m \leq d \leq k - r < i$.
	This proves that $X_{i, \tau^{-r}} = 1$, as desired.
\end{proof}

\subsection{Finding $G$-gadgets in $k$-graphs with large codegree}

We now turn our attention to the existence of $G$-gadgets.
We prove that all large complete $(k,k)$-graphs contained in a $k$-graph~$\h$ with $\delta_{k-1}(H)$ large have a $G$-gadget, for an arbitrary $2$-graph $G$ on~$[k]$.

\begin{lemma} \label{lemma:kisnice}
	Let $0 < 1/n, 1/t_0 \ll \gamma, 1/k$.
	Let $\h$ be a $k$-graph on $n$ vertices with $\delta_{k-1}(\h) \ge (1/2 + \gamma)n$ containing a complete $(k,k)$-graph $K$ with vertex partition $V_1, \dotsc, V_k$.
	Let $S \subseteq V(\h)$ be a set of vertices such that $|V(K) \cup S| \leq \gamma n /2$ and $|V_i \setminus S| \ge t_0$ for all $1 \leq i \leq k$.
	Let $G$ be a $2$-graph on~$[k]$.
	Then there exists a $G$-gadget for $K$ avoiding~$S$.
\end{lemma}

\begin{proof}
	Choose $0 < 1/t \ll \gamma, 1/k$ and let $t_0 \defined t + k^2$.
	Suppose that $ij \in E(G)$ and $|V_\ell \setminus S| \ge t + 2 |E(G)|$ for all $1 \leq \ell \leq k$.
	Let $U_\ell \subseteq V_\ell \setminus S$ with $|U_\ell| = t$ for all $1 \leq \ell \leq k$ and let~$R = [k] \setminus \{i, j\}$.
	Let $U = \bigcup_{1 \leq \ell \leq k} U_\ell$ and \[ T = \left\{ A \in \binom{U}{k-1} : | A \cap U_r| = 1 \text{ for all } r \in R \text{ and } |A \cap (U_i \cup U_j)| = 1  \right\}. \]
	Then $T$ has size $2 t^{k-1}$.
	By the codegree condition, all members in $T$ have $(1/2 + \gamma)n - |V(K) \cup S| \ge (1/2 + \gamma/2)n$ neighbours outside of $V(K) \cup S$ and by an averaging argument, there exists a vertex $w \notin V(K) \cup S$ such that $\h(w)$ satisfies $| \h(w) \cap T | \ge (1 + \gamma) t^{k-1} $.
	For all $u \in U_i \cup U_j$, $N_{\h(w) \cap T}(u)$ is a family of $(k-2)$-sets of $\bigcup_{r \in R} U_r$. We have that \begin{align*}
	& \sum_{(u_i, u_j) \in U_i \times U_j} |N_{\h(w) \cap T}(u_i) \cap N_{\h(w) \cap T}(u_j)| \\
	& \qquad \qquad \quad \ \ge \sum_{(u_i, u_j) \in U_i \times U_j} \left( d_{\h(w) \cap T}(u_i) + d_{\h(w) \cap T}(u_j) - t^{k-2} \right) \\
	& \qquad \qquad \quad \  = t |\h(w) \cap T| - t^k \ge t^k(1 + \gamma) - t^k = \gamma t^k,
	\end{align*} and by an averaging argument, there exists a pair $(x^\ast_i, x^\ast_j) \in U_i \times U_j$ such that $|N_{\h(w) \cap T}(x^\ast_i) \cap N_{\h(w) \cap T}(x^\ast_j)| \ge \gamma t^{k-2}$.
	
	By the choice of $t$ and by Theorem~\ref{theorem:kovarisosturanhypergraphs}, we have that $N_{\h(w) \cap T}(x^\ast_i) \cap N_{\h(w) \cap T}(x^\ast_j)$ contains a copy $K'$ of~$K^{k-2}_{k-2}(2)$. Define $W_{ij} = V(K') \cup \{ w, x_i^\ast, x_j^\ast \}$ and note that $|W_{ij}| = 2(k-2) + 3 = 2k - 1$.
	
	We now check that~\ref{item:propertyW3} holds for $W_{ij}$.
	Recall that, informally, this means that given any $k-1$ ordered clusters that miss~$V_i$, there exists a tight path with vertex set~$W_{ij}$, which starts with the same ordered $k-1$ clusters and ends with the same ordered $k-1$ clusters but with $V_j$ replaced by~$V_i$. 
	For all $r \in R$, let $U_r \cap V(K') = \{ x_r, x'_r\}$.
	Consider an arbitrary $\sigma \in S_k$ with $\sigma(1) = i$ and $\sigma(j') = j$.
	By construction, we have that \[ x_{\sigma(2)} x_{\sigma(3)} \dotsb x_{\sigma(j' - 1)} x^\ast_j x_{\sigma(j' + 1)} x_{\sigma(j' + 2)} \dotsb x_{\sigma(k)} w x'_{\sigma(2)} x'_{\sigma(3)} \dotsb x'_{\sigma(j' - 1)} x^\ast_i x'_{\sigma(j'+1)} x'_{\sigma(j'+2)} \dotsb x'_{\sigma(k)} \] is a spanning tight path in $\h[W_{ij}]$, of start-type $\sigma \tau$ and end-type~$(ij)\sigma$.
	Clearly $W_{ij}$ is an $ij$-gadget avoiding~$S$.
	
	Set $S' = S \cup W_{ij}$ and $G' = G - ij$.
	Repeating this construction for all edges in $E(G - ij)$ and using that $t_0 = t + k^2$, it is possible to conclude that $K$ has a $G$-gadget avoiding~$S$.
\end{proof}

\subsection{Auxiliary $k$-graphs $F_s$} \label{subsection:auxiliarygraphs}

Given a tight cycle $C_s^k$, we would like to find a $k$-graph $F_s$ such that $C_s^k \subseteq F_s$ and $F_s$ is obtained from a complete $(k,k)$-graph by adding ``few'' extra vertices.
This will be useful in Section~\ref{section:tilings}.

Let $K$ be a $(k,k)$-graph with vertex partition $V_1, \dotsc, V_k$.
Consider a $2$-graph $G$ on $[k]$ with $E(G) = \{ j_i j'_i : 1 \leq i \leq \ell \}$ and let $y_1, \dotsc, y_{\ell}$ be a set of $\ell$ vertices disjoint from $V(K)$.
Let $W_{G} := \{ y_1, \dotsc, y_{\ell} \}$.
We define the \emph{$G$-augmentation of $K$} to be the $k$-graph $F = F(K,G)$ such that \begin{align*}
V(F) & = V(K) \cup W_{G} \text{ and }\\
E(F) & = E(K) \cup \bigcup_{1 \leq i \leq \ell} ( E(H(y_i, j_i)) \cup E(H(y_i, j'_i)) ),
\end{align*} where $H(v,j)$ is a complete $(k,k)$-graph with partition $\{ v \}, V_1, V_2, \dotsc, V_{j-1}, V_{j+1}, \dotsc, V_k$.

The easy (but crucial) observation is that if $|V_i| \ge 2 \ell$ for all $1 \leq i \leq k$,
then the $G$-augmentation of $K$ contains a $G$-gadget for $K$ avoiding $\emptyset$.
%We write $X_G$ for $\{ w_1, \dotsc, w_{\ell} \}$.
Using that, we can prove the following.

\begin{proposition} \label{corollary:fs}
	Let $k \ge 3$, $s \ge 2k^2$ and $s \not\equiv 0 \bmod k$.
	Then there exists a $2$-graph $G_s$ on $[k]$ that is a disjoint union of paths, and $a_{s,1}, \dotsc, a_{s,k}, \ell \in \mathbb{N}$	such that
	$| a_{s,i} - a_{s,j}  | \leq 1$ for all $i, j \in [k]$,
	$\ell = |E(G_s)| \leq k - 1$,
	and if $K = K^{k}(a_{s,1}, \dotsc, a_{s,k})$, then $F_s$, the $G_s$-augmentation of $K$, contains a spanning $C_s^k$ and $|V(F_s) \setminus V(K)| = \ell$.
\end{proposition}

\begin{proof}
	Let $r \in \{1, \dotsc, k-1 \}$ be such that $s \equiv r \bmod k$.
	Let $G_s$ be the $2$-graph obtained from Lemma~\ref{lemma:gadgettightcycle} (with parameters $\sigma = \pi = \operatorname{id}$ and $r$).
	Note $G_s$ is a disjoint union of paths and thus $\ell = E(G_s) \leq k-1$.
	
	Suppose that $V_1, \dotsc, V_k$ are disjoint sets of size $\lfloor s/k \rfloor + 1$ and let $K'$ be the complete $(k,k)$-graph with partition $\{ V_1, \dotsc, V_k \}$.	
	For all $i \in [k]$ let $v_i \in V_i$ and consider the tight path $P = v_1 \dotsb v_k$.
	Note that $P$ has both start-type and end-type $\operatorname{id}$.
	Let $F'$ be the $G_s$-augmentation of $K'$.
	It is easily checked that $|V_i \setminus V(P)| \ge 2(k-1) \ge 2\ell$ and therefore there is a $G_s$-gadget for $K'$ in $F'$ avoiding $V(P)$.
	By the choice of $G_s$, $F'$ contains a tight cycle $C$ on $s$ vertices containing $P$ such that $V(C) \setminus V(K) = V(F') \setminus V(K') = W_{G_s}$ and, over the range $i \in [k]$, the values $|V(C) \cap V_i|$ differ at most by $1$.
	It is easily checked that letting $a_{s,i} := |V(C) \cap V_i|$ we obtain the desired properties.
\end{proof}

%Note that the statement of Lemma~\ref{lemma:gadgettightcycle} does not imply the corollary but its proof does.
%The graph $G_s$, the integers $a_{s,1}, \dotsc, a_{s,k}, \ell$ and the $k$-graph $F_s$ will be useful in Section~\ref{section:tilings}.

\section{Covering thresholds for tight cycles} \label{section:upperboundcoveringthresholds}

In this section, we prove the upper bounds for the covering codegree threshold for tight cycles, proving Proposition~\ref{proposition:coveringthresholdmodk} and Theorem~\ref{theorem:coveringthresholdnotmodk}.
We first prove Proposition~\ref{proposition:kkscovering}, which immediately implies Proposition~\ref{proposition:coveringthresholdmodk} since $K^k(s)$ contains a $C^k_{s'}$-covering for all $s' \equiv 0 \bmod k$ with $s' \leq sk$.
We will use the following classic result of K\H{o}v\'ari, S\'os and Tur\'an~\cite{KovariSosTuran1954}.

\begin{theorem}[K\H{o}v\'ari, S\'os and Tur\'an~\cite{KovariSosTuran1954}] \label{theorem:kovarisosturan}
	Let $z(m,n; s,t)$ denote the maximum possible number of edges in a bipartite $2$-graph $G$ with parts $U$ and $V$ for which $|U| = m$ and $|V| = n$, which does not contain a $K_{s,t}$ subgraph with $s$ vertices in $U$ and $t$ vertices in~$V$. Then \[ z(m,n; s,t) < (s-1)^{1/t} (n - t + 1) m^{1 - 1/t} + (t-1)m. \]
\end{theorem}

\begin{proposition} \label{proposition:kkscovering}
	For all $k \ge 3$ and $s \ge 1$, let $n, c \ge 2$ such that $1/n, 1/c \ll 1/k, 1/s$.
	Then $c(n, K^k(s)) \leq c n^{1 - 1/s^{k-1}}$.
\end{proposition}

\begin{proof}
	Let $\h$ be a $k$-graph on $n$ vertices with $\delta_{k-1}(\h) \ge  c n^{1 - 1/s^{k-1}} $.
	Fix a vertex $x \in V(\h)$ and consider the link $(k-1)$-graph $\h(x)$ of~$x$.
	Let $U_1 := E(\h(x))$.
	Note that
	\begin{align} | U_1 | \ge \frac{\binom{n-1}{k-2} \delta_{k-1}(\h)}{k - 1} \ge c^{1/2} n^{k - 1 - 1/s^{k-1}}. \label{eq:cotaenh} \end{align}
	
	Let $U_2 := V(\h) \setminus \{ x\}$.
	Consider the bipartite $2$-graph $B$ with parts $U_1$ and $U_2$, where $e \in U_1$ is joined to $u \in U_2$ if and only if $e \cup \{u\} \in E(\h)$.
	By the codegree condition of $\h$, all $(k-1)$-sets $e \in U_1$ have degree at least $\delta_{k-1}(\h) - 1$ in~$B$.
	Hence \begin{align} |E(B)| \ge | U_1|(\delta_{k-1}(\h) - 1) \ge | U_1| ( c n^{1 - 1/s^{k-1}} - 1 ). \label{eq:cotaeneB} \end{align}
	
	We claim there is a $K_{n^{k-1-1/s^{k-2}},s-1}$ as a subgraph in $B$, with $n^{k-1-1/s^{k-2}}$ vertices in~$U_1$ and $s-1$ vertices in $U_2$.
	Suppose not.
	Then, by Theorem~\ref{theorem:kovarisosturan}, \begin{align*}
	|E(B)|
	& \leq z(| U_1|, n-1; n^{k - 1 - 1/s^{k-2}}, s-1) \\ % \le  z(| E( \h(x) )|, n, n^{k - 1 - 1/s^{k-2}}, s-1) \\
	& < \left(n^{k - 1 - 1/s^{k-2}} \right)^{\frac{1}{s-1}} n | U_1|^{1 - \frac{1}{s-1}} + (s-1) | U_1| \\
	& = | U_1| \left( n \left( \frac{n^{k - 1 - 1/s^{k-2}}}{| U_1|}  \right)^{\frac{1}{s-1}} + s-1 \right) \\
	&  \stackrel{\mathclap{\eqref{eq:cotaenh}}}{\leq} | U_1| \left( c^{-\frac{1}{2(s-1)}} n^{1 - \frac{1}{s^{k-1}}} + s-1 \right) < | U_1|  n^{1 - \frac{1}{s^{k-1}}}. \end{align*}
	This contradicts~\eqref{eq:cotaeneB}.
	
	Let $K$ be a copy of $K_{n^{k-1-1/s^{k-2}},s-1}$ in~$B$.
	Let $W := V(K) \cap U_1$ and $X := \{ x_1, \dotsc, x_{s-1} \} = V(K) \cap U_2$.
	Since $|W| = n^{k - 1 - 1/s^{k-2}}$ and $1/n \ll 1/k, 1/s$, by Theorem~\ref{theorem:kovarisosturanhypergraphs}, $W$ contains a copy $K'$ of~$K^{k-1}(s)$.
	By construction, for all $y \in \{ x \} \cup X$ and all $e \in E(K')$, $\{ y \} \cup e \in E(\h)$.
	Hence, $\h[ \{ x \} \cup X \cup V(K')]$ contains a $K^k(s)$ covering $x$, as desired.
\end{proof}

We are ready to prove Theorem~\ref{theorem:coveringthresholdnotmodk}.

\begin{proof}[Proof of Theorem~\ref{theorem:coveringthresholdnotmodk}]
	Let $t \in \mathbb{N}$ be such that $1/{n_0} \ll 1/t \ll \gamma, 1/s$. Let $\h$ be a $k$-graph on $n \ge n_0$ vertices with $\delta_{k-1}(\h) \ge (1/2 + \gamma) n$. Fix a vertex $x$ and a copy $K$ of $K^k_k(t)$ containing~$x$, which exists by Proposition~\ref{proposition:kkscovering}. Let $V_1, \dotsc, V_k$ be the vertex partition of $K$ with $x \in V_1$. By the choice of $t$, $|V_i| \ge \max\{2 k^2 + 2, \lfloor s / k \rfloor + 2 \}$ for all $1 \leq i \leq k$.
	
	Let $x_1 = x$ and select arbitrarily vertices $x_i \in V_i$ for $2 \leq i \leq k$.
	Now $P = x_1 \dotsb x_k$ is a tight path on $k$ vertices with both start-type and end-type $\operatorname{id}$.
	Let $G$ be a complete $2$-graph on~$[k]$.
	By Lemma~\ref{lemma:kisnice}, there exists a $G$-gadget for $K$ avoiding~$V(P)$.
	Thus, by Lemma~\ref{lemma:gadgettightcycle}, there exists a tight cycle in $V(\h)$ on $s$ vertices containing $P$, and in turn,~$x$.
\end{proof}

\section{Absorption} \label{section:absorption}

We need the following ``absorbing lemma'', which is a special case of a lemma of Lo and Markstr\"om~\cite[Lemma 1.1]{LoMarkstroem2015}.
%The version we state is a particular case of the original lemma, which is stated in more generality (for example, it works for general $k$-graphs instead of just tight cycles). %  lemma was stated for general $k$-graphs $F$.% and in terms of $(i, \eta)$-closedness.

\begin{lemma}[{\cite[Lemma 1.1]{LoMarkstroem2015}}] \label{lemma:absorbingset}
	Let $s \ge k \ge 3$ and $0 < 1/n \ll \eta, 1/s$ and $0 < \alpha \ll \mu \ll \eta, 1/s$.
	Suppose that $\h$ is a $k$-graph on $n$ vertices and for all distinct vertices $x, y \in V(\h)$ there exist $\eta n^{s-1}$ sets $S$ of size $s-1$ such that $H[S \cup \{ x \}]$ and $H[ S \cup \{ y \} ]$ contain a spanning~$C_s^k$.
	Then there exists $U \subseteq V(\h)$ of size $|U| \leq \mu n$ with $|U| \equiv 0 \bmod s$ such that there exists a perfect $C_s^k$-tiling in $H[U \cup W]$ for all $W \subseteq V(\h) \setminus U$ of size $|W| \leq \alpha n$ with $|W| \equiv 0 \bmod s$.
\end{lemma}

Thus to find an absorbing set $U$, it is enough to find many $(s-1)$-sets $S$ as above for each pair $x, y \in V(\h)$.
First we show that we can find one such~$S$.

\begin{lemma} \label{lemma:closednessonecopy}
	Let $s \ge 5k^2$ with $s \not\equiv 0 \bmod k$.
	Let $1/n \ll \gamma, 1/s$.
	Let $\h$ be a $k$-graph on $n$ vertices with $\delta_{k-1}(\h) \ge (1/2 + \gamma)n$.
	Then for all pair of distinct vertices $x, y \in V(\h)$, there exists $S \subseteq V(\h) \setminus \{ x, y \}$ such that $|S| = s-1$ and both $\h[S \cup \{ x \}]$ and $\h[S \cup \{ y\}]$ contain a spanning~$C_s^k$.
\end{lemma}

\begin{proof}
	Let $1/n \ll 1/t \ll \gamma, 1/s$.
	Consider the $k$-graph $\h_{xy}$ with vertex set $V(\h_{xy}) = (V(H) \setminus \{ x, y \}) \cup \{ z \}$ (for some $z \notin V(\h)$) and edge set \[ E(\h_{xy}) = E(\h \setminus \{ x, y \}) \cup \{ \{ z \} \cup S : S \in N_{\h}(x) \cap N_{\h}(y)  \}.  \]
	Note that $|V(\h_{xy})| = n - 1$ and $\delta_{k-1}(\h_{xy}) \ge \gamma |V(\h_{xy})|$.
	By Proposition~\ref{proposition:kkscovering}, $H_{xy}$ contains a copy $K$ of $K^k_k(t)$ containing~$z$.
	Let $V_1, \dotsc, V_k$ be the vertex partition of $K$ with $z \in V_1$.
	
	Select arbitrarily vertices $v_i \in V_i$ for $2 \leq i \leq k$.
	Let $\h' = \h_{xy} \setminus \{ z, v_2, \dotsc, v_k \}$ and $K' = K \setminus \{ z, v_2, \dotsc, v_k \}$.
	Note that $\delta_{k-1}(\h') \ge (1/2 + \gamma/2) |V(\h')|$ and $K' \subseteq \h'$.
	By Lemma~\ref{lemma:kisnice} with $H'$ and $K'$ playing the roles of $H$ and $K$ respectively, there exists a $K_k$-gadget for $K'$ in~$H'$.
	Hence, there exists a $K_k$-gadget for $K$ in $\h_{xy}$ avoiding $\{ z, v_2, \dotsc, v_k \}$.
	
	Now we construct a copy of $C_s^k$ in $\h_{xy}$ containing~$z$.
	Note that $P = z v_2 \dotsb v_k$ is a tight path on $k$ vertices with start-type and end-type~$\operatorname{id}$.
	Since there exists a $K_k$-gadget for $K$ avoiding $V(P)$, by Lemma~\ref{lemma:gadgettightcycle} $H_{xy}$ contains a copy $C$ of $C_s^k$ containing~$z$.
	
	Finally, let $S = V(C) \setminus \{z\} \subseteq V(\h)$.
	By construction, $|S| = s-1$ and both $H[S \cup \{x\}]$ and $H[S \cup \{y\}]$ contain a spanning $C_s^k$ in $\h$, as desired.
\end{proof}

We now apply the standard supersaturation trick to find many sets~$S$.

\begin{lemma} \label{lemma:hypergeometric}
	Let $k \ge 3$ and $0 < 1/m \ll \gamma, 1/k$.
	Let $\h$ be a $k$-graph on $n \ge m$ vertices with $\delta_{k-1}(\h) \ge (1/2 + \gamma)n$.
	Let $x, y \in V(\h)$ be distinct.
	Then the number of $m$-sets $R \subseteq V(\h) \setminus \{ x, y \}$ such that $\delta_{k-1}( \h[ R \cup \{ x, y \} ] ) \ge (1/2 + \gamma/2) (m+2)$ is at least $\binom{n-2}{m} / 2$.
\end{lemma}

To prove Lemma~\ref{lemma:hypergeometric}, first we recall the following fact about concentration for hypergeometric random variables around their mean (see, e.g.,~\cite[p. 29]{JansonLuczakRucinski2000}).

\begin{lemma} \label{lemma:hypergeometricconcentracion}
	Let $a, \gamma > 0$ with $a + \gamma < 1$.
	Suppose that $S \subseteq [n]$ and $|S| \ge (a + \gamma) n$.
	Then \[ \left| \left\lbrace M \in \binom{[n]}{m} : |M \cap S| \leq am \right\rbrace \right| \leq \binom{n}{m} e^{- \frac{\gamma^2 m}{3(a + \gamma)}} \leq \binom{n}{m} e^{- \gamma^2 m / 3}. \]
\end{lemma}

\begin{proof}[Proof of Lemma~\ref{lemma:hypergeometric}]
	Let $T$ be a $(k-1)$-set in~$V(\h)$.
	Note that, since $1/n \leq 1/m \ll \gamma$,
	\begin{align*}
	|N_{H}(T) \setminus \{x,y\}|
	& \ge \left( \frac{1}{2} + \gamma \right)n - 2
	\ge \left( \frac{1}{2} + \frac{2}{3}\gamma \right)(n-2).	
	\end{align*}
	We call an $m$-set $R \subseteq V(H) \setminus \{x,y\}$ \emph{bad for $T$} if $|N_H(T) \cap R| \leq (1/2 + 3\gamma/5)m$.
	%Let $\Phi_T$ be the number of $m$-sets which are bad for $T$.
	An application of Lemma~\ref{lemma:hypergeometricconcentracion} (with $1/2 + 3\gamma/5$, $\gamma/15$, $n-2$, $N_H(T) \setminus \{x,y\}$ playing the roles of $a$, $\gamma$, $n$ and $S$, respectively) implies that the number of $m$-sets which are bad for $T$ is at most \begin{align*}
		\left| \left\lbrace R \in \binom{V(H) \setminus \{x,y\}}{m} : |N_H(T) \cap R| \leq (1/2 + 3\gamma/5)m \right\rbrace \right| \leq \binom{n-2}{m} e^{- \gamma^2 m / 675}.
	\end{align*}
	Say an $m$-set $R \subseteq V(\h) \setminus \{ x, y \}$ is \emph{good} if $\delta_{k-1}( R \cup \{x,y\} ) > (1/2 + 3 \gamma / 5)m$ (and \emph{bad}, otherwise).
	Note that for any good $m$-set $R$, \begin{align*}
	\delta_{k-1}(H[R \cup \{x,y\}]) > (1/2 + 3 \gamma / 5)m \ge (1/2 + \gamma/2)(m+2),
	\end{align*}
	thus it is enough to prove that there are at most $\binom{n-2}{m}/2$ bad $m$-sets.
	Note that $R$ is bad if and only if there exists a $(k-1)$-set $T \subseteq R \cup \{x,y\}$ such that $R$ is bad for $T$.
	Therefore, the number of bad sets is at most \begin{align*}
	\binom{m+2}{k-1} \binom{n-2}{m} e^{-\gamma^2 m / 675} \leq \frac{1}{2} \binom{n-2}{m},
	\end{align*} where the inequality follows from the choice of $m$.
\end{proof}

\begin{lemma}\label{lemma:theabsorbinglemma}
	Let $k \ge 3$ and $s \ge 5k^2$.
	Let $1/n \ll \alpha \ll \mu \ll \gamma, 1/s$.
	Let $\h$ be a $k$-graph on $n$ vertices with $\delta_{k-1}(\h) \ge (1/2 + \gamma) n$.
	Then, there exists $U \subseteq V(\h)$ of size $|U| \leq \mu n$ with $|U| \equiv 0 \bmod s$ such that there exists a perfect $C_s^k$-tiling in $H[U \cup W]$ for all $W \subseteq V(\h) \setminus U$ of size $|W| \leq \alpha n$ with $|W| \equiv 0 \bmod s$.
\end{lemma}

\begin{proof} 
	Let $\mu \ll \eta \ll 1/m \ll \gamma, 1/s$.
	Let $x, y$ be distinct vertices in~$V(\h)$.
	By Lemma~\ref{lemma:hypergeometric}, at least $\binom{n-2}{m}/2$ of the $m$-sets $R \subseteq V(\h) \setminus \{ x, y \}$ are such that $\delta_{k-1}( \h[ R \cup \{ x, y \} ] ) \ge (1/2 + \gamma / 2) (m+2)$.
	By Lemma~\ref{lemma:closednessonecopy}, each one of these subgraphs contains a set $S \subseteq R$ of size $s-1$ such that $\h[S \cup \{ x \}]$ and $\h[S \cup \{ y \}]$ have spanning copies of~$C_s^k$.
	Then the number of these sets $S$ in~$H$ is at least \[ \frac{\frac{1}{2} \binom{n-2}{m} }{\binom{n - 2 - (s-1)}{m - (s-1)}} = \frac{\binom{n-2}{s-1}}{2 \binom{m}{s-1}} \ge \eta n^{s-1}. \]
	Then the result follows from Lemma~\ref{lemma:absorbingset}.
\end{proof}

\section{Tiling thresholds for tight cycles} \label{section:tilingthresholds}

Now we prove Theorem~\ref{theorem:tilingthreshold} under the assumption that the following `almost perfect $C_s^k$-tiling lemma' holds.

\begin{lemma} \label{lemma:almostcstiling}
	Let $1/n \ll \alpha, \gamma, 1/s$, $k \ge 3$ and $s \ge 5k^2$ such that $s \not\equiv 0 \bmod k$.
	Let $\h$ be a $k$-graph on $n$ vertices with $\delta_{k-1}(\h) \ge (1/2 + 1/(2s) + \gamma)n$.
	Then $\h$ has a $C_s^k$-tiling covering at least $(1 - \alpha) n$ vertices.
\end{lemma}

Assuming Lemma~\ref{lemma:almostcstiling} is true, we use it to prove Theorem~\ref{theorem:tilingthreshold}.

\begin{proof}[Proof of Theorem~\ref{theorem:tilingthreshold}]
Choose $1/n \ll \alpha \ll \mu \ll \gamma, 1/k, 1/s$.
By Lemma~\ref{lemma:theabsorbinglemma}, there exists $U \subseteq V(\h)$ of size $ |U| \leq \mu n$ with $|U| \equiv 0 \bmod s$ such that there exists a perfect $C_s^k$-tiling in $\h[U \cup W]$ for all $W \subseteq V(\h) \setminus U$ of size $|W| \leq \alpha n$ with $|W| \equiv 0 \bmod s$.

Define $\h' = \h \setminus U$.
Then $\delta_{k-1}(\h') \ge \delta_{k-1}(\h) - |U| \ge (1/2 + 1/(2s) + \gamma/2)|V(\h')|$.
An application of Lemma~\ref{lemma:almostcstiling} (with $\gamma/2, |V(H')|$ playing the roles of $\gamma, n$, respectively, and noting the hierarchy of constants in both lemmas are consistent) implies that there exists a $C_s^k$-tiling~$\T'$ in~$\h'$ covering at least $(1 - \alpha)|V(\h')|$ vertices.
Let $W$ be the set of uncovered vertices by~$\T'$ in~$\h'$.
Then $|W| \leq \alpha n$ and $|W| \equiv 0 \bmod s$.
By the absorbing property of $U$, there exists a perfect $C_s^k$-tiling~$\T''$ in~$\h[U \cup W]$.
Then $\T' \cup \T''$ is a perfect $C_s^k$-tiling in~$\h$.
\end{proof}

The rest of the paper will be devoted to the proof of Lemma~\ref{lemma:almostcstiling}.

\section{Hypergraph regularity and regular slice lemma} \label{section:regularity}

To prove Lemma~\ref{lemma:almostcstiling} we will use the hypergraph regularity lemma, which requires the following definitions. 

\subsection{Regular complexes}
	Let $\pp$ be a partition of $V$ into vertex classes $V_1, \dotsc, V_s$. A subset $S \subseteq V$ is \emph{$\pp$-partite} if $|S \cap V_i| \leq 1$ for all $1 \leq i \leq s$.
	A hypergraph is \emph{$\pp$-partite} if all of its edges are $\pp$-partite, and it is \emph{$s$-partite} if it is $\pp$-partite for some partition $\pp$ with $|\pp| = s$.

	A hypergraph $\h$ is a \emph{complex} if whenever $e\in E(\h)$ and $e'$ is a non-empty subset of $e$ we have that $e'\in E(\h)$.
	All the complexes considered in this paper have the property that all vertices are contained in an edge.
	For a positive integer $k$, a complex $\h$ is a \emph{$k$-complex} if all the edges of $\h$ consist of at most $k$ vertices.
	The edges of size $i$ are called $i$-edges of~$\h$.
	Given a $k$-complex $\h$, for all $1 \leq i \leq k$ we denote by $\h_i$ the underlying $i$-graph of~$\h$: the vertices of $\h_i$ are those of $\h$ and the edges of $\h_i$ are the $i$-edges of~$\h$.
	Given $s\ge k$, a \emph{$(k,s)$-complex} $\h$ is an $s$-partite $k$-complex.

	Let $\h$ be a $\pp$-partite $k$-complex.
	For $i \leq k$ and $X \in \binom{\pp}{i}$, we write $\h_X$ for the subgraph of $\h_i$ induced by $\bigcup X$. % whose vertex classes are the members of~$X$.
	Note that $\h_X$ is an $(i,i)$-graph.
	In a similar manner we write $\h_{X^{<}}$ for the hypergraph on the vertex set $\bigcup X$, whose edge set is $\bigcup_{X' \subsetneq X} \h_{X'}$.
	Note that if $\h$ is a $k$-complex and $X$ is a $k$-set, then $\h_{X^<}$ is a $(k-1, k)$-complex.

	Given $i\ge 2$, consider an $(i,i)$-graph $\h_{i}$ and an $(i-1,i)$-graph $\h_{i-1}$ on the same vertex set, which are $i$-partite with respect to the same partition~$\pp$.
	We write $\K_i(\h_{i-1})$ for the family of all $\pp$-partite $i$-sets that form a copy of the complete $(i-1)$-graph $K_i^{i - 1}$ in~$\h_{i-1}$.
	We define the \emph{density of $\h_{i}$ with respect to $\h_{i-1}$} to be
	\[
		d(\h_{i}|\h_{i-1})=\frac{|\K_i(\h_{i-1})\cap E(\h_{i})|}{|\K_i(\h_{i-1})|} \quad \text{if} \quad |\K_i(\h_{i-1})|>0,
	\]
	and $d(\h_{i}|\h_{i-1})=0$ otherwise.
	More generally, if ${\bf Q}=(Q_1, \dotsc, Q_r)$ is a collection of $r$ subhypergraphs of $\h_{i-1}$, we define $\K_i({\bf Q}):=\bigcup_{j=1}^r \K_i(Q_j)$ and
	\[
		d(\h_{i}|{\bf Q})=\frac{|\K_i({\bf Q})\cap E(\h_{i})|}{|\K_i({\bf Q})|} \quad \text{if} \quad |\K_i({\bf Q})|>0,
	\]
	and $d(\h_{i}|{\bf Q})=0$ otherwise.

	We say that $\h_{i}$ is \emph{$(d_i,\eps,r)$-regular with respect to $\h_{i-1}$} if for all $r$-tuples~${\bf Q}$ with $|\K_i({\bf Q})|>\eps |\K_i(\h_{i-1})|$ we have $d(\h_{i}|{\bf Q}) = d_i \pm \eps$.
	Instead of $(d_i, \eps, 1)$-regularity we simply refer to \emph{$(d_i, \eps)$-regularity}; we also say simply that $\h_i$ is $(\eps, r)$-regular with respect to $\h_{i-1}$ to mean that there exists some $d_i$ for which $\h_{i}$ is $(d_i, \eps, r)$-regular with respect to~$\h_{i-1}$.
	Given an $i$-graph $G$ whose vertex set contains that of $\h_{i-1}$, we say that~$G$ is \emph{$(d_i, \eps, r)$-regular with respect to $\h_{i-1}$} if the $i$-partite subgraph of $G$ induced by the vertex classes of $\h_{i-1}$ is $(d_i, \eps, r)$-regular with respect to~$\h_{i-1}$.
	
	Given $3 \leq k \leq s$ and a $(k,s)$-complex $\h$ with vertex partition $\pp$, we say that \emph{$\h$ is $(d_k, d_{k-1}, \dotsc, d_2,\eps_k,\eps,r)$-regular} if the following conditions hold:
	\begin{enumerate}[label={\rm (\roman*)}]
		\item For all $2 \leq i \leq k-1$ and $A \in \binom{\pp}{i}$, $\h_A$ is $(d_i, \eps)$-regular with respect to $(\h_{A^<})_{i-1}$, and % For all $2 \leq i \leq k-1$ and for all $A \in \binom{\pp}{i}$, $\h_i[\bigcup A]$ is $(d_i, \eps)$-regular with respect to $\h_{i-1}[\bigcup A]$, and
		\item for all $A \in \binom{\pp}{k}$, the induced subgraph $\h_A$ is $(d_k, \eps_k, r)$-regular with respect to $(\h_{A^<})_{i-1}$.
	\end{enumerate} Sometimes we denote $(d_k, \dotsc, d_2)$ by $\bfd$ and write $(\bfd, \eps_k, \eps, r)$-regular to mean $(d_k, \dotsc, d_2, \eps_k, \eps, r)$-regular.

We will need the following ``regular restriction lemma'' which states that the restriction of regular complexes to a sufficiently large set of vertices in each vertex class is still regular, with somewhat degraded regularity properties.

\begin{lemma}[Regular restriction lemma
	 {\cite[Lemma~24]{AllenBottcherCooleyMycroft2017}}] \label{lemma:regularrestrictionlemma}
	Let $k, m \in \mathbb{N}$ and $\beta, \eps, \eps_k, d_2, \dotsc, d_k$ be such that \[ \frac{1}{m} \ll \eps \ll \eps_k, d_2, \dotsc, d_{k-1} \qquad \text{and} \qquad \eps_k \ll \beta, \frac{1}{k}. \]
	Let $r, s \in \mathbb{N}$ and $d_k > 0$.
	Set $\bfd = (d_k, \dotsc, d_2)$.
	Let $G$ be a $(\bfd, \eps_k, \eps, r)$-regular $(k,s)$-complex with vertex classes $V_1, \dotsc, V_s$ each of size~$m$.
	Let $V'_i \subseteq V_i$ with $|V'_i| \ge \beta m$ for all $1 \leq i \leq s$. Then the induced subcomplex $G[V'_1 \cup \dotsb \cup V'_s]$ is $(\bfd, \sqrt{\eps_k}, \sqrt{\eps}, r)$-regular.
\end{lemma}

\subsection{Statement of the regular slice lemma}

In this section we state the version of the regularity lemma (Theorem~\ref{theorem:regularslices}) due to Allen, B\"ottcher, Cooley and Mycroft~\cite{AllenBottcherCooleyMycroft2017}, which they call the \emph{regular slice lemma}.
A similar lemma was previously applied by Haxell, \L{}uczak, Peng, R\"odl, Ruci\'{n}ski and Skokan in the case of $3$-graphs \cite{HaxellEtAl2009}.
This lemma says that all $k$-graphs $G$ admit a regular slice $\J$, which is a regular multipartite $(k-1)$-complex whose vertex classes have equal size such that $G$ is regular with respect to~$\J$.

Let $t_0, t_1 \in \mathbb{N}$ and $\eps > 0$. We say that a $(k-1)$-complex $\J$ is \emph{$(t_0, t_1, \eps)$-equitable} if it has the following two properties: \begin{enumerate}[label={\rm (\roman*)}]
	\item There exists a partition $\pp$ of $V(\J)$ into $t$ parts of equal size, for some $t_0 \leq t \leq t_1$, such that $\J$ is $\pp$-partite.
	We refer to $\pp$ as the \emph{ground partition} of $\J$, and to the parts of $\pp$ as the \emph{clusters} of~$\J$.
	\item There exists a \emph{density vector} $\bfd = (d_{k-1}, \dotsc, d_2)$ such that, for all $2 \leq i \leq k-1$, we have $d_i \ge 1/t_1$ and $1/d_i \in \mathbb{N}$, and the $(k-1)$-complex $\J$ is $(\bfd, \eps, \eps, 1)$-regular.
\end{enumerate}
Let $X \in \binom{\pp}{k}$.
We write $\hat{\J}_X$ for the $(k-1, k)$-graph~$(\J_{X^<})_{k-1}$.
A $k$-graph~$G$ on $V(\J)$ is \emph{$(\eps_k, r)$-regular with respect to $\hat{\J}_X$} if there exists some $d$ such that~$G$ is $(d, \eps_k, r)$-regular with respect to~$\hat{\J}_X$.
We also write $d^\ast_{\J, G}(X)$ for the density of $G$ with respect to~$\hat{\J}_X$, or simply $d^\ast(X)$ if $\J$ and $G$ are clear from the context.

\begin{definition}[Regular slice]
	Given $\eps, \eps_k > 0$, $r, t_0, t_1 \in \mathbb{N}$, a $k$-graph~$G$ and a $(k-1)$-complex $\J$ on $V(G)$, we call $\J$ a \emph{$(t_0, t_1, \eps, \eps_k, r)$-regular slice for $G$} if $\J$ is $(t_0, t_1, \eps)$-equitable and $G$ is $(\eps_k, r)$-regular with respect to all but at most $\eps_k \binom{t}{k}$ of the $k$-sets of clusters of $\J$, where $t$ is the number of clusters of~$\J$.
\end{definition}

Given a regular slice $\J$ for a $k$-graph $G$, we keep track of the relative densities $d^\ast(X)$ for $k$-sets $X$ of clusters of $\J$, which is done via a weighted $k$-graph.

\begin{definition}
	Given a $k$-graph $G$ and a $(t_0, t_1, \eps)$-equitable $(k-1)$-complex~$\J$ on $V(G)$, we let $R_{\J}(G)$ be the complete weighted $k$-graph whose vertices are the clusters of $\J$, and where each edge $X$ is given weight~$d^\ast(X)$. When $\J$ is clear from the context we write $R(G)$ instead of~$R_{\J}(G)$.
\end{definition}

The regular slice lemma (Theorem~\ref{theorem:regularslices}) guarantees the existence of a regular slice~$\J$ with respect to which $R(G)$ resembles $G$ in various senses.
In particular, $R(G)$ inherits the codegree condition of $G$ in the following sense.

Let $G$ be a $k$-graph on $n$ vertices.
Given a set $S \in \binom{V(G)}{k - 1}$, recall that $\deg_G(S)$ is the number of edges of $G$ which contain $S$.
The \emph{relative degree $\overline{\deg}(S; G)$ of $S$ with respect to $G$} is defined to be \[ \overline{\deg}(S; G) = \frac{\deg_G(S)}{n - k +1}. \]
Thus, $\overline{\deg}(S; G)$ is the proportion of $k$-sets of vertices in $G$ extending~$S$ which are in fact edges of~$G$.
To extend this definition to weighted $k$-graphs $G$ with weight function $d^\ast$, we define \[ \overline{\deg}(S; G) = \frac{\sum_{e \in E(G): S \subseteq e} d^{\ast}(e)}{n - k +1}. \]
Finally, for a collection $\mathcal{S}$ of $(k-1)$-sets in~$V(G)$, the \emph{mean relative degree $\overline{\deg}(\mathcal{S}; G)$ of $\mathcal{S}$ in $G$} is defined to be the mean of $\overline{\deg}(S; G)$ over all sets $S \in \mathcal{S}$.

We will need an additional property of regular slices.
Suppose $G$ is a $k$-graph, $\mathcal{S}$ is a $(k-1)$-graph on the same vertex set, and $\J$ is a regular slice for $G$ on $t$ clusters.
We say $\J$ is \emph{$(\eta, \mathcal{S})$-avoiding} if for all but at most $\eta \binom{t}{k-1}$ of the $(k-1)$-sets $Y$ of clusters of $\J$, it holds that $|\J_Y \cap \mathcal{S}| \leq \eta | \J_Y|$.

We can now state the version of the regular slice lemma that we will use.

\begin{theorem}[Regular slice lemma~{\cite[Lemma 6]{AllenBottcherCooleyMycroft2017}}] \label{theorem:regularslices}
Let $k \in \mathbb{N}$ with $k \ge 3$.
For all $t_0 \in \mathbb{N}$, $\eps_k > 0$ and all functions $r: \mathbb{N} \rightarrow \mathbb{N}$ and $\eps: \mathbb{N} \rightarrow (0, 1]$, there exist $t_1, n_1 \in \mathbb{N}$ such that the following holds for all $n \ge n_1$ which are divisible by~$t_1!$.
Let $G$ be a $k$-graph on $n$ vertices,
and let $\mathcal{S}$ be a $(k-1)$-graph on the same vertex set with $|E(\mathcal{S})| \leq \theta \binom{n}{k-1}$.
Then there exists a $(t_0, t_1, \eps(t_1), \eps_k, r(t_1))$-regular slice $\J$ for $G$ such that, for all $(k-1)$-sets $Y$ of clusters of $\J$, we have $\overline{\deg}(Y; R(G)) = \overline{\deg}(\J_Y; G) \pm \eps_k$,
and furthermore $\J$ is $(3 \sqrt{\theta}, \mathcal{S})$-avoiding.
\end{theorem}

We remark that the original statement of \cite[Lemma 6]{AllenBottcherCooleyMycroft2017} % is different.
%The original lemma includes many other properties which are satisfied by the regular slice $\J$, but we only state the properties we need. On the other hand, the original statement
did not include the ``avoiding'' property with respect to a fixed $(k-1)$-graph $\mathcal{S}$.
This, however, can be obtained easily from their proof.
We sketch this in Appendix~\ref{appendix:avoiding}.

\subsection{The $d$-reduced $k$-graph and strong density} \label{subsection:reducedgraph}

Once we have a regular slice $\J$ for a $k$-graph $G$, we would like to work within $k$-tuples of clusters with respect to which $G$ is both regular and dense. To keep track of those tuples, we introduce the following definition.

\begin{definition}[The $d$-reduced $k$-graph]
	Let $G$ be a $k$-graph and $\J$ be a $(t_0, t_1, \eps, \eps_k, r)$-regular slice for~$G$.
	Then for $d > 0$ we define the \emph{$d$-reduced $k$-graph $R_d(G)$ of $G$} to be the $k$-graph whose vertices are the clusters of $\J$ and whose edges are all $k$-sets of clusters $X$ of $\J$ such that $G$ is $(\eps_k, r)$-regular with respect to $X$ and $d^\ast(X) \ge d$.
	Note that $R_d(G)$ depends on the choice of $\J$ but this will always be clear from the context.
\end{definition}

The next lemma states that for regular slices $\J$ as in Theorem~\ref{theorem:regularslices}, the codegree conditions are also preserved by~$R_d(G)$.

\begin{lemma}[{\cite[Lemma 8]{AllenBottcherCooleyMycroft2017}}] \label{lemma:dreducedkeepsdegrees}
	Let $k, r, t_0, t \in \mathbb{N}$ and $\eps, \eps_k > 0$.
	Let $G$ be a $k$-graph and let $\J$ be a $(t_0, t_1, \eps, \eps_k, r)$-regular slice for~$G$.
	Then for all $(k-1)$-sets $Y$ of clusters of $\J$, we have \[ \overline{\deg}(Y; R_d(G)) \ge \overline{\deg}(Y; R(G)) - d - \zeta(Y), \] where $\zeta(Y)$ is defined to be the proportion of $k$-sets $Z$ of clusters with $Y \subseteq Z$ that are not $(\eps_k, r)$-regular with respect to~$G$.
\end{lemma}

For $0 \leq \mu, \theta \leq 1$, we say that a $k$-graph $\h$ on $n$ vertices is \emph{$(\mu, \theta)$-dense} if there exists $\mathcal{S} \subseteq \binom{V(\h)}{k - 1}$ of size at most $\theta \binom{n}{k-1}$ such that, for all $S \in \binom{V(\h)}{k - 1} \setminus \mathcal{S}$, we have $\deg_{\h}(S) \ge \mu(n - k + 1)$. In particular, if $\h$ has $\delta_{k-1}(\h) \ge \mu n$, then it is $(\mu, 0)$-dense.

By using Lemma~\ref{lemma:dreducedkeepsdegrees}, we show that  $R_d(G)$ `inherits' the property of being  $(\mu, \theta)$-dense.

\begin{lemma} \label{lemma:inheritanceofdensity}
	Let $1/n \ll 1/t_1 \leq 1/t_0 \ll 1/k$ and $\mu, \theta, d, \eps, \eps_k > 0$.
	Suppose that $G$ is a $k$-graph on $n$ vertices, that $G$ is $(\mu, \theta)$-dense and let $\mathcal{S}$ be the $(k-1)$-graph on $V(G)$ whose edges are precisely $\{ S \in \binom{V(G)}{k-1} : \deg_G(S) < \mu (n-k+1) \}$.
	Let $\J$ be a $(t_0, t_1, \eps, \eps_k, r)$-regular slice for $G$ such that for all $(k-1)$-sets $Y$ of clusters of $\J$, we have $\overline{\deg}(Y; R(G)) = \overline{\deg}(\J_Y; G) \pm \eps_k$, and furthermore $\J$ is $(3 \sqrt{\theta}, \mathcal{S})$-avoiding.
	Then $R_d(G)$ is $((1 - \sqrt{\theta})\mu - d - \eps_k - \sqrt{\eps_k}, 3\sqrt{\theta} + 3\sqrt{\eps_k} )$-dense.
\end{lemma}

\begin{proof}
	Let $\pp$ be the ground partition of $\J$ and $t = | \pp |$.
	Let $m = n/t$.
	Clearly $|V| = m$ for all $V \in \pp$.
%	Let $\mathcal{S}$ be the set of all $S \subseteq \binom{V(G)}{k-1}$ such that  $\deg_{G}(S) < \mu(n - k +1)$.
%	Then $|\mathcal{S}| \leq \theta \binom{n}{k-1}$.	
	Let $\Y_1$ be the set of all $Y \in \binom{\pp}{k - 1}$ such that $|\J_Y \cap \mathcal{S}| \ge 3 \sqrt{\theta} |\J_Y|$.
	Since $\J$ is $(3 \sqrt{\theta}, \mathcal{S})$-avoiding, $|\Y_1| \leq 3 \sqrt{\theta} \binom{t}{k-1}$.
%	Since the $\J_Y$ are disjoint, it follows that $\sqrt{\theta} m^{k-1} |\Y_1| \leq |\mathcal{S}|$, implying that $|\Y_1| \leq \sqrt{\theta} \binom{n}{k-1} / m^{k-1} \leq 2 \sqrt{\theta} \binom{t}{k-1}$, where the last inequality holds thanks to $1/n \ll 1/t_1 \leq 1/t$.
	
	For all $Y \in \binom{\pp}{k - 1}$, let $\zeta(Y)$ be defined as in Lemma~\ref{lemma:dreducedkeepsdegrees}. % the proportion of $Z \in \binom{\pp}{k}$ with $Y \subseteq Z$ that are not $(\eps_k, r)$-regular with respect to~$G$.
	Let $\Y_2$ be the set of all $Y \in \binom{\pp}{k - 1}$ with $\zeta(Y) > \sqrt{\eps_k}$.
	Since $G$ is $(\eps_k, r)$-regular with respect to all but at most $\eps_k \binom{t}{k}$ of the $k$-sets of clusters of $\pp$, it follows that $|\Y_2| \sqrt{\eps_k} (t - k + 1) / k \leq \eps_k \binom{t}{k}$, namely, $|\Y_2| \leq \sqrt{\eps_k} \binom{t}{k-1}$.%, since $1/t \leq 1/t_0 \ll 1/k$.
	
	Then it follows that $|\Y_1 \cup \Y_2| \leq 3 (\sqrt{\theta} + \sqrt{\eps_k}) \binom{t}{k-1}$.
	We will show that all $Y \in \binom{\pp}{k - 1} \setminus (\Y_1 \cup \Y_2)$ will have large codegree in $R_d(G)$, thus proving the lemma.
	
	Consider any $Y \in \binom{\pp}{k - 1} \setminus (\Y_1 \cup \Y_2)$.
	Since $Y \notin \Y_2$, $\zeta(Y) \leq \sqrt{\eps_k}$.
	By Lemma~\ref{lemma:dreducedkeepsdegrees}, we have \begin{align*}
	\overline{\deg}(Y; R_d(G))
	& \ge \overline{\deg}(Y; R(G)) - d - \zeta(Y) \\
	& \ge \overline{\deg}(Y; R(G)) - d - \sqrt{\eps_k} \\
	& \ge \overline{\deg}(\J_Y; G) - \eps_k - d - \sqrt{\eps_k}.\end{align*}
	So it suffices to show that $\overline{\deg}(\J_Y; G) \ge (1 - 3 \sqrt{\theta}) \mu$.	
	Recall that $\overline{\deg}(\J_Y; G)$ is the mean of $\overline{\deg}(S; G)$ over all $S \in \J_Y$.
	Since $Y \notin \Y_1$, $|\J_Y \cap \mathcal{S}| \leq \sqrt{\theta} |\J_Y|$.
	By definition, for all $S \in \J_Y \setminus \mathcal{S}$, $\deg_G(S) \ge \mu(n - k + 1)$.
	Thus $\overline{\deg}(\J_Y; G) \ge (1 - \sqrt{\theta}) \mu$, as required.
\end{proof}

For $0 \leq \mu, \theta \leq 1$, a $k$-graph $\h$ on $n$ vertices is \emph{strongly $(\mu, \theta)$-dense} if it is $(\mu, \theta)$-dense and, for all edges $e \in E(\h)$ and all $(k-1)$-sets $X \subseteq e$, $\deg_{\h}(X) \ge \mu (n - k + 1)$.
We prove that all $(\mu, \theta)$-dense $k$-graphs contain a strongly $(\mu', \theta')$-dense subgraph, for some degraded constants~$\mu', \theta'$.

\begin{lemma} \label{lemma:stronglydense}
	Let $n \ge 2k$ and $0 < \mu, \theta < 1$.
	Suppose that $\h$ is a $k$-graph on $n$ vertices that is $(\mu, \theta)$-dense.
	Then there exists a sub-$k$-graph $\h'$ on $V(\h)$ that is strongly $(\mu - 2^k \theta^{1/(2k - 2)}, \theta + \theta^{1/(2k - 2)})$-dense.
\end{lemma}

\begin{proof}
	Let $\mathcal{S}_1$ be the set of all $S \in \binom{V(\h)}{k-1}$ such that $\deg_\h(S) < \mu(n - k + 1)$.
	Thus, $|\mathcal{S}_1| \leq \theta \binom{n}{k - 1}$.
	Let $\beta = \theta^{1/(k-1)}$.
	Now, for all $j \in \{ k - 1, k - 2, \dotsc, 1\}$ in turn we construct $\A_j \subseteq \binom{V(\h)}{j}$ in the following way.
	Initially, let $\A_{k - 1} = \mathcal{S}_1$.
	Given $j > 1$ and $\A_j$, we define $\A_{j-1} \subseteq \binom{V(\h)}{j - 1}$ to be the set of all $X \in \binom{V(\h)}{j - 1}$ such that there exist at least $\beta (n - j + 1)$ vertices $w \in V(\h)$ with $X \cup \{ w \} \in \A_j$.
	
	\begin{claim}
		For all $1 \leq j \leq k-1$, $|\A_j| \leq \beta^{j} \binom{n}{j}$.
	\end{claim}
	
	\begin{proofclaim}
		We prove it by induction on $k - j$. When $j = k-1$ it is immediate.
		Now suppose $2 \leq j \leq k - 1$ and that $|\A_{j}| \leq \beta^{j} \binom{n}{j}$.
		By double counting the number of tuples $(X, w)$ where $X$ is a $(j-1)$-set in $\A_{j-1}$ and $X \cup \{ w \} \in \A_j$ we have $|\A_{j-1}| \beta (n - j + 1) \leq j |\A_j|$.
		By the induction hypothesis it follows that \[ |\A_{j-1}| \leq \frac{j}{\beta(n - j + 1)} |\A_j| \leq \beta^{j-1} \binom{n}{j - 1}. \qedhere \]
	\end{proofclaim}
	
	For all $1 \leq j \leq k - 1$, let $F_j$ be the set of edges $e \in E(\h)$ such that there exists $S \in \A_j$ with $S \subseteq e$, and let $F = \bigcup_{j = 1}^{k - 1} F_j$.  Define $\h' = \h - F$. We will show that it satisfies the desired properties.
	
	For each $j$-set, there are at most $\binom{n - j}{k - j}$ $k$-edges containing it.
	Thus, for all $1 \leq j \leq k - 1$, the claim above implies that
	\[ |F_j| \leq |\A_j| \binom{n - j}{k - j} \leq \beta^j \binom{n}{j} \binom{n - j}{k - j} = \beta^j \binom{k}{j} \binom{n}{k}. \]
	Therefore
	\[ | F | \leq \sum_{j=1}^{k - 1} |F_j| \leq \binom{n}{k} \sum_{j=1}^{k - 1} \binom{k}{j} \beta^{j}  \leq 2^k \beta  \binom{n}{k}. \]
	Let $\mathcal{S}_2$ be the set of all $ S \in \binom{V(\h)}{k - 1}$ contained in more than $2^k \sqrt{\beta} (n - k + 1)$ edges of~$F$.
	It follows that $|\mathcal{S}_2| \leq \sqrt{\beta} \binom{n}{k - 1}$.
	This implies that $|\mathcal{S}_1 \cup \mathcal{S}_2| \leq (\theta + \sqrt{\beta}) \binom{n}{k - 1} = (\theta + \theta^{1/(2k - 2)}) \binom{n}{k - 1}$.
	Now consider an arbitrary $S \in \binom{V(\h)}{k - 1} \setminus (\mathcal{S}_1 \cup \mathcal{S}_2)$. As $S \notin \mathcal{S}_1$, it follows that $\deg_{\h}(S) \ge \mu(n - k + 1)$. As $S \notin \mathcal{S}_2$, it follows that \begin{align*}
	\deg_{\h'}(S)
	& \ge \deg_{\h}(S) - 2^k \sqrt{\beta} (n - k + 1) \ge (\mu - 2^k \theta^{1/(2k - 2)}) (n - k + 1).
	\end{align*} Therefore, $\h'$ is $(\mu - 2^k \theta^{1/(2k - 2)}, \theta + \theta^{1/(2k - 2)})$-dense.
	
	Let $e \in E(\h')$ and let $X \in \binom{e}{k-1}$.
	It is enough to prove that $X \notin \mathcal{S}_1 \cup \mathcal{S}_2$.
	As $e \notin F_{k-1}$, it follows that $X \notin \A_{k-1} = \mathcal{S}_1$.
	So it is enough to prove that $X \notin \mathcal{S}_2$.
	Suppose the contrary, that $X \in \mathcal{S}_2$.
	Then $X$ is contained in more than $2^k \sqrt{\beta} (n - k + 1)$ edges $e' \in E(F)$.
	Let $W = N_{F}(X)$.
	For all $w \in W$, fix a set $A_w \in \bigcup_{j=1}^{k-1}{\A_j}$ such that $A_w \subseteq X \cup \{ w \}$ and let $T_w = X \cap A_w$.
	If $A_w \subseteq X$ then $A_w \subseteq e \in E(\h')$, a contradiction.
	Hence $w \in A_w$ for all $w \in W$, and therefore $|T_w| = |A_w| - 1 \leq k-2 < |X|$ for all $w \in W$.
	We deduce $T_w \neq X$ for all $w \in W$.	
	By the pigeonhole principle, there exists $T \subsetneq X$ and $W_T \subseteq W$ such that for all $w \in W_T$, $T_w = T$ and $|W_T| \ge |W|/(2^{k-1}) \ge 2 \sqrt{\beta} (n - k + 1) > \sqrt{\beta} n$.
	
	Suppose $|T| = t \ge 1$. Then for all $w \in W_T$, $T \cup \{ w \} = A_w \in \A_{t + 1}$, so there are at least $\sqrt{\beta} n \ge \beta (n - t)$ vertices $w \in V(\h)$ such that $T \cup \{ w \} \in \A_{t+1}$.
	Therefore, $T \in \A_t$ and $T \subseteq X \subseteq e$, which is a contradiction because $e \notin F_t$.
	Hence, we may assume that $T = \emptyset$.
	Then for all $w \in W_T$, $\{ w \} \in \A_1$.
	And so $|\A_1| \ge |W_T| > \sqrt{\beta} n$, contradicting the claim.
\end{proof}

\subsection{The embedding lemma}

We will need a version of ``embedding lemma'' which gives sufficient conditions to find a copy of a $(k,s)$-graph $\h$ in a regular $(k, s)$-complex~$G$.

Suppose that $G$ is a $(k,s)$-graph with vertex classes $V_1, \dotsc, V_s$, which all have size~$m$.
Suppose also that $\h$ is a $(k,s)$-graph with vertex classes $X_1, \dotsc, X_s$ of size at most~$m$.
We say that a copy of $\h$ in $G$ is \emph{partition-respecting} if for all $1 \leq i \leq s$, the vertices corresponding to those in $X_i$ lie within~$V_i$.

Given a $k$-graph~$G$ and a $(k-1)$-graph~$J$ on the same vertex set, we say that \emph{$G$ is supported on $J$} if for all $e \in E(G)$ and all $f \in \binom{e}{k - 1}$, $f \in E(J)$.

We state the following lemma which can be easily deduced from a lemma stated by Cooley, Fountoulakis, K\"uhn and Osthus~\cite{CooleyFountoulakisKuehnEtAl2009}.

\begin{lemma}[Embedding lemma {\cite[Theorem 2]{CooleyFountoulakisKuehnEtAl2009}}] \label{lemma:embeddinglemma}
	Let $k, s, r, t, m_0 \in \mathbb{N}$ and let $d_2, \dotsc, d_{k-1}, d, \eps, \eps_k > 0$ be such that $1/d_i \in \mathbb{N}$ for all $2 \leq i \leq k-1$, and \[ \frac{1}{m_0} \ll \frac{1}{r}, \eps \ll \eps_k, d_2, \dotsc, d_{k-1} \qquad \text{and} \qquad \eps_k \ll d, \frac{1}{t}, \frac{1}{s}. \]
	Then the following holds for all $m \ge m_0$.
	Let $\h$ be a $(k,s)$-graph on $t$ vertices with vertex classes $X_1, \dotsc, X_s$.
	Let $\J$ be a $(d_{k-1}, \dotsc, d_2, \eps, \eps, 1)$-regular $(k-1, s)$-complex with vertex classes $V_1, \dotsc, V_s$ all of size~$m$.
	Let $G$ be a $k$-graph on $\bigcup_{1 \leq i \leq s} V_i$ which is supported on $\J_{k-1}$ such that for all $e \in E(\h)$ intersecting the vertex classes $\{ X_{i_j} : 1 \leq j \leq k \}$, the $k$-graph $G$ is $(d_e, \eps_k, r)$-regular with respect to the $k$-set of clusters $\{ V_{i_j} : 1 \leq j \leq k \}$, for some $d_e \ge d$ depending on~$e$.
	Then there exists a partition-respecting copy of $\h$ in~$G$.
\end{lemma}

The differences between Lemma~\ref{lemma:embeddinglemma} and \cite[Theorem 2]{CooleyFountoulakisKuehnEtAl2009} are discussed in Appendix~\ref{appendix:embedding}.

\section{Almost perfect $C_s^k$-tilings} \label{section:tilings}

The aim of this section is to prove Lemma~\ref{lemma:almostcstiling}, that is, finding an almost perfect $C_s^k$-tiling.
Throughout this section, we fix $k \ge 3$ and $s \ge 5k^2$ with $s \not\equiv 0 \bmod k$.
Let $G_s, W_{G_s}, a_{s,1}, \dotsc, a_{s,k}, \ell, F_s$ be given by Proposition~\ref{corollary:fs}.
Recall that $F_s$ contains a spanning $C_s^k$.
Therefore, an $F_s$-tiling in~$\h$ implies the existence of a $C_s^k$-tiling in $H$ of the same size.

Here we summarise some useful inequalities that will be used throughout the section.
Let $M_s = \max_i a_{s,i}$ and $m_s = \min_i a_{s,i}$.
We have \begin{align}
\ell + \sum_{i=1}^k a_{s,i} = s, \quad M_s \leq m_s+1, \quad \text{and} \quad 1 \leq \ell \leq k - 1. \label{eq:lacotabonita}
\end{align}
From this, we can easily deduce \begin{align}
m_s + 1 \ge M_s \ge \frac{s - \ell}{k} \ge \frac{s - k + 1}{k}. \label{eq:lacotabonita2}
\end{align}

%The above bounds imply that $M_s \ge (s-k+1)/k$.
%Using $s \ge 5k^2$ we can deduce $kM_s/s \ge 1 - (k-1)/(5k^2)$ and $M_s \ge \lceil (5k^2 + 1)/k \rceil - 1$, and both expressions in the right hand size are minimised when $k=3$.
%In summary, \begin{align}
%m_s \leq M_s \leq m_s + 1,
%\quad \frac{k M_s}{s} \ge \frac{43}{45} \quad \text{and} \quad
%\frac{m_s}{M_s} \ge \frac{14}{15}. \label{eq:lacotabendita}
%\end{align}
Define $E_s = K^{k}(M_s)$, the complete $(k,k)$-graph with each part of size $M_s$.
Given an $\{ F_s, E_s \}$-tiling $\T$ in $\h$, let $\F_\T$ and~$\E_\T$ be the set of copies of $F_s$ and $E_s$ in $\T$, respectively.
Define \[ \phi(\T) = \frac{1}{n} \left( n - s \left( |\F_\T| + \frac{3}{5} |\E_\T| \right) \right).  \]
Note that if $\E_\T = \emptyset$, then $\T$ is an $F_s$-tiling covering all but $\phi(\T)n$ vertices.
Let $\phi(\h)$ be the minimum of $\phi(\T)$ over all $\{ F_s, E_s\}$-tilings~$\T$ in~$\h$.
Given $n \ge k$ and $0 \leq \mu, \theta < 1$, let $\Phi(n, \mu, \theta)$ be the maximum of $\phi(\h)$ over all $(\mu, \theta)$-dense $k$-graphs $\h$ on $n$ vertices.
Note that $\phi(H)$ and $\Phi(n, \mu, \theta)$ depend on $k$ and $s$ but they will be clear from the context.

\begin{lemma} \label{lemma:integertilingdense}
	Let $k \ge 3$ and $s \ge 5k^2$ with $s \not\equiv 0 \bmod k$.
	Let $1/n, \theta \ll \alpha, \gamma, 1/k, 1/s$.
	Then $\Phi(n, 1/2 + 1/(2s) + \gamma, \theta) \leq \alpha$.
\end{lemma}

We now show that Lemma~\ref{lemma:integertilingdense} implies Lemma~\ref{lemma:almostcstiling}.

\begin{proof}[Proof of Lemma~\ref{lemma:almostcstiling}]
	Fix $\alpha, \gamma > 0$. Note that $|V(F_s)| = s$ and $|V(E_s)| = k M_s$.
	Let $\delta = 7/10$.
	Using $s \ge 5k^2$, \eqref{eq:lacotabonita2} and $k \ge 3$, we deduce $kM_s/s \ge 1 - (k-1)/(5k^2) \ge 43/45$.
	Hence 
	\begin{align}
	3 s / 5 \leq  43 s \delta/45 \leq  \delta k M_s. \label{eqn:almostcstiling}
	\end{align}
	
	Define $\alpha_1 = \alpha (1 - \delta)$ and choose some $\theta \ll \alpha, \gamma, 1/k, 1/s$.
	Since $1/n \ll \alpha, \gamma, 1/k, 1/s$ as well, Lemma~\ref{lemma:integertilingdense} (with $\alpha_1$ in place of $\alpha$) implies that $\Phi(n, 1/2 + 1/(2s) + \gamma, \theta) \leq \alpha_1$.
	
	Let $\h$ be a $k$-graph on $n$ vertices with $\delta_{k-1}(\h) \ge (1/2 + 1/(2s) + \gamma)n$.
	Then $\phi(\h) \leq \Phi(n, 1/2 + 1/(2s) + \gamma, 0) \leq \Phi(n, 1/2 + 1/(2s) + \gamma, \theta) \leq \alpha_1$.
	Let $\T$ be an $\{ F_s, E_s \}$-tiling in $\h$ with $\phi(\T) \leq \alpha_1$.
	Hence, 
	\begin{align*}
	1 - \alpha_1 \leq 1 - \phi(\T) 
	\leq \frac{s}{n} \left( | \F_\T| + \frac{3 }{5 } |\E_\T| \right) 
	\overset{\mathclap{\text{\eqref{eqn:almostcstiling}}}}{\leq}
	\frac{1}{n} \left( s |\F_\T| + \delta k M_s |\E_\T|  \right).
	\end{align*}
	As $\T$ is a tiling, we have that $s |\F_\T| + k M_s |\E_\T| \leq n$.
	Hence, $1 - \alpha_1 \leq (1 - \delta) s | \F_\T |/n + \delta$ and so \[ s |\F_\T| \ge \left( 1 - \frac{\alpha_1}{1 - \delta} \right) n = (1 - \alpha) n. \]
	Therefore $H$ contains an $\F_s$-tiling $\F_\T$ covering all but at most $\alpha n$ vertices, implying the existence of a $C_s^k$-tiling of the same size.
\end{proof}

\subsection{Weighted fractional tilings} \label{subsection:weightedfractionaltilings}

Our strategy for proving Lemma~\ref{lemma:integertilingdense} is to apply the regular slice lemma (Theorem~\ref{theorem:regularslices}).  
In the reduced $k$-graph, we find a fractional $\{F^\ast_s, E^\ast_s \}$-tiling for some simpler $k$-graphs $F^\ast_s$ and~$E^\ast_s$.
By using the regularity methods, this fractional tiling can then be lifted to an actual tiling with copies of $F_s, E_s$ in the original $k$-graph, which covers a similar proportion of vertices.

%To prove Lemma~\ref{lemma:integertilingdense},
To define the $k$-graphs $F^\ast_s$ and $E^\ast_s$, we use the notion of $G$-augmentation introduced in Subsection~\ref{subsection:auxiliarygraphs}.
Let $K$ be a $k$-edge with vertices $\{ x_1, \dotsc, x_k \}$.
Let~$G_s$ be the $2$-graph on $[k]$ given by Corollary~\ref{corollary:fs}.
Let $F^\ast_s$ be the $G_s$-augmentation of $K$ (with respect to the vertex partition $V_i := \{x_i\}$ for all $i \in [k]$). %, as defined in Section~\ref{subsection:auxiliarygraphs}.
Let $V(F^\ast_s) = \{ x_1, \dotsc, x_k \} \cup \{ y_1, \dotsc, y_{\ell} \} $, where $\ell = |E(G_s)|$.
We refer to $c(F^\ast_s) = \{ x_1, \dotsc, x_k \}$ as the set of \emph{core vertices of $F^\ast_s$} and $p(F^\ast_s) = \{ y_1, \dotsc, y_{\ell} \}$ as the set of \emph{pendant vertices of $F^\ast_s$}.
Define the function $\alpha: V(F^\ast_s) \rightarrow \mathbb{N}$ to be such that for $u \in V(F^\ast_s)$, \[ \alpha(u) = \begin{cases}
a_{s,i} & \text{if } u = x_i, \\ 1 & \text{if } u \in p(F^\ast_s).
\end{cases} \] Note that there is a natural $k$-graph homomorphism $\theta$ from $F_s$ to $F^\ast_s$ such that for all $u \in V(F^\ast_s)$, $|\theta^{-1}(u)| = \alpha(u)$.
Observe that \eqref{eq:lacotabonita2}, $s \ge 5k^2$ and $k \ge 3$ imply that $\alpha(u) = 1$ if and only if $u$ is a pendant vertex.

Let $\F^\ast_s(\h)$ be the set of copies of $F^\ast_s$ in~$\h$.
Given $v \in V(H)$ and $F^\ast_s \in \F^\ast_s(\h)$, define \[ \alpha_{F^\ast_s}(v) = \begin{cases}
\alpha(u) & \text{if $v$ corresponds to vertex $u \in V(F^\ast_s)$,} \\ 0 & \text{otherwise.}
\end{cases} \]
Given $v \in V(\h)$ and $e \in E(\h)$, define \[ \alpha_{e}(v) = \begin{cases}
M_s & \text{if $v \in e$,} \\ 0 & \text{otherwise.}
\end{cases} \]
We now define a \emph{weighted fractional $\{ F^\ast_s, E_s^\ast \}$-tiling of $\h$} to be a function $\omega^\ast : \F^\ast_s(\h) \cup E(\h) \rightarrow [0,1]$ such that, for all vertices $v \in V(\h)$,
\[ \omega^\ast(v) \defined \sum_{F^\ast_s \in \F^\ast_s(\h)} \omega^\ast(F^\ast_s) \alpha_{F^\ast_s}(v) + \sum_{e \in E(\h)} \omega^\ast(e) \alpha_e(v) \leq 1. \]
Note that if (contrary to our assumptions) $a_{s,1} = \dotsb = a_{s,k} = 1$,
then we have $\alpha_{F^\ast_s}(v) = \mathbf{1}\{ v \in V(F^\ast_s) \}$ and $\alpha_{e}(v) = \mathbf{1}\{ v \in e \}$ implying that $\omega^\ast$ is the standard fractional $\{ F_s, E_s \}$-tiling.
Note that the definition depends on $k$ and the functions $\alpha_{F^\ast_s}$ and $ \alpha_e$, but those will always be clear from the context.%If $\alpha$, $s$ and $k$ are clear from the context (which will be all the time) we will simply write \emph{weighted fractional $\{ F^\ast_s, E^\ast_s \}$-tiling}.

Define the \emph{minimum weight of $\omega^\ast$} to be \[ \omega^\ast_{\min} = \min_{\substack{J \in \F^\ast_s(\h) \cup E(\h) \\ v \in V(\h) \\ \omega^\ast(J) \alpha_J(v) \neq 0 }} \omega^\ast(J) \alpha_J(v). \]
Analogously to $\phi(\T)$, define \[ \phi( \omega^\ast ) = \frac{1}{n} \left( n - s \left( \sum_{F^\ast_s \in \F^\ast_s(\h)} \omega^\ast(F^\ast_s) + \frac{3}{5} \sum_{e \in E(\h)} \omega^\ast(e) \right) \right). \]
Given $c > 0$ and a $k$-graph $\h$, let $\phi^\ast(\h, c)$ be the minimum of $\phi(\omega^\ast)$ over all weighted fractional $\{ F^\ast_s, E^\ast_s \}$-tilings $\omega^\ast$ of $\h$ with $\omega^\ast_{\min} \ge c$.
Note that $\phi^\ast(H, c)$ also depends on $k$, $s$, $\alpha_{F^\ast_s}$ and $\alpha_e$, which will always be clear from the context.

Let $\T$ be an $\{ F_s, E_s \}$-tiling. We say that a vertex $v$ is \emph{saturated under~$\T$} if it is covered by a copy of $F_s$ and $v$ corresponds to a vertex in $W_{G_s}$ under that copy.
Let $S(\T)$ denote the set of all saturated vertices under~$\T$.
Define $U(\T)$ as the set of all uncovered vertices under~$\T$.

Analogously, given a weighted fractional $\{ F^\ast_s, E^\ast_s \}$-tiling $\omega^\ast$, we say that a vertex $v$ is \emph{saturated under $\omega^\ast$} if
\begin{align*}
\sum_{\substack{F^\ast_s \in \F_s^\ast(\h) \\ \alpha_{F^\ast_s}(v) = 1 }} \omega^\ast(F^\ast_s) \alpha_{F^\ast_s} (v) = 1,
\end{align*}
that is, $\omega^\ast(v) = 1$ and all its weight comes from copies of~$F^\ast_s$ such that~$v$ corresponds to a pendant vertex.
%where the sum runs over all copies $F^\ast_s \in \F_s^\ast(\h)$ such that $\alpha_{F^\ast_s}(v) = 1$, that is, the copies of $F^\ast_s$ where $v$ corresponds to a pendant vertex.
Let $S(\omega^\ast)$ be the set of all saturated vertices under~$\omega^\ast$.
Also, define $U(\omega^\ast)$ as the set of all vertices $v \in V(\h)$ such that $\omega^\ast(v) = 0$.

\begin{proposition}  \label{proposition:twopendantpoints}
	Let $k \ge 3$ and $s \ge 5k^2$ with $s \not\equiv 0 \bmod k$.
	Let $H$ be a $k$-graph on $n$ vertices.
	Let $\omega^\ast$ be a weighted fractional $\{ F^\ast_s, E^\ast_s \}$-tiling in~$H$.
	Then the following holds: \begin{enumerate}[label={\rm (\roman{*})}]
		\item \label{proposition:twopendantpointsi} $s \sum_{F^\ast \in \F^\ast_s} \omega^\ast (F^\ast) + k M_s  \sum_{e \in E(\h)} \omega^\ast (e) \leq n$. In particular, $\sum_{F^\ast \in \F^\ast_s} \omega^\ast (F^\ast) \leq n/s$ and  $\sum_{e \in E(\h)} \omega^\ast (e) \leq n/ (k M_s)$,
		\item \label{proposition:twopendantpointsii} $|S(\omega^\ast)| \leq \ell n / s$,
		\item \label{proposition:twopendantpointsiii} if $S' \subseteq S(\omega^\ast)$ with $|S'| > n/s$, then there exists $F^\ast \in \F^\ast_s(\h)$ with $\omega^\ast(F^\ast) > 0$ such that $|p(F^\ast) \cap S'| \ge 2$.
	\end{enumerate}
\end{proposition}

\begin{proof}
	For \ref{proposition:twopendantpointsi}, note that
	\begin{align*}
	n
	& \ge \sum_{v \in V(\h)} \sum_{F^\ast \in \F^\ast_s(H)} \omega^\ast(F^\ast) \alpha_{F^\ast}(v) + \sum_{v \in V(\h)} \sum_{e \in E(H)} \omega^\ast(e) \alpha_{e}(v) \\
	& = \sum_{F^\ast \in \F^\ast_s(H)} \omega^\ast(F^\ast) \sum_{v \in V(\h)} \alpha_{F^\ast}(v) + \sum_{e \in E(H)} \omega^\ast(e) \sum_{v \in V(\h)} \alpha_e(v)  \\
	& = s \sum_{F^\ast \in \F^\ast_s(H)} \omega^\ast(F^\ast) + k M_s \sum_{e \in E(H)} \omega^\ast(e).
	\end{align*}
	
	To prove \ref{proposition:twopendantpointsii}, recall that all of the vertices $v \in S(\omega^\ast)$ only receive weight from pendant vertices, and all copies of $F \in \F^\ast_s(\h)$ have precisely $\ell$ pendant vertices, and therefore \begin{align*}
	|S(\omega^\ast)| = \sum_{v \in S(\omega^\ast)} \sum_{F^\ast \in \F^\ast_s(\h)} \omega^\ast(F^\ast) \alpha_{F^\ast}(v) \leq \ell \sum_{F^\ast \in \F^\ast_s(\h)} \omega^\ast(F^\ast) \stackrel{\mathclap{\ref{proposition:twopendantpointsi}}}{\leq} \ell n / s. \end{align*}
	
	Finally, for \ref{proposition:twopendantpointsiii}, suppose the contrary, that, for all $F^\ast \in \F^\ast_s(\h)$ with $\omega^\ast(F^\ast) > 0$, we have $\sum_{v \in S'} \alpha_{F^\ast}(v) = |p(F^\ast) \cap S'| \leq 1$. Then \begin{align*}
	|S'|
	& = \sum_{v \in S'} \sum_{F^\ast \in \F^\ast_s(\h)} \omega^\ast(F^\ast) \alpha_{F^\ast}(v) = \sum_{F^\ast \in \F^\ast_s(\h)} \omega^\ast(F^\ast) \sum_{v \in S'} \alpha_{F^\ast}(v) \\
	& \leq \sum_{F^\ast \in \F^\ast_s(\h)} \omega^\ast(F^\ast) \leq n/s,
	\end{align*} a contradiction.
\end{proof}

Note that $F_s$ admits a natural perfect weighted fractional $F^\ast_s$-tiling, defined as follows.
Let $a = \prod_{1 \leq i \leq k} a_{s,i}$.
Let $F$ be a copy of $F_s$ and suppose that $V(F) = V_1 \cup \dotsb \cup V_k \cup W$, where $V_1, \dotsc, V_k$ forms a complete $(k,k)$-graph with $|V_i| = a_{s,i}$ for all $1 \leq i \leq k$ and $|W| = \ell$.
Note that $a \leq M^k_s$.
For all $(v_1, \dotsc, v_k) \in V_1 \times \dotsb \times V_k$, the vertices $\{ v_1, \dotsc, v_k \} \cup W$ span a copy of $F^\ast_s$, where we identify $\{ v_1, \dotsc, v_k \}$ with the core vertices of $F^\ast_s$ and $W$ with the pendant vertices of~$F^\ast_s$.
Define $\omega^\ast$ by assigning to all such copies the weight $1 / a$.
A similar method shows that $E_s$ admits a perfect weighted fractional $E^\ast_s$-tiling, by setting $\omega^\ast(e) = M^{-k}_s$ for all $e \in E_s$.

We can naturally extend these constructions to find a weighted fractional $\{ F^\ast_s, E^\ast_s \}$-tiling given an $\{F_s, E_s\}$-tiling, by repeating the above procedure over all copies of $F_s$ and $E_s$.
The following proposition collects useful properties of the obtained fractional tiling, for future reference.
All of them straightforward to check by using the construction outlined above, so we omit its proof.

\begin{proposition} \label{proposition:fractionalw0}
	Let $k \ge 3$ and $s \ge 5k^2$ with $s \not\equiv 0 \bmod k$.
	Let $\h$ be a $k$-graph and let $\T$ be an $\{ F_s, E_s \}$-tiling in~$\h$.
	Then there exists a weighted fractional $\{ F^\ast_s, E^\ast_s \}$-tiling $\omega^\ast$ such that
	\begin{enumerate}[label={\rm (\roman*)}]
		\item $\phi(\T) = \phi(\omega^\ast)$, \label{proposition:fractionalw0first}
		\item $|\F_\T| = \sum_{F^\ast \in \F^\ast_s(\h)} \omega^\ast(F^\ast)$,
		\item $ |\E_\T| = \sum_{e \in E(\h)} \omega^\ast(e)$,
		\item $S(\omega^\ast) = S(\T)$ and $U(\omega^\ast) = U(\T)$, \label{proposition:fractionalw0uncovered}
		\item for all $F^\ast \in \F_s^\ast(\h)$, $\omega^\ast(F^\ast) \in \{ 0, a^{-1} \}$, where $a = \prod_{1 \leq i \leq k} a_{s,i}$, \label{proposition:fractionalw0weightofF}
		\item for all $e \in E(\h)$, $\omega^\ast(e) \in \{ 0, M^{-k}_s \}$, moreover if $e \in E(E_s)$ for some $E_s \in \E_\T$, then $\omega^\ast(e) = M^{-k}_s$, \label{proposition:fractionalw0weightofE}
		\item $\omega^\ast_{\min} \ge M_s^{-k}$, and \label{proposition:fractionalw0minimumweight}
		\item $\omega^\ast(v) \in \{0,1\}$ for all $v \in V(\h)$. \label{proposition:fractionalw0verticesarefull}\end{enumerate} 
\end{proposition}

%\begin{proof}
%	Let $F \in \F_\T$ be a copy of $F_s$ and suppose that $V(F) = V_1 \cup \dotsb \cup V_k \cup W$, where $V_1, \dotsc, V_k$ forms a complete $(k,k)$-graph with $|V_i| = a_{s,i}$ for all $1 \leq i \leq k$ and $|W| = \ell$.
%	Note that $a \leq M^k_s$.
%	For all $(v_1, \dotsc, v_k) \in V_1 \times \dotsb \times V_k$, the vertices $\{ v_1, \dotsc, v_k \} \cup W$ span a copy of $F^\ast_s$, where we identify $\{ v_1, \dotsc, v_k \}$ with the core vertices of $F^\ast_s$ and $W$ with the pendant vertices of~$F^\ast_s$.
%	Define $\omega^\ast$ on $\F^\ast_s$ by assigning to all such copies the weight $1 / a$ and doing the same procedure with all $F \in \F$. All other copies of $F^\ast_s$ in $\h$ receive weight~$0$.
%	
%	Also, for all $e \in E(\h)$, set \[ \omega^\ast(e) = \begin{cases}
%	M^{-k}_s & \text{if $e \in E(J)$ for some $J \in \E_\T$,} \\
%	0	& \text{otherwise.}
%	\end{cases} \]
%	
%	This defines a function $\omega^\ast: \F^\ast_s(\h) \cup E(\h) \rightarrow [0,1]$ which is a weighted fractional $\{ F^\ast_s, E^\ast_s \}$-tiling of $\h$ with $\omega^\ast_{\min} \ge \min\{ M_s^{-k}, a^{-1} \} = M_s^{-k}$.
%	It is straightforward to check \ref{proposition:fractionalw0first}--\ref{proposition:fractionalw0verticesarefull}.
%\end{proof}

The next lemma assures that if $R$ is a reduced $k$-graph of $\h$, then $\phi(H)$ is roughly bounded above by $\phi^\ast(R, c)$.

\begin{lemma} \label{lemma:blowuptiling}
	Let $k \ge 3$ and $s \ge 5k^2$ with $s \not\equiv 0 \bmod k$.
	Let $c \ge \beta > 0$ and \[ 1/{n} \ll \eps, 1/r \ll \eps_k \ll 1/t_1 \leq 1/t_0 \ll \beta, c, 1/s, 1/k \] and \[ \eps_k \ll d, 1/k, 1/s. \] 
	Let $\h$ be a $k$-graph on $n$ vertices and $\J$ be a $(t_0, t_1, \eps, \eps_k, r)$-regular slice for $\h$, and $R = R_d(\h)$ be its $d$-reduced $k$-graph obtained from~$\J$.
	Then $\phi(\h) \leq \phi^\ast(R, c) + s \beta / c$.
\end{lemma}

\begin{proof}
	Let $\omega^\ast$ be a weighted fractional $\{ F^\ast_s, E^\ast_s \}$-tiling on $R$ such that $\phi(\omega^\ast) = \phi^\ast(R, c)$ and $\omega^\ast_{\min} \ge c$.
	Let $t = |V(R)|$ and let $m = n/t$, so that each cluster in $\J$ has size~$m$.
	Let $n^\ast_F$ be the number of $F^\ast_s \in \F^\ast_s(R)$ with $\omega^\ast(F^\ast_s) > 0$ and $n^\ast_E$ be the number of $E \in E(R)$ with $\omega^\ast(E) > 0$.
	Note that \[n^\ast_F + n^\ast_E \leq t/c.\]
	For all clusters $U \in V(R)$, we subdivide $U$ into disjoint sets $\{ U_J \}_{J \in \F^\ast_s(R) \cup E(R) }$ of size $|U_J| = \lfloor \omega^\ast(J) \alpha_J(U) m \rfloor$.
	
	In the next claim, we show that if $\omega^\ast(J) > 0$ for some $J \in \F^\ast_s(R) \cup E(R)$ then we can find a large $F_s$-tiling or a large $E_s$-tiling on $\bigcup_{U \in V(J)} U_{J}$.
	
	\begin{claim}
		For all $J \in \F_s^\ast(R) \cup E(R)$ with $\omega^\ast(J) > 0$, $\h \left[ \bigcup_{U \in V(J)} U_J \right]$ contains
		\begin{enumerate}[label={\rm (\roman{*})}]
			\item an $F_s$-tiling $\F_J$ with $|\F_J| \ge m (\omega^\ast(J) - \beta )$ if $J \in \F^\ast_s(R)$; or
			\item an $E_s$-tiling $\E_J$ with $|\E_J| \ge m (\omega^\ast(J) - \beta )$ if $J \in E(R)$.
		\end{enumerate}
	\end{claim}
	
	\begin{proofclaim}
		We will only consider the case when $J \in \F^\ast_s(R)$, as the case $J \in E(R)$ is proved similarly.
		
		Suppose $c(J) = \{ X_1, \dotsc, X_k \}$ and $p(J) = \{Y_1, \dotsc, Y_\ell \}$, so $V(J) = c(J) \cup p(J)$.
		We will first show that if $X'_i \subseteq X_i$ for all $1 \leq i \leq k$ and $Y'_j \subseteq Y_j$ for all $1 \leq j \leq \ell$ are such that $|X'_i| = |Y'_j| \ge \beta m$, then $\h \left[ \bigcup_{1 \leq i \leq k} X'_i \cup \bigcup_{1 \leq j \leq \ell} Y'_j \right]$ contains a copy $F$ of $F_s$ such that $|V(F) \cap X'_i| = a_{s,i}$ for all $1 \leq i \leq k$ and $|V(F) \cap Y'_j| = 1$ for all $1 \leq j \leq \ell$.
		
		Indeed, take $X'_i, Y'_j$ as above and construct the subcomplex $\h'$ obtained by restricting $\h$ along with $\J$ to the subsets $X'_i, Y'_j$ and then deleting the edges in $\h$ not supported in $k$-tuples of clusters corresponding to edges in~$E(J)$.
		Then $\h'$ is a $(k, k + \ell)$-complex.
		Since $\J$ is $(t_0, t_1, \eps)$-equitable, there exists a density vector $\bfd = (d_{k-1}, \dotsc, d_2)$ such that, for all $2 \leq i \leq k-1$, we have $d_i \ge 1/t_1$, $1/d_i \in \mathbb{N}$ and $\J$ is $(d_{k-1}, \dotsc, d_2, \eps, \eps, 1)$-regular.
		As $J \subseteq R$, all edges $e$ in $E(J) \cap E(R)$ induce $k$-tuples $X_e$ of clusters in $\h$ with $d^\ast(X_e) = d_e \ge d$ and $\h$ is $(d_e, \eps_k, r)$-regular with respect to~$X_e$.
		By Lemma~\ref{lemma:regularrestrictionlemma}, the restriction of $X_e$ to the subsets $\{ X'_1, \dotsc, X'_k, Y'_1, \dotsc, Y'_\ell \}$ is $(d_e, \sqrt{\eps_k},  \sqrt{\eps}, r)$-regular.
		Hence, by Lemma~\ref{lemma:embeddinglemma}, there exists a partition-respecting copy $F$ of $F_s$ in $\h'$, that is, $F$ satisfies $|V(F) \cap X'_i| = a_{s,i}$ for all $1 \leq i \leq s$ and $|V(F) \cap Y'_j| = 1$ for all $1 \leq j \leq \ell$, as desired.
		
		Now consider the largest $F_s$-tiling $\F_J$ in $\h \left[ \bigcup_{U \in V(J)} U_J \right]$ such that all $F \in \F_J$ satisfy $|V(F) \cap X_i| = a_{s,i}$ for all $1 \leq i \leq k$ and $|V(F) \cap Y_j| = 1$ for all $1 \leq j \leq \ell$.
		Let $V(\F_J) = \bigcup_{F \in \F_J} V(F)$.
		By the discussion above, we may assume that $\left|U_J \setminus V(\F_J) \right| < \beta m$ for some $U \in V(J)$.
		%By the discussion above, we may assume that $\left|X_i \setminus \left( \bigcup_{F \in \F_J} V(F) \right) \right| < \beta m$ for some $1 \leq i \leq k$ or $\left|Y_j \setminus \left( \bigcup_{F \in \F_J} V(F) \right) \right| < \beta m$ for some $1 \leq j \leq \ell$.
%		For $1 \leq i \leq k$ let $X_{i,J} = (X_i)_J$, and for $1 \leq j \leq \ell$ let $Y_{j,J} = (Y_j)_J$.
		A simple calculation shows that $\left|(Y_{j})_J \setminus V(\F_J) \right| < \beta m$ for all $1 \leq j \leq \ell$ and $\left|(X_{i})_J \setminus V(\F_J) \right| < a_{s,i} \beta m$ for all $1 \leq i \leq k$.
		Therefore, $\F_J$ covers at least $sm(\omega^\ast(J) - \beta)$ vertices and it follows that $|\F_J| \ge m (\omega^\ast(J) - \beta)$.
	\end{proofclaim}
	
	Now consider the $\{ F_s, E_s \}$-tiling $\T = \F_\T \cup \E_\T$ in $\h$, where $\F_\T = \bigcup_{J \in \F^\ast_s(R)} \F_J$ and $\E_\T = \bigcup_{E \in E(R)} \E_J$ as given by the claim (and we take $\F_J = \E_J = \emptyset$ whenever $\omega^\ast(J) = 0$).
	Therefore \begin{align*}
	|\F_\T| + \frac{3}{5} |\E_\T|
	& \ge \sum_{ \substack{ F^\ast_s \in \F^\ast_s(R) \\ \omega^\ast(F^\ast_s) > 0}} m(\omega^\ast(F^\ast_s) - \beta) + \frac{3}{5} \sum_{ \substack{ E \in E(R) \\ \omega^\ast(E) > 0}} m(\omega^\ast(E) - \beta ) \\
	& \ge m \left( \sum_{F^\ast_s \in \F^\ast_s(R)} \omega^\ast(F^\ast_s) + \frac{3}{5} \sum_{ E \in E(R)} \omega^\ast(E) - \beta (n^\ast_F + n^\ast_E) \right) \\
	& \ge m \left( \sum_{F^\ast_s \in \F^\ast_s(R)} \omega^\ast(F^\ast_s) + \frac{3}{5} \sum_{ E \in E(R)} \omega^\ast(E) - \frac{\beta t}{c} \right) \\
	& = m t \left( \frac{1 - \phi(\omega^\ast)}{s} - \frac{\beta}{c}  \right) = \frac{n}{s} \left(  1 - \phi(\omega^\ast) - \frac{\beta s}{c} \right).
	\end{align*}
	Thus we have $\phi(\h) \leq \phi(\T) \leq \phi(\omega^\ast) + s \beta / c = \phi^\ast(R, c) + s \beta / c$.
\end{proof}

\subsection{Proof of Lemma~\ref{lemma:integertilingdense}} \label{subsection:proofofalmostperfecttiling}

We begin with some lemmas before formally proving Lemma~\ref{lemma:integertilingdense}.

\begin{lemma} \label{lemma:casialpha}
	Let $k \ge 3$ and $s \ge 5k^2$ with $s \not\equiv 0 \bmod k$.
	Let $\mu + \gamma/3 \leq 2/3$.
	Then $\Phi(n, \mu, \theta) \leq \Phi((1 + \gamma)n, \mu + \gamma/3, \theta) + s \gamma$.
\end{lemma}

\begin{proof}
	Let $\h$ be a $k$-graph on $n$ vertices that is $(\mu, \theta)$-dense. Consider the $k$-graph $\h'$ on the vertices $V(\h) \cup A$ obtained from $\h$ by adding a set of $|A| = \gamma n$ vertices and adding all of the $k$-edges that have non-empty intersection with~$A$. Since \[ \frac{\mu + \gamma}{1 + \gamma} \ge \mu + \gamma / 3 \] as $\mu + \gamma/3 \leq 2/3$, $\h'$ is $(\mu + \gamma/3, \theta)$-dense.
	
	Let $\T'$ be an $\{ F_s, E_s \}$-tiling on $\h'$ satisfying $\phi(\T') = \phi(\h')$.
	Consider the $\{ F_s, E_s \}$-tiling $\T$ in $\h$ obtained from $\T'$ by removing all copies of $F_s$ or $E_s$ intersecting with~$A$. It follows that
	\begin{align*}
	1 - \phi(\T)
	& = \frac{s}{n} \left( |\F_\T| + \frac{3}{5} |\E_\T|  \right)
	\ge \frac{s}{n} \left( |\F_{\T'}| + \frac{3}{5} |\E_{\T'}|  \right) - s \gamma \\
	& \ge \frac{s}{(1 + \gamma) n} \left( |\F_{\T'}| + \frac{3}{5} |\E_{\T'}|  \right) - s \gamma = 1 - \phi(\T') - s \gamma.
	\end{align*}
	Hence, $\phi(\h) \leq \phi(\T) \leq \phi(\T') + s \gamma \leq \Phi((1 + \gamma)n, \mu + \gamma/3, \theta) + s \gamma$.
\end{proof}

The next lemma shows that given an $\{ F_s, E_s \}$-tiling $\T$ of a strongly $(\mu, \theta)$-dense $k$-graph $\h$ with $\phi(T)$ ``large'', we can always find a better weighted fractional $\{ F^\ast_s, E^\ast_s \}$-tiling in terms of~$\phi^\ast$.

\begin{lemma} \label{lemma:fractionalisbetter}
	Let $k \ge 3$, $s \ge 5k^2$ with $s \not\equiv 0 \bmod k$, and $c = s^{-2k}$. % such that the following is true.
%	Suppose $1/n \ll \gamma, \alpha, 1/k, 1/s$, $\nu \ll 1/k, 1/s, \gamma$ and $\theta \ll \alpha, 1/k$.	
	For all $\gamma > 0$ and $0 \leq \alpha \leq 1$ there exists $n_0 = n_0(k, s, \gamma, \alpha) \in \mathbb{N}$ and $\nu = \nu(k, s, \gamma) > 0$ and $\theta = \theta(\alpha, k)$ such that following holds for all $n \ge n_0$.
	Let $\h$ be a $k$-graph on $n$ vertices that is strongly $(1/2 + 1/(2s) + \gamma, \theta)$-dense and $\phi(\h) \ge \alpha$. Then $\phi^\ast(\h, c) \leq (1 - \nu) \phi(\h)$.
\end{lemma}

We defer the proof of Lemma~\ref{lemma:fractionalisbetter} to the next subsection and now we use it to prove Lemma~\ref{lemma:integertilingdense}.

\begin{proof}[Proof of Lemma~\ref{lemma:integertilingdense}]
	Consider a fixed $\gamma > 0$.
	Suppose the result is false, that is, there exists $\alpha > 0$ such that for all $n \in \mathbb{N}$ and $\theta^\ast > 0$ there exists $n' > n$ satisfying $\Phi(n', 1/2 + 1/(2s) + \gamma, \theta^\ast) > \alpha$.
	Let $\alpha_0$ be the supremum of all such $\alpha$.
	Apply Lemma~\ref{lemma:fractionalisbetter} (with parameters $\gamma/2, \alpha_0/2$ playing the roles of $\gamma, \alpha$) to obtain $n_0 = n_0(k, s, \gamma/2, \alpha_0/2)$, $\nu = \nu(k, s, \gamma/2)$ and $\theta = \theta(\alpha_0/2, k)$.
%	Without loss of generality we may assume $c \leq s^{-k}$.
	Let \begin{align*}
		0 < \eta \ll \nu, \gamma, \alpha_0, 1/s.
	\end{align*}
	By the definition of $\alpha_0$, there exists $\theta_1 > 0$ and $n_1 \in \mathbb{N}$ such that for all $n \ge n_1$, \begin{align} \Phi(n, 1/2 + 1/(2s) + \gamma, \theta_1) \leq \alpha_0 + \eta/2. \label{eq:firstsupremumbound}\end{align}
	Now we prepare the setup to use the regular slice lemma (Theorem~\ref{theorem:regularslices}).
	Let $\beta, \eps_k, \eps, d, \theta^\ast, \theta' > 0$ and $t_0, t_1, r, n_2 \in \mathbb{N}$ be such that \begin{align*}
	1/n_2 \ll \eps, 1/r \ll \ 
	& \eps_k, 1/t_1 \ll 1/t_0 \ll \beta \ll \gamma' \ll \eta, c=s^{-k}, 1/s, 1/k, 1/n_0, 1/n_1; \\
	& \eps_k \ll d \ll \gamma'; \\
	& \eps_k \ll \theta' \ll \theta^\ast \ll \gamma', \theta, \theta_1
	\end{align*} and $n_2 \equiv 0 \bmod t_1!$.
	
	Let $\h$ be a $(1/2 + 1/(2s) + \gamma, \theta')$-dense $k$-graph on $n \ge n_2$ vertices with \begin{align}
	\phi(\h) > \alpha_0 - \eta, \label{eq:secondsupremumbound}
	\end{align} such $\h$ exists by the definition of~$\alpha_0$. By removing at most $t_1! - 1$ vertices we get a $k$-graph $\h'$ on at least $n_2$ vertices such that $|V(\h')|$ is divisible by $t_1!$ and $\h'$ is $(1/2 + 1/(2s) + \gamma - \gamma', 2 \theta')$-dense.
	
	Let $\mathcal{S}$ be the set of $(k-1)$-tuples $T$ of vertices of $V(H')$ such that $\deg_{H'}(T) < (1/2 + 1/(2s) + \gamma - \gamma')(|V(H')| - k + 1)$.
	Thus $|\mathcal{S}| \leq 2 \theta' \binom{|V(H')|}{k-1}$.	
	By Theorem~\ref{theorem:regularslices}, there exists a $(t_0, t_1, \eps, \eps_k, r)$-regular slice $\J$ for $\h'$ such that for all $(k-1)$-sets $Y$ of clusters of $\J$, we have $\overline{\deg}(Y; R(\h')) = \overline{\deg}(\J_Y; \h') \pm \eps_k$, and furthermore, $\J$ is $(3 \sqrt{2 \theta'}, \mathcal{S})$-avoiding.
	
	Let $R = R_d(\h')$ be the $d$-reduced $k$-graph obtained from $\h'$ and~$\J$.
	Since $\theta', d, \eps_k \ll \gamma'$, $\eps_k \ll \theta'$ and $\J$ is $(3 \sqrt{2 \theta'}, \mathcal{S})$-avoiding, Lemma~\ref{lemma:inheritanceofdensity} implies that $R$ is $(1/2 + 1/(2s) + \gamma - 2 \gamma', 5 \sqrt{\theta'})$-dense.
	By Lemma~\ref{lemma:stronglydense}, there exists a subgraph $R' \subseteq R$ on the same vertex set that is strongly $(1/2 + 1/(2s) + \gamma - 3 \gamma', \theta^\ast)$-dense as $\theta' \ll \gamma', 1/k, \theta^\ast$.
	Since the vertices of $R'$ are the clusters of $\J$, we have $|V(R')| \ge t_0 \ge n_1$.
	By the fact that $\theta^\ast \leq \theta_1$, Lemma~\ref{lemma:casialpha} (with $9 \gamma '$ playing the role of $\gamma$) and~\eqref{eq:firstsupremumbound}, we deduce that
	\begin{align*}
	\phi(R')
	& \leq \Phi( |V(R')|, 1/2 + 1/(2s) + \gamma - 3 \gamma', \theta^\ast ) \\
	& \leq \Phi( (1 + 9 \gamma') |V(R')|, 1/2 + 1/(2s) + \gamma, \theta^\ast ) + 9 \gamma' s \\
	& \leq \alpha_0 + \eta/2 + 9 \gamma' s \leq \alpha_0 + \eta.
	\end{align*}
	
	We further claim that $\phi^\ast(R', c) \leq \alpha_0 - 2 \eta$.
	Note that $c = s^{-k}$ and $\alpha_0 \ge 4 \eta$.
	Therefore, if $\phi(R') < \alpha_0 / 2$, then the claim holds by Proposition~\ref{proposition:fractionalw0}.
	Thus we may assume that $\phi(R') \ge \alpha_0/2$.
	Note that $|V(R')| \ge t_0 \ge n_0$, $\gamma - 3 \gamma' \ge \gamma / 2$, and $\theta^\ast \leq \theta$.
	By the choice of $n_0$, $\nu$, and $\theta$ (given by Lemma~\ref*{lemma:fractionalisbetter}), % (with $\alpha_0/2$ playing the role of $\alpha$),
	we have \[ \phi^\ast(R', c) \leq (1 - \nu) \phi(R') \leq (1 - \nu) (\alpha_0 + \eta) \leq \alpha_0 - 2 \eta, \] where the last inequality holds since $\eta \ll \nu, \alpha_0$.
	%Thus $\phi^\ast(R', c) \leq \alpha_0 - 2 \eta$ holds.
	Finally, recall that $\beta \ll \eta, c$, so an application of Lemma~\ref{lemma:blowuptiling} implies that \[ \phi(\h) \leq \phi^\ast(R, c) + s \beta / c \leq \phi^\ast(R', c) + s \beta / c \leq \alpha_0 - \eta,  \] contradicting \eqref{eq:secondsupremumbound}.
\end{proof}

\subsection{Proof of Lemma~\ref{lemma:fractionalisbetter}}

Before proceeding with the full details of the proof of Lemma~\ref{lemma:fractionalisbetter}, we first give a rough outline of the proof.
Let $\T$ be an $\{ F_s, E_s \}$-tiling of $\h$ satisfying $\phi(\T) = \phi(H)$.
By Proposition~\ref{proposition:fractionalw0}, we obtain a weighted fractional $\{ F^\ast_s, E^\ast_s \}$-tiling $\omega^\ast_0$ with $\phi(\omega^\ast_0) = \phi(\T)$, $U(\omega^\ast_0) = U(\T)$ and $(\omega^\ast_0)_{\min} \ge M_s^{-k}$.
Our aim is to sequentially define weighted fractional $\{ F^\ast_s, E^\ast_s \}$-tilings $\omega^\ast_1, \omega^\ast_2, \dotsc, \omega^\ast_t$ such that $\phi(\omega^\ast_{j-1}) - \phi(\omega^\ast_{j}) \ge \nu_1 / n$ for all $j \in [t]$, where $\nu_1$ is a fixed positive constant.
We will follow this procedure for $t = \Omega(n)$ steps, and we will show that $\omega^\ast_t$ satisfies the required properties.

Moreover, we will construct $\omega^\ast_{j+1}$ based on $\omega^\ast_j$ by changing the weights of $\F_s(\h)$ and $E(\h)$ on a small number of vertices, such that no vertex has its weight changed more than once during the whole procedure.
Recall that $U(\T)$ is the set of uncovered vertices.
If $|U(\T)|$ is large then we construct $\omega^\ast_{j+1}$ from $\omega^\ast_j$ via assigning weights to edges that contain at least $k-1$ vertices in~$U(\T)$.
Suppose that $|U(\T)|$ is small.
Since $\phi(\T) \ge \alpha$, not all of the weight of $\omega^\ast_0$ can contributed by  copies of $F^\ast_s$.
Thus there must exist edges $e \in E(H)$ with positive weight under~$\omega^\ast_0$.
We use this to find $e \in E(H)$ with $\omega^\ast_j(e) > 0$.
The crucial property is that a copy of $F^\ast_s$ might be obtained from an edge by adding a few extra vertices to it.
We use this to obtain $\omega^\ast_{j+1}$ from $\omega^\ast_j$ by reducing the weight on $e$ before assigning weight to some copy of $F^\ast_s$ which originates from $e$.
More care is needed to ensure that $\omega^\ast_{j+1}$ is indeed a weighted fractional $\{ F^\ast_s, E^\ast_s \}$-tiling.
Ideally we would like that the extra vertices which are added to $e$ to form a copy of $F^\ast_s$ are not saturated, if possible.

We summarise and recall the relevant properties of $F^\ast_s$, which was originally as defined at the beginning of Subsection~\ref{subsection:weightedfractionaltilings}.
There exists a $2$-graph $G_s$ on $[k]$ with $\ell \leq k-1$ edges which consists of a disjoint union of paths.
Suppose $e_1, \dotsc, e_{\ell}$ is an enumeration of the edges of $G_s$ and $e_i = j_i j'_i$ for all $1 \le i \le s$.
If $X = \{ x_1, \dotsc, x_k \}$, then we may describe $F^\ast_s$ as having vertices $V(F^\ast_s) = X \cup \{ y_1, \dotsc, y_{\ell} \}$, and the edges of $F^\ast_s$ are $X$ together with $(X \setminus \{ x_{j_i} \}) \cup \{ y_{i} \}$ and $(X \setminus \{ x_{j'_i} \}) \cup \{ y_{i} \}$ for all $1 \le i \le \ell$.
We call $c(F^\ast_s) = X$ and $p(F^\ast_s) = \{ y_1, \dotsc, y_{\ell} \}$ the core and pendant vertices of $F^\ast_s$, respectively.

The following two lemmas are needed for the case when $U(\T)$ is small.
The idea is the following: suppose $H$ is a $k$-graph on $n$ vertices with $\delta_{k-1}(H) \ge (1/2 + 1/(2s) + \gamma)n$.
If $X$ is a $k$-edge in $H$, we would like to extend it into a copy $F$ of $F^\ast_s$ such that $c(F) = X$.
Lemma~\ref{lemma:extendingedgetoF} will indicate where should we look for the vertices of $p(F)$.

\begin{lemma} \label{lemma:bipartiteauxgraph}
	Let $k \ge 3$, $s \ge 2k^2$ and $\ell \leq k - 1$.
%	Let $\h$ be a $k$-graph on $n$ vertices.
%	Let $X = \{ x_1, \dotsc, x_k \} \subseteq V(\h)$ and let $N_i = N_{\h}( X \setminus \{ x_i \} )$.
	Suppose that $N_i \subseteq [n]$ are such that $|N_i| \ge (1/2 + 1/(2s) + \gamma)n$ for all $1 \leq i \leq k$.
	Let $G$ be a $2$-graph on $\{ N_1, \dotsc, N_k \}$ such that $N_i N_j \in E(G)$ if and only if $|N_i \cap N_j| \leq (\ell/s + \gamma)n$.
	Then $G$ is bipartite.
\end{lemma}

\begin{proof}
	We will show that $G$ does not have any cycle of odd length.
	It suffices to show that $N_{i_1} N_{i_{2j+1}} \notin E(G)$ for all paths $N_{i_1} \dotsb N_{i_{2j+1}}$ in $G$ on an odd number of vertices.
	
	For any $S \subseteq [n]$, write $\overline{S} := [n] \setminus S$.	
	First, note that if $N_i$ is adjacent to $N_j$ in $G$, then $|N_i \setminus \overline{N_j}| = |N_i \cap N_j| \leq \left( \ell/s + \gamma \right) n$ and $|\overline{N_j} \setminus N_i| \leq (n - |N_j|) - (|N_i| - |N_i \cap N_j|) \leq \left( \ell/s - \gamma \right) n$.
	Hence if $N_i N_j N_k$ is a path on three vertices in $G$, then 
	\begin{align*}
	|N_i \setminus N_k|  \leq |N_i \setminus \overline{N_j}| + |\overline{N_j} \setminus N_k| \leq 2 \ell n / s.
	\end{align*}
		
	Now consider a path in $G$ on an odd number of vertices.
	Without loss of generality (after a suitable relabelling), we assume the path is given by $N_1 N_2 \dotsb N_{2j+1}$ for some $j$ which necessarily satisfies $2j+1 \leq k$.
	By using the previous bounds repeatedly, we obtain
	\begin{align*}
		|N_{1} \setminus N_{2j+1}|
		& \leq |N_1 \setminus N_3| + |N_3 \setminus N_5| + \dotsb + |N_{2j-1} \setminus N_{2j+1}|
		\leq \frac{2 \ell j n}{s} \leq \frac{\ell (k-1) n}{s}.
	\end{align*}
	Since $\ell \leq k-1$ and $s > 2k^2$, we obtain 
	\begin{align*}
	|N_{1} \cap N_{2j+1}|
	& \ge |N_{1}| - \frac{\ell(k-1)n}{s}
	\ge \left( \frac{1}{2} + \frac{1}{2s} + \gamma \right) n - \frac{(k-1)^2 n }{s}
	> \left( \frac{\ell}{s} + \gamma \right) n.
	\end{align*}
	Hence, $N_{1} N_{2j+1} \notin E(G)$ as desired.
\end{proof}

\begin{lemma} \label{lemma:extendingedgetoF}
	Let $k \ge 3$ and $s \ge 5k^2$ with $s \not\equiv 0 \bmod k$.
	Let $1/n \ll \gamma, 1/k$ and $\theta > 0$.
	Let $H$ be a strongly $(1/2 + 1/(2s) + \gamma, \theta)$-dense $k$-graph on $n$ vertices.
	%Let $F^\ast_s$ as defined at the beginning of Subsection~\ref{subsection:weightedfractionaltilings}.
	Let $X = \{ x_1, \dotsc, x_k \}$ be an edge of $H$ and let $N_i = N_H(X \setminus \{ x_i \})$ for all $1 \leq i \leq k$.
	Let $S \subseteq V(H)$ with $|S| \leq (\ell/s + \gamma/3)n$ and $y_0 \in N_1 \cap N_2$.
	Suppose either $|N_1 \cap N_2| < (\ell/s + 2 \gamma / 3)n$ or $|N_i \cap N_j| \ge (\ell/s + 2 \gamma/3)n$ for all $1 \leq i, j \leq k$.
	Then there exists a copy $F^\ast$ of $F^\ast_s$ such that $c(F^\ast) = X$ and $p(F^\ast) \cap (S \setminus \{ y_0 \}) = \emptyset$.
\end{lemma}

\begin{proof}
	Note that $|N_{i}| \ge (1/2 + 1/(2s) + \gamma)(n-k+1) \ge (1/2 + 1/(2s) + 2\gamma/3)n$ for all $1 \leq i \leq k$.
	Let $G$ be the $2$-graph on $[k]$ such that $ij \in E(G)$ if and only if $|N_i \cap N_j| < (\ell/s + 2\gamma/3)n$.
	Note that if $ij \notin E(G)$, then $|N_i \cap N_j| \ge (\ell/s + 2 \gamma/3)n \ge |S|+\ell$. %from $1/n \ll \gamma, 1/k$ and $|S| \leq (\ell/s + \gamma/3)n$ we get
	
	Recall that $G_s$, the $2$-graph which defines $F^\ast_s$, is a disjoint union of paths.
	By our assumption, either $12 \in E(G)$ or $G$ is empty.
	By Lemma~\ref{lemma:bipartiteauxgraph}, $G$ is bipartite.
	Thus, in either case, there exists a bijection $\phi: V(G_s) \rightarrow [k]$ such that $\{ \phi(j_i) \phi(j'_i) : j_i j'_i \in E(G_s) \} \cap E(G) \subseteq \{ 12 \}$.
	
	Let $e_1, \dotsc, e_{\ell}$ be an enumeration of the edges of~$E(G_s)$.
	Consider $e_i = j_i j'_i \in E(G_s)$.
	If $\{ \phi(j_i), \phi(j'_i) \} = \{ 1, 2 \}$, then let $y_i = y_0$.
	Otherwise, $\phi(j_i) \phi(j'_i) \notin E(G)$ and therefore $|N_{\phi(j_i)} \cap N_{\phi(j'_i)}| \ge |S| + \ell$.
	Thus we can greedily pick $y_i \in (N_{\phi(j_i)} \cap N_{\phi(j'_i)}) \setminus S$ such that $y_1, \dotsc, y_{\ell}$ are pairwise distinct.
	Then there exists a copy $F^\ast$ of $F^\ast_s$ with $c(F^\ast) = X$ and $p(F^\ast) = \{ y_1, \dotsc, y_{\ell} \}$, which satisfies the required properties.
\end{proof}

Now we are ready to prove Lemma~\ref{lemma:fractionalisbetter}.

\begin{proof}[Proof of Lemma~\ref{lemma:fractionalisbetter}]
	We may assume that $\gamma \ll \alpha, 1/k, 1/s$.
	Recall that our aim is to define a sequence of fractional $\{ F^\ast_s, E^\ast_s \}$-tilings $\omega^\ast_0, \dotsc, \omega^\ast_t$, for some $t \ge 0$.
	Let \[ \nu_1 = \frac{s}{25 k M_s^k} \text{,} \quad \nu_2 = \frac{\gamma}{40 k^3 s^{k}}, \quad \text{and} \quad \nu = \frac{\nu_1 \nu_2}{2}. \]
	Choose $\theta \ll \alpha, 1/k$ and $1/n_0 \ll \alpha, \gamma, 1/k, 1/s$.
	Let $\h$ be a strongly $(1/2 + 1/(2s) + \gamma, \theta)$-dense $k$-graph on $n \ge n_0$ vertices with $\phi(\h) \ge \alpha$.
	Choose $t = \lfloor \nu_2 \phi(H) n \rfloor$.
	
	Recall that $G_s, \ell, F_s, m_s, M_s$ are given by Proposition~\ref{corollary:fs} and they satisfy~\eqref{eq:lacotabonita} and \eqref{eq:lacotabonita2}. %and~\eqref{eq:lacotabendita}.
%	We deduce some inequalities that will be useful during the rest of the proof.
%	Note that since $\ell \leq k$, $m_s \ge (s - \ell)/k - 1$, $k \ge 3$ and $s \ge 5k^2$, it follows that \begin{align}
%	\frac{\ell}{m_s} \leq \frac{1}{4}. \label{eq:usefulinequality}
%	\end{align}
%	Also,~\eqref{eq:lacotabonita} implies that $m_s + 1 \ge M_s \ge \lceil (5k^2 + 1)/k \rceil - 1$ and the right hand size is minimised when $k=3$. Thus, \begin{align}
%	\frac{m_s}{M_s} \ge \frac{14}{15}. \label{eq:lacotabendita}
%	\end{align}	
	Let $\T$ be an $\{ F_s, E_s \}$-tiling on $\h$ with $\phi(\T) = \phi(\h)$.
	Apply Proposition~\ref{proposition:fractionalw0} and obtain a weighted fractional $\{ F^\ast_s, E^\ast_s \}$-tiling $w^\ast_0$ satisfying all the properties of the proposition.
	
	Given that $\omega^\ast_j$ has been defined for some $0 \leq j \leq t$, define \[ A_j = \{ v \in V(\h) : \forall J \in \F^\ast_s(\h) \cup E(\h), \ \omega^\ast_j(J) \alpha_J(v) = \omega^\ast_0(J) \alpha_J(v) \}.\]
	So $A_j$ is the set of vertices such that $\omega^\ast_j$ is ``identical to $w^\ast_0$''.
	Note that by Proposition~\ref{proposition:fractionalw0}\ref{proposition:fractionalw0verticesarefull}, for all $v \in A_j$,
	\begin{align}
	\omega^\ast_j(v) = \omega^\ast_0(v) \in \{0,1\}.  \label{eq:untouchedvertices}
	\end{align}
	Clearly we have $A_0 = V(\h)$.
	Let $\mathcal{T}^+_0 = \{ J \in \F^\ast_s(\h) \cup E(\h) : \omega^\ast_0(J) > 0 \}$.
	The set $A_j$ will indicate where we should look for graphs $J \in \mathcal{T}^+_0$ whose weight on $\omega^\ast_j$ is known (by knowing the weight on $J \in \omega^\ast_0$), and we will modify those to define the subsequent weighting $\omega^\ast_{j+1}$.
	
%	The following remarks state more precisely the properties of $A_j$ and $\T^+_0$ we will use.
%	Note that for every $E \in \mathcal{T}^+_0 \cap E(H)$, Proposition~\ref{proposition:fractionalw0}\ref{proposition:fractionalw0weightofE} implies that $\omega^\ast_0(E) = M_s^{-k}$.
%	Similarly, Proposition~\ref{proposition:fractionalw0}\ref{proposition:fractionalw0weightofF} implies that for every $F \in \mathcal{T}^+_0 \cap \F^\ast_s(H)$, $\omega^\ast_0(F) = ( \prod_{i=1}^k a_{s,i} )^{-1}$, where for each $1 \leq i \leq k$, $a_{s,i} \in \{ m_s, M_s \}$ satisfies $M_s - 1 \leq a_{s,i} \leq M_s$ by \eqref{eq:lacotabonita}.
%	Thus a simple calculation implies that $\omega^\ast_0(F) - M_s^{-k}$ is either $0$ (if $a_{s,i} = M_s$ for all $1 \leq i \leq k$) or at least $M_s^{-k}(M_s - 1)^{-1}$, and in the last case this is at least $M_s^{-k-1} \ge s^{-2k} = c$.
	By Proposition~\ref{proposition:fractionalw0} and \eqref{eq:lacotabonita}, we have that for all $J \in \mathcal{T}^+_0$, if $V(J) \cap A_j \neq \emptyset$, then $\omega^\ast_j(J) = \omega^\ast_0(J)$ and therefore
	\begin{align}
		\omega^\ast_j(J) - \frac{1}{M_s^{-k}} \begin{cases}
		= 0 & \text{if $J \in E(H)$ or $m_s = M_s$,} \\
		\ge c & \text{otherwise.}
		\end{cases}
		\label{eq:trackingweight}
	\end{align}
	
	Now we turn to the task of making the construction of $\omega^\ast_1, \dotsc, \omega^\ast_t$ explicit.
	
	\begin{claim} \label{claim:sequenceoftilings}
		There is a sequence of weighted fractional $\{ F^\ast_s, E^\ast_s \}$-tilings $\omega^\ast_1, \dotsc, \omega^\ast_t$ such that for all $1 \leq j \leq t$, \begin{enumerate}[label={\rm (\roman*)}]
			\item \label{claim:sequenceoftilingsii} $A_j \subseteq A_{j-1}$ and $|A_j| \ge |A_{j-1}| - 5k^2$;
			\item \label{claim:sequenceoftilingsiii} $(\omega^\ast_j)_{\min} \ge c$ and
			\item \label{claim:sequenceoftilingsiv} $\phi(\omega^\ast_j) \leq \phi(\omega^\ast_{j-1}) - \nu_1/n$.% \leq \phi(\h) - \nu_1 j / n$,
		\end{enumerate}
	\end{claim}

	Note that Lemma~\ref{lemma:fractionalisbetter} follows immediately from Claim~\ref{claim:sequenceoftilings} as $\phi(\omega^\ast_t) \leq \phi(\h) - \nu_1 t / n \leq (1 - \nu) \phi(\h) $.
	
	\medskip \noindent \emph{Proof of Claim~\ref{claim:sequenceoftilings}.}
	Suppose that, for some $0 \leq j < t$, we have already defined $\omega^\ast_0, \omega^\ast_1, \dotsc, \omega^\ast_j$ satisfying \ref{claim:sequenceoftilingsii}--\ref{claim:sequenceoftilingsiv}.
	We write $U_i = U(\omega^\ast_i)$, for each $0 \le i \le j$.
	Observe that $U_0 = U(\T)$ by the choice of $\omega^\ast_0$ and Proposition~\ref{proposition:fractionalw0}\ref{proposition:fractionalw0uncovered}.
	Note that \ref{claim:sequenceoftilingsii} implies that $|A_j| \ge |A_0| - 5 k^2 j \ge n - 5k^2 \nu_2 \phi(H)n \ge (1 - \alpha \gamma / 40)n$, and therefore 
	\begin{align}
	|V \setminus A_j| = n - |A_j| \leq \frac{\alpha \gamma}{40} n. \label{eqn:Aislarge}
	\end{align}	
%	Note that $U(\T) \setminus U_j \subseteq U(\T) \setminus A_j$.
%	Also $|A_j| \ge n - 5k^2j \ge n - 5k^2 \nu_2 \phi(H) n \ge (1 - \alpha \gamma/40)n$.
%	Thus we have \begin{align}
%	 |U(\T) \setminus U_j| \leq n - |A_j| \leq \frac{\alpha \gamma}{40} n. \label{eqn:Aislarge}
%	\end{align}
	
	Now our task is to construct $\omega^\ast_{j+1}$.
	We will use the following shorthand notation.		
	For all $J \in \F^\ast_s(\h) \cup E(\h)$, if we have already specified the values of $\omega^\ast_{j+1}$, then let \[\partial(J) = \omega^\ast_{j+1}(J) - \omega^\ast_j(J).\]	
	The proof splits on two cases depending on the size of~$U_0$.
	
	\medskip \noindent \textbf{Case 1: $|U_0| \ge 3 \alpha  n / 4$.}
	%Note first that $U_0 \setminus U_j \subseteq U_0 \setminus A_j$.
	Note that $(U_0\setminus U_j)\cap A_j=\emptyset$, which implies that $A_j \cap U_0\subseteq A_j \cap U_j$.
%	By \ref{claim:sequenceoftilingsii} and Proposition~\ref{proposition:fractionalw0}\ref{proposition:fractionalw0uncovered}, $|A_j \cap U_j| \ge |U(\T)| - 5k^2 j \ge |U(\T)| - 5 k^2 \nu_2 n \ge \alpha n / 2$.
	By~\eqref{eqn:Aislarge}, $|A_j \cap U_j| \ge |A_j \cap U_0| \ge |U_0| - \alpha \gamma n / 40 \ge 3 \alpha n / 4 - \alpha \gamma n / 40 \ge \alpha n / 2$.
	Together with $1/n \ll \alpha$, we get \[ \binom{|U_j \cap A_j |}{k - 1} \ge \binom{ \alpha n / 2 }{k - 1} \ge \frac{\alpha^{k-1}}{2^k}\binom{n}{k - 1} \ge \theta \binom{n}{k - 1} + k^2 \binom{n}{k - 2} \] as $\theta, 1/n \ll \alpha, 1/k$.
	Since $\h$ is strongly $(1/2 + 1/(2s) + \gamma, \theta)$-dense, we can (greedily) find $k$ disjoint $(k-1)$-sets $W_1, \dotsc, W_k$ of $U_j \cap A_j$ such that $\deg(W_i) \ge (1/2 + 1/(2s) + \gamma)(n - k + 1)$ for all $1 \leq i \leq k$.
	Define $N_i = N(W_i) \cap A_j$.
	Then \begin{align}
	|N_i|
	& \ge \left( \frac{1}{2} + \frac{1}{2s} + \gamma \right)(n-k+1) - (n - |A_j|) \stackrel{\mathclap{\eqref{eqn:Aislarge}}}{\ge} \left( \frac{1}{2} + \frac{1}{2s} + \frac{\gamma}{2} \right) n. \label{eq:cotaenni}
	\end{align}
	
	Suppose that for some $1 \leq i \leq k$, there exists $x \in N_i \cap U_j$.
	Then $e = \{ x \} \cup W_i \in E(\h)$, so we can define $\omega^\ast_{j+1}(e ) = 1$ and $\omega^\ast_{j+1}(J) = \omega^\ast_j(J)$ for all $J \in (\F^\ast_s(\h) \cup E(\h)) \setminus \{ e \}$.
	In this case, $|A_{j+1}| = |A_j| - k \ge |A_j| - 5k^2$, $(\omega^\ast_{j+1})_{\min} = (\omega^\ast_j)_{\min} \ge c$ and $\phi(\omega^\ast_{j+1}) = \phi(\omega^\ast_j) - 3s/(5n) \leq \phi(\omega^\ast_j) - \nu_1 / n$ so we are done.
	Thus, we may assume that \begin{align} \bigcup_{1 \leq i \leq k} N_i \subseteq A_j \setminus U_j. \label{eq:ni} \end{align}	
%	By \eqref{eq:untouchedvertices}, for all $x \in \bigcup_{1 \leq i \leq k} N_i$, $\omega^\ast_j(x) = \omega^\ast_0(x) = 1$.
	
	For all $F^\ast \in \F^\ast_s(\h)$ and $e \in E(\h)$, define 
	\begin{align*}
	d_{F^\ast} = \sum_{i=1}^{k} \left| N_i \cap c(F^\ast) \right|
	\text{ and }
	d_e = \sum_{i=1}^{k} |N_i \cap e|.
	\end{align*}

	\medskip \noindent \textbf{Case 1.1: there exists $F^\ast \in \F^\ast_s(H)$ with $\omega^\ast_j(F^\ast) > 0$ and $d_{F^\ast} \ge k+1$.}
	There exist distinct $i, i' \in \{1, \dotsc, k\}$ and distinct $x\in N_i\cap c(F^\ast)$, $x' \in N_{i'}\cap c(F^\ast)$ such that both $e_1 = W_i \cup \{ x \}$ and $e_2 = W_{i'} \cup \{ x' \}$ are edges in~$\h$.
	Note that since $x \in A_j$, by~\eqref{eq:trackingweight} we have $\omega^\ast_j(F^\ast) = \omega^\ast_0(F^\ast) \ge M_s^{-k}$.
	Also, since $x, x' \in c(F^\ast)$, $\alpha_{F^\ast}(x), \alpha_{F^\ast}(x') \ge m_s$.
	Define $\omega^\ast_{j+1}$ to be such that
	\[ \partial(J) %= \omega^\ast_{j+1}(J) - \omega^\ast_{j}(J)
	= \begin{cases}
	m_s M_s^{-(k+1)} & \text{if $J \in \{ e_1, e_2\}$,}\\
	-M_s^{-k} & \text{if $J = F^\ast$,} \\
	0 & \text{otherwise.}
	\end{cases} \]
	Then $\omega^\ast_{j+1}$ is a weighted fractional $\{ F^\ast_s, E^\ast_s \}$-tiling.
	First, note that $|A_{j+1}| = |A_j| - (3k + \ell - 2) \ge |A_j| - 5k^2$.
	Secondly, using \eqref{eq:trackingweight} we have that $\omega^\ast_j(F^\ast)$ is either $0$ or at least $c$.
	Thus we obtain
	\begin{align*}
	(\omega^\ast_{j+1})_{\min}
	& \ge \min \{ (\omega^\ast_j)_{\min}, M_s \omega^{\ast}_{j+1}(e_1), c \} 
	 \ge \min \{ c, m_s M_s^{-k}, c \} \ge c.
	\end{align*}
	Finally,
	\begin{align*}
	\phi(\omega^\ast_{j}) - \phi(\omega^\ast_{j+1})
	& = \frac{s}{n} \left( \partial(F^\ast) + \frac{3}{5} (\partial(e_1) + \partial(e_2) )  \right) = \frac{s}{n M_s^k} \left( \frac{6 m_s}{5 M_s} - 1 \right).
	\end{align*}
	Using \eqref{eq:lacotabonita2}, $s \ge 5k^2$, $\ell \leq k-1$ and $k \ge 3$, we can lower bound $m_s/M_s$ by \begin{align*}
	\frac{m_s}{M_s} \ge \frac{M_s - 1}{M_s} \ge \frac{s - \ell - k}{s - \ell} = 1 - \frac{k}{s -  \ell} \ge 1 - \frac{k}{5k^2 - k+1} \ge \frac{40}{43}.
	\end{align*}
	We deduce $\phi(\omega^\ast_{j}) - \phi(\omega^\ast_{j+1}) \ge 5 s / (43 M_s^k n) \ge \nu_1 / n$, so we are done in this subcase.
	
	\medskip \noindent \textbf{Case 1.2: there exists $e \in E(H)$ with $\omega^\ast_j(e) > 0$ and $d_{e} \ge k+1$.}
	We prove this case using a similar argument used in Case 1.1.
	There exist distinct $i, i' \in \{1, \dotsc, k\}$ and distinct $x, x' \in e$ such that both $e_1 = W_i \cup \{ x \}$ and $e_2 = W_{i'} \cup \{ x' \}$ are edges in~$\h$.
	Since $x \in A_j$, Proposition~\ref{proposition:fractionalw0}\ref{proposition:fractionalw0weightofE} and~\eqref{eq:trackingweight} implies that $\omega^\ast_j(e) = M_s^{-k}$.
	Define $\omega^\ast_{j+1}$ to be such that \[ \partial(J) = \begin{cases}
	-M_s^{-k} & \text{if $J = e$,} \\
	M_s^{-k} & \text{if $J \in \{  e_1, e_2 \}$,} \\
	0 & \text{otherwise.}
	\end{cases} \]
	Then $\omega^\ast_{j+1}$ is a weighted fractional $\{ F^\ast_s, E^\ast_s \}$-tiling with $|A_{j+1}| = |A_j| - (3k - 2) \ge |A_j| - 5k^2$.
	Note $\omega^\ast_{j+1}(e) = 0$ and $\omega_{j+1}^\ast(e_i) > \omega_{j}^\ast(e_i)$ for $i \in [2]$, so we have $(\omega^\ast_{j+1})_{\min} \ge (\omega^\ast_j)_{\min} \ge c$.
	Note that\begin{align*}
	\phi(\omega^\ast_{j}) - \phi(\omega^\ast_{j+1})
	& = \frac{3 s}{5 n} \left(\partial(e_1) + \partial(e_2) + \partial(e) \right) = \frac{ 3s}{ 5 M_s^k n} \ge \frac{\nu_1}{n},
	\end{align*}	
	so this finishes the proof of this subcase.
	
	\medskip \noindent \textbf{Case 1.3: Both Cases~1.1 and~1.2 do not hold.}
	Thus $d_{J} \leq k$ for all $J \in \F^\ast_s(\h) \cup E(\h)$ with $\omega^\ast_j(J) > 0$.
	Recall that $\alpha_{F^\ast}(v) \leq M_s$ if $v \in c(F^\ast)$ and $\alpha_{F^\ast}(v) = 1$ if $v \in p(F^\ast)$.
	Thus, for all $F^\ast \in \F^\ast_s(\h)$ with $\omega^\ast_j(F^\ast) > 0$, we have
	\begin{align*}
	\sum_{i=1}^{k} \sum_{x \in N_i} \alpha_{F^\ast}(x)
	& \leq \sum_{i=1}^{k} \left( M_s |N_i \cap c(F^\ast)| + |N_i \cap p(F^\ast)| \right) \\
	& =  M_s d_{F^\ast} + \sum_{i=1}^{k} |N_i \cap p(F^\ast)| \leq k(M_s + \ell) \leq s + k^2.
	\end{align*} Therefore,
	\begin{align}
	\sum_{F^\ast \in \F^\ast_s}  \sum_{i=1}^{k} \sum_{x \in N_i} \omega^\ast_0(F^\ast) \alpha_{F^\ast}(x)
	%	& = \sum_{F^\ast \in \F^\ast_s} \omega^\ast_0(F^\ast) \sum_{i=1}^{k} \sum_{x \in N_i} \alpha_{F^\ast}(x)  \nonumber \\
	& \leq (s + k^2) \sum_{F^\ast \in \F^\ast_s(\h)} w^\ast_0(F^\ast). \label{eq:boundofthedF}
	\end{align}	
	Similarly, for $e \in E(\h)$ with $\omega^\ast_j(e) > 0$, we obtain
	\begin{align*}
	\sum_{i=1}^k \sum_{x \in N_i} \alpha_e(x)
	& = \sum_{i=1}^k M_s |e \cap N_i| = M_s d_e \leq k M_s.
	\end{align*}
	Hence, \begin{align}
	\sum_{e \in E(\h)} \sum_{i=1}^k \sum_{x \in N_i} \omega^\ast_0(e) \alpha_e(x) \leq k M_s \sum_{e \in E(\h)} w^\ast_0(e).
	\label{eq:boundofthedE}
	\end{align}	
	Combining everything, we deduce \begin{align*}
	\sum_{i=1}^k |N_i|
	& = \sum_{i=1}^k \sum_{x \in N_i} 1
	\stackrel{\eqref{eq:ni}, \eqref{eq:untouchedvertices}}{=} \sum_{i=1}^k \sum_{x \in N_i} \omega^\ast_0(x) \\
	& = \sum_{i=1}^k \sum_{x \in N_i} \sum_{J \in \F^\ast_s(\h) \cup E(\h)} \omega^\ast_0(J) \alpha_J(x)
	= \sum_{J \in \F^\ast_s(\h) \cup E(\h)} \sum_{i=1}^k \sum_{x \in N_i} \omega^\ast_0(J) \alpha_J(x) \\
	& \stackrel{\mathclap{\eqref{eq:boundofthedF}, \eqref{eq:boundofthedE}}}{\leq}
	\quad (s + k^2) \sum_{F^\ast \in \F_s^\ast(\h)} w^\ast_0(F^\ast) + k M_s \sum_{e \in E(\h)} w^\ast_0(e) \\
	& \stackrel{\mathclap{\text{Prop.}~\ref{proposition:twopendantpoints}\ref{proposition:twopendantpointsi}}}{\leq}
	\quad n + k^2 \sum_{F^\ast \in \F_s^\ast(\h)} w^\ast_0(F^\ast) \leq n + \frac{k^2}{s}n \leq \frac{6n}{5},
	\end{align*} where the last inequality uses $s \ge 5k^2$.
	This contradicts \eqref{eq:cotaenni} and finishes the proof of Case~1.
	
	\medskip \noindent \textbf{Case 2: $|U_0| < 3 \alpha  n / 4$.}
	Write $\F$, $\E$ for $\F_\T$, $\E_\T$, respectively.
	Note that $n = s |\F| + k M_s | \E| + |U_0|$.
	Hence, \[ \alpha \leq \phi(\T) \leq 1 - \frac{s}{n} | \F | \leq \frac{1}{n} \left( k M_s | \E| + |U_0| \right) \leq \frac{k M_s | \E |}{n} + \frac{3 \alpha}{4}. \]
	Using that $s \ge 5k^2$, that $k \ge 3$, that $1/n \ll \alpha, \gamma \leq 1$ and \eqref{eqn:Aislarge}, we have \[ | \E | \ge \frac{\alpha n}{4 k M_s} \ge \frac{\alpha \gamma n}{40} + 1 \ge  n - |A_j| + 1. \]
	Hence there exists $E_s \in \E$ with $V(E_s) \subseteq A_j$.
	By Proposition~\ref{proposition:fractionalw0}\ref{proposition:fractionalw0weightofE}, there exists an edge $X = \{ x_1, \dotsc, x_k \} \in E(\h)$ such that $X \subseteq A_j$ and \[ w^\ast_j(X) = w^\ast_0(X) = M_s^{-k}. \]
	
%	Let $N_i = N_{\h}(X \setminus \{ x_i \})$ for all $1 \leq i \leq k$.
%	Since $X \in E(\h)$ and $\h$ is strongly $(1/2 + 1/(2s) + \gamma, \theta)$-dense, it follows that \[|N_i| \ge (1/2 + 1/(2s) + \gamma)(n - k +1) \ge (1/2 + 1/(2s) + 2 \gamma/ 3) n\] for all $1 \leq i \leq k$.
%	Define the graph $G_X$ on $\{ N_1, \dotsc, N_k \}$ where $N_u N_{v} \in E(G_X)$ if and only if $|N_u \cap N_{v}| \leq (\ell/s + 2 \gamma/ 3)n$.
%	By Lemma~\ref{lemma:bipartiteauxgraph}, $G_X$ is bipartite.
%	
%	Let $G_s$ be as defined in Proposition~\ref{corollary:fs}.
%	Since $G_s$ is a disjoint union of paths and $G_X$ is bipartite, there is a bijection $\psi: V(G_s) \rightarrow V(G_X)$ such that $\psi(u) \psi(v) \in E(G_X)$ for at most one edge $uv \in E(G_s)$.
%	Fix one such~$\psi$.
%	Without loss of generality (relabelling the vertices of $G_X$ if necessary), assume that $\psi(u) = N_u$ for all $u \in V(G_s) = [k]$.
%	Then $|N_u \cap N_v| > (\ell/s + 2 \gamma / 3 )n$ for all but at most one $uv \in E(G_s)$.

	We would like to use Lemma~\ref{lemma:extendingedgetoF} to find copies $F$ of $F^\ast_s$ with $c(F) = X$, and decrease the weight of $X$ to be able to increase the weight of an appropriate copy of $F^\ast_s$.	
	Recall that $S(\omega^\ast_j)$ is the set of saturated vertices with respect to $\omega^\ast_j$.
	We write $S_j = S(\omega^\ast_j)$ and let  $S' = S_j \cup (V(H) \setminus A_j)$.
	Proposition~\ref{proposition:twopendantpoints}\ref{proposition:twopendantpointsii} and~\eqref{eqn:Aislarge} together imply that $|S'| \leq (\ell/s + \gamma/40)n$.
	
	For all $1 \leq i \leq k$, let $N_i = N_H(X \setminus \{ x_i \})$.
	We may assume (by relabelling) that either $|N_1 \cap N_2| < (\ell/s + 2\gamma/3)n$ or $|N_i \cap N_j| \ge (\ell/s + 2\gamma/3)$ for all $1 \leq i, j \leq k$.
	
	\medskip \noindent \textbf{Case 2.1: $(N_1 \cap N_2) \setminus S' \neq \emptyset$.}
	In this case, select $y \in (N_1 \cap N_2) \setminus S'$ and apply Lemma~\ref{lemma:extendingedgetoF} with $S', y$ playing the roles of $S, y_0$.
	We obtain a copy $F_1$ of $F^\ast_s$ such that $c(F_1) = X$ and $p(F_1) \cap S' = \emptyset$.
	Then $p(F_1) \subseteq A_j \setminus S_j$.
	Let $P_0 = p(F_1) \setminus U_j$.
	For $p \in p(F_1) \cap U_j$, by \eqref{eq:untouchedvertices}, $\omega^\ast_j(p) = 0$.
	For every $p \in P_0$, by the definitions of $A_j$ and $U_j$, there exists $J_p \in \T^+_0$ such that $p \in V(J_p)$, and since $p \notin S_j$ we also can choose $J_p$ such that $\alpha_{J_p}(p) \ge m_s$.
	(The $J_p$ might coincide for different $p \in P_0$.)
	Define $\omega^\ast_{j+1}$ to be such that \[ \partial(J) = \begin{cases}
	M_s^{-k} & \text{if $J = F_1$,} \\
	- M_s^{-k} & \text{if $J = X$,} \\
	- M_s^{-k}/m_s & \text{if $J = J_p$ for some $p \in P_0$,} \\
	0 & \text{otherwise.} \\
	\end{cases} \]
	Then $\omega^\ast_{j+1}$ is a weighted fractional $\{ F^\ast_s, E^\ast_s \}$-tiling.
	First, note that $|A_{j+1}| \ge |A_j| - (|V(F_1)| + \sum_{p \in P_0} |V(J_p)| ) \ge |A_j| - (2k+2k^2) \ge |A_j| - 5k^2$.
	Secondly, \eqref{eq:trackingweight} implies that $\omega^{\ast}_{j+1}(X) = 0$ and $\omega^\ast_{j+1}(F_1) \ge c$, and moreover, for all $p \in P_0$, $\omega^{\ast}_{j+1}(J_p) \ge M_s^{-k}(1 - 1/m_s) \ge M_s^{-k-1} \ge c$.	Thus
	$(\omega^\ast_{j+1})_{\min} \ge c$.
	%	Since for all $i \in I$, $\omega^\ast_j(J_i) \ge M_s^{-k}$, we have $(\omega^\ast_{j+1})_{\min} \ge s^{-k} - s^{-k}/m_s \ge s^{-2k} = c$.
	Finally, since $|P_0| \leq |p(F_1)| = \ell$, we have
	\begin{align*}
	\phi(\omega^\ast_j) - \phi(\omega^\ast_{j+1})
	& \ge \frac{s}{n} \left(  \partial(F_1) + \frac{3}{5} \partial(X) + \sum_{p \in P_0} \partial(J_p) \right)
	\ge \frac{s}{n M_s^k} \left( \frac{2}{5}  - \frac{|P_0|}{m_s} \right) \\
	& \ge \frac{s}{n M_s^k} \left( \frac{2}{5}  - \frac{\ell}{m_s} \right).
	\end{align*}
	By \eqref{eq:lacotabonita2}, $\ell \leq k-1$ and $s \ge 5 k^2$,  we get \begin{align*}
	\frac{\ell}{m_s} \leq \frac{k-1}{M_s - 1} \leq \frac{k}{M_s} \leq \frac{k^2}{s - \ell} \leq \frac{k^2}{5k^2 - k + 1} \leq \frac{1}{4},
	\end{align*} where the last inequality holds for every $k \ge 3$.
	Thus $\phi(\omega^\ast_j) - \phi(\omega^\ast_{j+1}) \ge 3s / (20 n M^k_s) \ge \nu_1 / n$ and we are done.
	
	\medskip \noindent \textbf{Case 2.2: $N_1 \cap N_2 \subseteq S'$.}	
	Since $H$ is strongly $(1/2 + 1/(2s) + \gamma, \theta)$-dense and $1/n \ll \gamma, 1/k$, we deduce $|N_1 \cap N_2| \ge (1/s + \gamma)n$.
	Using $N_1 \cap N_2 \subseteq S'$ and~\eqref{eqn:Aislarge}, we have $|N_1 \cap N_2 \cap S_j \cap A_j| \ge (1/s + \gamma/2)n$.
	By Proposition~\ref{proposition:twopendantpoints}\ref{proposition:twopendantpointsiii}, there exists $F_2 \in \F^\ast_s(\h) \cap \T^+_0$ and $|p(F_2) \cap N_1 \cap N_2 \cap S_j \cap A_j| \ge 2$.
	Let $y'_1, y''_1$ be two distinct vertices in $p(F_2) \cap N_1 \cap N_2 \cap S_j \cap A_j$.
%	Let $F'_1, F''_1 \in \F^\ast_s(\h)$ be the copies of $F^\ast_s$ with $c(F'_1) = c(F''_1) = X$ and $p(F'_1) = \{ y'_1, y_2, y_3, \dotsc, y_\ell \}$ and $p(F''_1) = \{ y''_1, y_2, y_3, \dotsc, y_\ell\}$, respectively.
%	Note that since $V(F_2) \cap A_j \neq \emptyset$ and $w^\ast_j(F_2) > 0$, then by Proposition~\ref{proposition:fractionalw0}\ref{proposition:fractionalw0weightofF} we deduce, $\omega^\ast_j(F_2) = \omega^\ast_0(F_2) = ( \prod_{1 \leq i \leq k} a_{s,i} )^{-1} \ge m_s^{-k}$.
%	Since $V(F_2) \cap A_j \neq \emptyset$ and $w^\ast_j(F_2) > 0$, by \eqref{eq:trackingweight} we get $\omega^\ast_j(F_2) \ge M_s^{-k}$.	
	We claim that
	\begin{align}
	\parbox{0.85\textwidth}{there exists $F'_2 \in \F^\ast_s(\h)$ such that $p(F'_2) \setminus p(F_2) \subseteq A_j \setminus (S_j \cup X)$,
	their core vertices satisfy $c(F'_2) = c(F_2)$,
	and $\{ y'_1, y''_1 \} \setminus p(F'_2) \neq \emptyset$.} \label{eqn:F2prime}
	\end{align} 
	To see where we are heading, if we have found such $F'_2$, then our aim will be to define $\omega^\ast_{j+1}$ by decreasing the weight of $F_2$ and $X$, which will allow us then to increase the weight of $F'_2$ and a copy $F'_1$ of $F^\ast_s$ such that $c(F'_1) = X$ and $\{ y'_1, y''_1 \} \cap p(F'_1) \neq \emptyset$.
	
	Let us check \eqref{eqn:F2prime} holds.
	Let $Z = c(F_2) = \{ z_1, \dotsc, z_k \}$ and for every $1 \leq i \leq k$ let $Z_i = N_H(Z \setminus \{ z_i \})$.
	Since $y'_1 \in p(F_2)$, without loss of generality (by relabelling) we may assume that $y'_1 \in Z_1 \cap Z_2$.
	Suppose first that $(Z_1 \cap Z_2) \setminus (S' \cup X \cup V(F_2))$ is non-empty.
	Select any $y'''_1 \in (Z_1 \cap Z_2) \setminus (S' \cup X \cup V(F_2))$.
	Thus there exists $F'_2 \in \F^\ast_s(\h)$ such that $c(F'_2) = Z$, $p(F'_2) = (p(F_2) \setminus \{ y'_1 \}) \cup \{ y'''_1 \}$, $p(F'_2) \setminus p(F_2) = \{y'''_1 \} \subseteq A_j \setminus (S_j \cup X)$ and $y'_1 \in \{ y'_1, y''_1 \} \setminus p(F'_2)$, as desired.
	Hence, we may assume $Z_1 \cap Z_2 \subseteq S' \cup X \cup V(F_2)$.
	This implies that $|Z_1 \cap Z_2| \leq |S' \cup X \cup V(F_2)| \leq (\ell/s + \gamma/40)n + |X|+|V(F_2)| < (\ell/s + 2 \gamma/3)n$.
	Apply Lemma~\ref{lemma:extendingedgetoF} (with $Z$, $Z_i$, $S' \cup X \cup V(F_2)$, $y'_1$ playing the roles of $X$, $N_i$, $S$ and $y_0$, respectively) to obtain $F'_2 \in \F^\ast_s$ such that $c(F'_2) = Z$ and $p(F'_2) \cap ( S' \cup X \cup V(F_2) \setminus \{ y'_1 \} ) = \emptyset$.
	It is easily checked that $F'_2$ satisfies~\eqref{eqn:F2prime}.
	
	Now take such an $F'_2$ and assume (after relabelling, if necessary) that $y'_1 \notin p(F'_2)$.
	Apply Lemma~\ref{lemma:extendingedgetoF} (with $X, N_i, S' \cup V(F'_2), y'_1$ playing the roles of $X$, $N_i$, $S$ and $y_0$, respectively) to obtain $F'_1$ such that $c(F'_1) = X$ and $p(F'_1) \cap (S' \setminus \{ y'_1 \}) = \emptyset$.
	
	Let $P' = (p(F'_1) \setminus \{ y'_1 \}) \cup (p(F'_2) \setminus p(F_2))$ and observe that $P' \subseteq A_j \setminus S_j$.
	Let $P'_0 = P' \setminus U_j$.
	Arguing as in the previous case we see that for every $p \in P' \cap U_j$, $\omega^\ast_j(p) = 0$, and for every $p \in P'_0$ there exists $J_p \in \T^+_0$ such that $p \in V(J_p)$ and $\alpha_{J_p}(p) \ge m_s$.
	
	Let $\omega^\ast_{j+1}$ be such that \[ \partial(J) = \begin{cases}
		M_s^{-k} & \text{if $J = F'_1 $}, \\
		M_s^{-(k+1)} m_s & \text{if $J =  F'_2$}, \\
		- M_s^{-k} & \text{if $J \in \{ X, F_2 \}$}, \\
		- M_s^{-k} / m_s & \text{if $J = J_p$ for some $p \in P'_0$,} \\
		0 & \text{otherwise.} \\
	\end{cases} \]
	Since $p(F'_1) \cup p(F'_2) \subseteq P' \cup p(F_2)$, the decrease of weight in $F_2$ and the $J_p$ implies that the vertices in $p(F'_1) \cup p(F'_2)$ get weight at most $1$ under $\omega^\ast_{j+1}$.
	Using that, it is not difficult to check that $\omega^\ast_{j+1}$ is indeed a weighted fractional $\{ F^\ast_s, E^\ast_s \}$-tiling.
	
	Note that $A_{j} \setminus A_{j+1} \subseteq V(F'_1) \cup V(F_2) \cup V(F'_2) \cup \bigcup_{p \in P'_0} V(J_p)$ and $|P'_0| \leq |p(F'_1)| + |p(F'_2)| = 2 \ell$.
	Using that $\ell \leq k-1$, we deduce $|A_{j+1}| \ge |A_j| - 3(k+\ell) - |P'_0|(k+\ell) \ge |A_j| - (3 + 2 \ell)(k+\ell) \ge |A_j| - 5k^2$.
%	We need to check that $(\omega^\ast_{\min})_{j+1} \ge c$.
	Similarly as in the previous case, we deduce from \eqref{eq:trackingweight} that %for all $J \in \J$, $\omega^\ast_{j+1}(J) \ge M_s^{-k}(1 - 1/m_s) \ge M_s^{-k-1} \ge c$.
%	Other than that, only $X$ and $F_2$ had their weights decreased, but $\omega_{j+1}^\ast(X) = 0$ and \eqref{eq:trackingweight} implies that $\omega^\ast_{j+1}(F_2)$ is either $0$ or at least $c$. In summary,
	$(\omega^\ast_{j+1})_{\min} \ge c$.
	 
	Using that $|P'_0| \leq 2 \ell$, we deduce \begin{align*}
		\phi(\omega^\ast_{j}) - \phi(\omega^\ast_{j+1})
		& \ge \frac{s}{n} \left( \partial(F'_1) + \partial(F'_2) + \partial(F_2) + \frac{3}{5} \partial(X) + \sum_{p \in P'_0} \partial(J_p) \right) \\
		& = \frac{s}{n M_s^k} \left( 1 + \frac{m_s}{M_s} - 1 - \frac{3}{5} - \frac{|P'_0|}{m_s} \right)
		\ge \frac{s}{n M_s^k} \left( \frac{m_s}{M_s} - \frac{3}{5} - \frac{2 \ell}{m_s} \right).
	\end{align*}
	From \eqref{eq:lacotabonita2}, $s \ge 5k^2$ and $\ell \leq k-1$, we deduce \begin{align*}
	\frac{m_s}{M_s} - \frac{3}{5} - \frac{2 \ell}{m_s}
	& \ge \frac{2}{5} - \frac{1}{M_s} - \frac{2 \ell}{m_s}
	\ge \frac{2}{5} - \frac{1 + 2 \ell}{m_s}
	\ge \frac{2}{5} - \frac{k(1 + 2 \ell)}{s - \ell - k} \\
	& \ge \frac{2}{5} - \frac{2k^2 - k}{5k^2 - 2k + 1}
	= \frac{k+2}{25k^2 - 10 k + 5} \ge \frac{k+2}{25k^2} \ge \frac{1}{25 k}. \end{align*}
	Thus we get $\phi(\omega^\ast_{j}) - \phi(\omega^\ast_{j+1}) \ge s/(25 M_s^{k}kn) \ge \nu_1/n$ and we are done.
	This finishes the proof of Case 2.2 and of Claim~\ref{claim:sequenceoftilings}.
	\hfill $\square$
	
	This concludes the proof of Lemma~\ref{lemma:fractionalisbetter}.
\end{proof}

\section{Remarks and further directions} \label{section:remarks}

The following family of examples gives lower bounds for the Tur\'an problems of tight cycles on a number of vertices not divisible by $k$ (and hence for the tiling and covering problem, as well).
We acknowledge and thank a referee for suggesting this construction.
We are not aware of its appearance in the literature before, although it bears some resemblance to examples considered by Mycroft to give lower bounds for tiling problems~\cite[Section 2]{Mycroft2016}.

\begin{construction} \label{construction:thanksreferee}
	Let $k \ge 2$ and $p > 1$ be a divisor of~$k$.
	For $n > 0$, we define the $k$-graph $H^k_{n,p}$ as follows.
	Given a vertex set~$V$ of size $n$, partition it into $p$ disjoint vertex sets $V_1, \dotsc, V_{p}$ of size as equal as possible.
	Assume that every $x \in V_i$ is labelled with $i$, for all $1 \le i \le p$.
	Let $H^k_{n,p}$ be the $k$-graph on $V$ where the edges are the $k$-sets such that the sum of the labels of its vertices is congruent to $1$ modulo $p$.
\end{construction}

Using this construction, we deduce the following lower bounds for $\ex_{k-1}(n, C^k_s)$ when $s$ is not divisible by $k$ (and therefore, also for $c(n, C^k_s)$).

\begin{proposition} \label{proposition:turanthresholds}
	Let $s > k \ge 2$ with $s$ not divisible by $k$.
	Let $p$ be a divisor of~$k$ which does not divide $s$.
	Then $\ex_{k-1}(n, C^k_s) \ge \lfloor n/p \rfloor - k + 2$.
	In particular, $\ex_{k-1}(n, C^k_s) \ge \lfloor n/k \rfloor - k + 2$.
\end{proposition}

\begin{proof}
	Given $k, p, n$, let $H = H^k_{n,p}$ be the $k$-graph given by Construction~\ref{construction:thanksreferee}.
	Since the sets~$V_i$ are chosen to have size as equal as possible, we deduce $|V_i| \ge \lfloor n/p \rfloor$ holds for all $1 \le i \le p$.
	It is easy to check that no edge of $\h$ is entirely contained in any set $V_i$,
	and that, for every $(k-1)$-set~$S$ in~$V$, $N(S) = V_j \setminus S$ for some~$j$.
	Thus $\delta_{k-1}(H) \ge \lfloor n/p \rfloor - k + 2$.
	
	We show that $\h$ is $C^k_s$-free.
	Let $C$ be a tight cycle on $t$ vertices in $\h$.
	It is enough to show that $p$ divides $t$ (since $p$ does not divide $s$, it will follow that $t \neq s$).
	Recall from Construction~\ref{construction:thanksreferee} that every $x \in V_i$ is labelled with~$i$.
	We double count the sum $T$ of the labels of vertices, over all the edges of $C$.
	On one hand, $T \equiv 0 \bmod k$ since each vertex appears in exactly $k$ edges of $C$ and thus is counted $k$ times.
	Since $p$ divides $k$, $T \equiv 0 \bmod p$.
	On the other hand, the sum of the labels of a single edge is congruent to $1$ modulo $p$ and there are $t$ of them, thus $T \equiv t \bmod p$.
	This implies that $p$ divides $t$.
\end{proof}

Now we discuss covering thresholds.
Let $s > k \ge 3$.
Theorem~\ref{theorem:coveringthresholdnotmodk} and Proposition~\ref{proposition:lowerboundscovering} imply that $c(n, C_s^k) = (1/2 + o(1))n$ for all admissible pairs $(k,s)$ with $s \ge 2k^2$.
A natural open question is to determine $c(n, C_s^k)$ for the non admissible pairs~$(k,s)$.
The smallest case not covered by our constructions is when $(k,s) = (6,8)$, and Proposition~\ref{proposition:turanthresholds} implies that $c(n, C^6_8) \ge \lfloor n/3 \rfloor - 4$.

\begin{question}
	Is the lower bound for $c(n, C_s^k)$ given by Proposition~\ref{proposition:turanthresholds} asymptotically tight, for non admissible pairs $(k,s)$? 
	In particular, is $c(n, C^6_8) = (1/3 + o(1))n$?
\end{question}

Now, we consider the Tur\'an thresholds.
Theorem~\ref{theorem:coveringthresholdnotmodk} and Proposition~\ref{proposition:lowerboundscovering} also show that $\ex_{k-1}(n, C_s^k) = (1/2 + o(1))n$ for $k$ even, $s \ge 2k^2$ and $(k, s)$ is an admissible pair.
We would like to know the asymptotic value of $\ex_{k-1}(n, C_s^k)$ in the cases not covered by our constructions.
Proposition~\ref{proposition:turanthresholds} implies that $\ex_{k-1}(n, C_s^k) \ge \lfloor n/k \rfloor - k+ 2$ for $s$ not divisible by $k$; but on the other hand, if $s \equiv 0 \bmod k$ then $\ex_{k-1}(n, C_s^k) = o(n)$, which follows easily from Theorem~\ref{theorem:kovarisosturanhypergraphs}.

The simplest open case is when $k = 3$ and $s$ is not divisible by $3$.
Note that $C^3_4 = K^3_4$, and the lower bound $\ex_{2}(n, C^3_4) \ge (1/2 + o(1))n$ holds in this case~\cite{CzygrinowNagle2001}.
We conjecture that in the case $k=3$, for $s > 4$ and not divisible by three, the lower bound given by Proposition~\ref{proposition:turanthresholds} describes the correct asymptotic behaviour of~$\ex_{k-1}(n, C^k_s)$.

\begin{conjecture}
	$\ex_2(n, C^3_s) = (1/3 + o(1))n$ for every $s > 4$ with $s \not\equiv 0 \bmod 3$.
\end{conjecture}

%For the lower bound, consider disjoint vertex sets $V_1$, $V_2$ and $V_3$ each of size~$n/3$.
%Let $\h$ be the $3$-graph on $V_1 \cup V_2 \cup V_3$ with $E(\h) = \{ xyz : x \in V_i; y, z \in V_{i+1} \}$, where the indices are considered modulo $3$.
%However we do not know any non-trivial lower bound for $\ex_{k-1}(n, C_s^k)$ for $k \ge 5$ odd and $s \not\equiv 0 \bmod k$ with $s \ge k$.

Finally, we discuss tiling thresholds.
Let $(k,s)$ be an admissible pair such that $s \ge 5k^2$.
If $k$ is even, then Theorem~\ref{theorem:tilingthreshold} and Proposition~\ref{proposition:lowerboundstiling} imply that $t(n, C_s^k) = (1/2 + 1/(2s) + o(1))n$.
We conjecture that for $k$ odd, the bound given by Proposition~\ref{proposition:lowerboundstiling} is asymptotically tight.

\begin{conjecture}
	Let $(k,s)$ be an admissible pair such that $k \ge 3$ is odd and $s \ge 5k^2$.
	Then $t(n, C_s^k) = (1/2 + k/(4 s(k-1) + 2k  ) + o(1))n$.
\end{conjecture}

Note that, for $k$ odd, the extremal example given by Proposition~\ref{proposition:lowerboundstiling} is an example of the so-called space barrier construction.
However, it is different from the common construction which is obtained by attaching a new vertex set $W$ to an $F$-free $k$-graph and adding all possible edges incident with~$W$.
On the other hand, for $k$ even, it is indeed the common construction of a space barrier.

It also would be interesting to find bounds on the Tur\'an, covering and tiling thresholds that hold whenever $k < s \leq 5k^2$. The known thresholds for these kind of $k$-graphs do not necessarily follow the pattern of the bounds we have found for longer cycles. For example, note that $C^k_{k+1}$ is a complete $k$-graph on $k+1$ vertices, which suggests that for lower values of $s$ the problem behaves in a different way. Concretely, when $(k, s) = (3, 4)$, it is known that $t(n, C^3_4) = (3/4 + o(1))n$ \cites{KeevashMycroft2014, LoMarkstroem2015}.

\begin{question}
	Given $k \ge 3$, what is the minimum $s$ such that $t(n, C_s^k) \leq (1/2 + 1/(2s) + o(1))n$ holds?
\end{question}

\section*{Acknowledgements}

We thank Richard Mycroft and Guillem Perarnau for their valuable comments and insightful discussions.
We also thank an anonymous referee for their comments and suggestions that simplified some parts and vastly improved the presentation of the paper.
In particular, we are grateful for their suggestions of a simpler proof of Lemma~\ref{lemma:bipartiteauxgraph} and Construction~\ref{construction:thanksreferee}.
\bibliography{may2019}

\begin{bibdiv}
\begin{biblist}

\bib{Abbasi1998}{thesis}{
      author={Abbasi, Sarman},
       title={The solution of the {E}l-{Z}ahar problem},
        type={Ph.D. Thesis},
organization={Rutgers University},
        date={1998},
}

\bib{AllenBottcherCooleyMycroft2017}{article}{
      author={{Allen}, P.},
      author={{B{\"o}ttcher}, J.},
      author={{Cooley}, O.},
      author={{Mycroft}, R.},
       title={{Tight cycles and regular slices in dense hypergraphs}},
        date={2017},
     journal={J. Combin. Theory Ser. A},
      volume={149},
       pages={30\ndash 100},
}

\bib{AlonYuster1996}{article}{
      author={Alon, Noga},
      author={Yuster, Raphael},
       title={${H}$-factors in dense graphs},
        date={1996},
     journal={J. Combin. Theory Ser. B},
      volume={66},
      number={2},
       pages={269\ndash 282},
}

\bib{CooleyFountoulakisKuehnEtAl2009}{article}{
      author={Cooley, Oliver},
      author={Fountoulakis, Nikolaos},
      author={K{\"u}hn, Daniela},
      author={Osthus, Deryk},
       title={{E}mbeddings and {R}amsey numbers of sparse $k$-uniform
  hypergraphs},
        date={2009},
     journal={Combinatorica},
      volume={29},
      number={3},
       pages={263\ndash 297},
}

\bib{CorradiHajnal1963}{article}{
      author={Corr{\'a}di, Kereszt{\'e}ly},
      author={Hajnal, Andr{\'a}s},
       title={{O}n the maximal number of independent circuits in a graph},
        date={1963},
     journal={Acta Math. Hungar.},
      volume={14},
      number={3--4},
       pages={423\ndash 439},
}

\bib{Czygrinow2016}{article}{
      author={Czygrinow, Andrzej},
       title={Tight co-degree condition for packing of loose cycles in
  3-graphs},
        date={2016},
        ISSN={1097-0118},
     journal={J. Graph Theory},
      volume={83},
      number={4},
       pages={317\ndash 333},
         url={http://dx.doi.org/10.1002/jgt.21998},
}

\bib{CzygrinowNagle2001}{article}{
      author={Czygrinow, Andrzej},
      author={Nagle, Brendan},
       title={A note on codegree problems for hypergraphs},
        date={2001},
     journal={Bull. Inst. Combin. Appl.},
      volume={32},
       pages={63\ndash 69},
}

\bib{Dirac1952}{article}{
      author={Dirac, Gabriel~Andrew},
       title={{S}ome theorems on abstract graphs},
        date={1952},
     journal={Proc. London Math. Soc.},
      volume={3},
      number={1},
       pages={69\ndash 81},
}

\bib{ElZahar1984}{article}{
      author={El-Zahar, Mohamed~H.},
       title={On circuits in graphs},
        date={1984},
        ISSN={0012-365X},
     journal={Discrete Math.},
      volume={50},
      number={2-3},
       pages={227\ndash 230},
         url={http://dx.doi.org/10.1016/0012-365X(84)90050-5},
}

\bib{Erdoes1964}{article}{
      author={Erd\H{o}s, Paul},
       title={{O}n extremal problems of graphs and generalized graphs},
        date={1964},
     journal={Israel J. Math.},
      volume={2},
      number={3},
       pages={183\ndash 190},
}

\bib{Falgas-RavryZhao2016}{article}{
      author={Falgas-Ravry, Victor},
      author={Zhao, Yi},
       title={Codegree thresholds for covering $3$-uniform hypergraphs},
        date={2016},
     journal={SIAM J. Discrete Math.},
      volume={30},
      number={4},
       pages={1899\ndash 1917},
}

\bib{GaoHanZhao2016}{article}{
      author={Gao, Wei},
      author={Han, Jie},
      author={Zhao, Yi},
       title={Codegree conditions for tiling complete $k$-partite $k$-graphs
  and loose cycles},
        date={2016},
     journal={Combin. Probab. Comput. (to appear) arXiv:1612.07247},
}

\bib{HajnalSzemeredi1970}{incollection}{
      author={Hajnal, A.},
      author={Szemer{\'e}di, E.},
       title={{P}roof of a conjecture of {P}. {E}rd{\H{o}}s},
        date={1970},
   booktitle={Combinatorial theory and its applications, {II} ({P}roc.
  {C}olloq., {B}alatonf\"ured, 1969)},
   publisher={North-Holland, Amsterdam},
       pages={601\ndash 623},
}

\bib{HanZangZhao2015}{article}{
      author={Han, Jie},
      author={Zang, Chuanyun},
      author={Zhao, Yi},
       title={{M}inimum vertex degree thresholds for tiling complete
  $3$-partite $3$-graphs},
        date={2017},
     journal={J. Combin. Theory Ser. A},
      volume={149},
       pages={115\ndash 147},
}

\bib{HaxellEtAl2009}{article}{
      author={Haxell, P.~E.},
      author={\L{}uczak, T.},
      author={Peng, Y.},
      author={R\"odl, V.},
      author={Ruci\'nski, A.},
      author={Skokan, J.},
       title={The {R}amsey number for 3-uniform tight hypergraph cycles},
        date={2009},
     journal={Combin. Probab. Comput.},
      volume={18},
      number={1-2},
       pages={165\ndash 203},
}

\bib{JansonLuczakRucinski2000}{book}{
      author={Janson, Svante},
      author={{\L}uczak, Tomasz},
      author={Ruci\'{n}ski, Andrzej},
       title={Random graphs},
      series={Wiley-Interscience Series in Discrete Mathematics and
  Optimization},
   publisher={Wiley-Interscience, New York},
        date={2000},
        ISBN={0-471-17541-2},
}

\bib{KeevashMycroft2014}{article}{
      author={Keevash, Peter},
      author={Mycroft, Richard},
       title={A geometric theory for hypergraph matching},
        date={2015},
     journal={Mem. Amer. Math. Soc.},
      volume={233},
      number={1098},
       pages={vi+95},
}

\bib{KovariSosTuran1954}{article}{
      author={K\H{o}v{\'a}ri, Tam{\'a}s},
      author={S{\'o}s, Vera},
      author={Tur{\'a}n, P{\'a}l},
       title={{O}n a problem of {K}. {Z}arankiewicz},
organization={Institute of Mathematics Polish Academy of Sciences},
        date={1954},
     journal={Colloq. Math.},
      volume={3},
      number={1},
       pages={50\ndash 57},
}

\bib{Komlos2000}{article}{
      author={Koml{\'o}s, J{\'a}nos},
       title={{T}iling {T}ur{\'a}n theorems},
        date={2000},
     journal={Combinatorica},
      volume={20},
      number={2},
       pages={203\ndash 218},
}

\bib{KomlosSarkoezySzemeredi2001}{article}{
      author={Koml{\'o}s, J{\'a}nos},
      author={S{\'a}rk{\"o}zy, G{\'a}bor},
      author={Szemer{\'e}di, Endre},
       title={{P}roof of the {A}lon--{Y}uster conjecture},
        date={2001},
     journal={Discrete Math.},
      volume={235},
      number={1},
       pages={255\ndash 269},
}

\bib{KuehnMycroftOsthus2010}{article}{
      author={K{\"u}hn, Daniela},
      author={Mycroft, Richard},
      author={Osthus, Deryk},
       title={{H}amilton $\ell$-cycles in uniform hypergraphs},
        date={2010},
     journal={J. Combin. Theory Ser. A},
      volume={117},
      number={7},
       pages={910\ndash 927},
}

\bib{KuehnOsthus2006}{article}{
      author={K{\"u}hn, Daniela},
      author={Osthus, Deryk},
       title={{M}atchings in hypergraphs of large minimum degree},
        date={2006},
     journal={J. Graph Theory},
      volume={51},
      number={4},
       pages={269\ndash 280},
}

\bib{LoMarkstroem2015}{article}{
      author={Lo, Allan},
      author={Markstr{\"o}m, Klas},
       title={${F}$-factors in hypergraphs via absorption},
        date={2015},
     journal={Graphs Combin.},
      volume={31},
      number={3},
       pages={679\ndash 712},
}

\bib{Mycroft2016}{article}{
      author={Mycroft, Richard},
       title={{P}acking {$k$}-partite {$k$}-uniform hypergraphs},
        date={2016},
     journal={J. Combin. Theory Ser. A},
      volume={138},
       pages={60\ndash 132},
}

\bib{RoedlRucinskiSzemeredi2009}{article}{
      author={R{\"o}dl, Vojtech},
      author={Ruci{\'n}ski, Andrzej},
      author={Szemer{\'e}di, Endre},
       title={{P}erfect matchings in large uniform hypergraphs with large
  minimum collective degree},
        date={2009},
     journal={J. Combin. Theory Ser. A},
      volume={116},
      number={3},
       pages={613\ndash 636},
}

\bib{Szemeredi2013}{incollection}{
      author={Szemer{\'e}di, Endre},
       title={Is laziness paying off? (``absorbing'' method)},
        date={2013},
   booktitle={Colloquium {D}e {G}iorgi 2010--2012},
      editor={Zannier, Umberto},
   publisher={Scuola Normale Superiore},
     address={Pisa},
       pages={17\ndash 34},
         url={http://dx.doi.org/10.1007/978-88-7642-457-1_3},
}

\bib{Wang2010}{article}{
      author={Wang, Hong},
       title={Proof of the {E}rd{\H{o}}s-{F}audree conjecture on
  quadrilaterals},
        date={2010},
        ISSN={0911-0119},
     journal={Graphs Combin.},
      volume={26},
      number={6},
       pages={833\ndash 877},
         url={http://dx.doi.org/10.1007/s00373-010-0948-3},
}

\bib{Wang2012}{article}{
      author={Wang, Hong},
       title={Disjoint 5-cycles in a graph},
        date={2012},
        ISSN={1234-3099},
     journal={Discuss. Math. Graph Theory},
      volume={32},
      number={2},
       pages={221\ndash 242},
         url={http://dx.doi.org/10.7151/dmgt.1605},
}

\bib{Zhao2016}{incollection}{
      author={Zhao, Yi},
       title={Recent advances on {Dirac-type} problems for hypergraphs},
        date={2016},
   booktitle={Recent trends in combinatorics},
      editor={Beveridge, Andrew},
      editor={Griggs, Jerrold~R.},
      editor={Hogben, Leslie},
      editor={Musiker, Gregg},
      editor={Tetali, Prasad},
   publisher={Springer International Publishing},
       pages={145\ndash 165},
}

\end{biblist}
\end{bibdiv}

\appendix

\section{Hypergraph regularity}

In Section~\ref{section:regularity} we stated modified versions of some regularity statements which follow from easy modifications of the original statements or proofs.
In this appendix we sketch how to guarantee those properties hold.

\subsection{Avoiding fixed $(k-1)$-graphs} \label{appendix:avoiding}

Our version of the Regular Slice Lemma (Theorem~\ref{theorem:regularslices}) includes an additional property (that of ``avoiding'' a fixed $(k-1)$-graph $\mathcal{S}$ on the same vertex set as $G$) which is not present in the original statement \cite[Lemma 10]{AllenBottcherCooleyMycroft2017}.
We claim that extra property follows already from their proof by doing one simple extra step.

Their proof of the Regular Slice Lemma can be summarised as follows (we refer the reader to \cite{AllenBottcherCooleyMycroft2017} for the precise definitions).
First, they obtain an ``equitable family of partitions'' $\mathcal{P}^\ast$ from (a strengthened version of) the Hypergraph Regularity Lemma.
This can be used to find suitable complexes in the following way: first, for each pair of clusters of $\mathcal{P}^\ast$, select a $2$-cell uniformly at random. Then, for each triple of clusters of $\mathcal{P}^\ast$ select a $3$-cell uniformly at random which is supported on the corresponding previously selected $2$-cells; and so on, until we select $(k-1)$-cells.
This will always output a $(t_0, t_1, \eps)$-equitable $(k-1)$-complex $\J$, and the task is to check that, with positive probability, $\J$ is actually a $(t_0, t_1, \eps, \eps_k, r)$-regular slice satisfying the ``desired properties'' with respect to the reduced $k$-graph.

Having selected $\J$ at random as before, the most technical part of the proof is to show that the ``desired properties'' of the reduced $k$-graph (labelled (a), (b) and (c) in \cite[Lemma 10]{AllenBottcherCooleyMycroft2017}) hold with probability tending to $1$ whenever $n$ goes to infinity.
Thankfully, that part of the proof does not require any modification for our purposes.
Moreover, the selected $\J$ will be a $(t_0, t_1, \eps, \eps_k, r)$-regular slice with probability at least $1/2$.
This is shown by upper bounding the expected number of $k$-sets of clusters of $\mathcal{J}$ for which $G$ is not $(\eps_k, r)$-regular, and an application of Markov's inequality (cf. \cite[pp. 65--66]{AllenBottcherCooleyMycroft2017}).
It is a natural adaptation of this method that will show that $\J$ is also $(3 \theta^{1/2}, \mathcal{S})$-avoiding with probability at least $2/3$.

Let $\mathcal{S}$ be a $(k-1)$-graph on $V(G)$ of size at most $\theta \binom{n}{k-1}$.
We only need to consider the edges of $\mathcal{S}$ which are $\mathcal{P}$-partite.
Every $\mathcal{P}$-partite edge of $\mathcal{S}$ is supported in exactly one $(k-1)$-cell of the family of partitions $\mathcal{P}^\ast$, which by \cite[Claim 32]{AllenBottcherCooleyMycroft2017} is present in $\J$ with probability $p = \prod_{i=2}^{k-1} d_i^{\binom{k-1}{j}}$.
Thus the expected size of $|E(\mathcal{S}) \cap E(\J_{k-1})|$ is at most $|E(\mathcal{S})|p \leq \theta p \binom{n}{k-1}$.
By Markov's inequality, with probability at least $2/3$ we have $|E(H) \cap E(\J_{k-1})| \leq 3 \theta p \binom{n}{k-1}$.
By the previous discussion, with positive probability $\J$ satisfies all of the properties of \cite[Lemma 10]{AllenBottcherCooleyMycroft2017} and also that $|E(\mathcal{S}) \cap E(\J_{k-1})| \leq 3 \theta p \binom{n}{k-1}$.
Thus we may assume $\J$ satisfies all of the previous properties simultaneously, and it is only necessary to check that $\J$ is $(3 \theta^{1/2}, \mathcal{S})$-avoiding.

Let $t$ be the number of clusters of $\cP$ and $m$ the size of a cluster in~$\cP$.
For each $(k-1)$-set of clusters $Y$, $\J_Y$ has $(1 \pm \eps_k/10) p m^{k-1}$ edges (see \cite[Fact 7]{AllenBottcherCooleyMycroft2017}).
We say a $(k-1)$-set of clusters $Y$ is \emph{bad} if $|\J_Y \cap E(\mathcal{S})| > \sqrt{6 \theta} |\J_Y|$ and let $\Y$ be the set of bad $(k-1)$-sets.
Then \begin{align*}
3 \theta p \binom{n}{k-1} & \ge \sum_{Y} |\J_Y \cap E(\mathcal{S})| \ge |\Y| \sqrt{6 \theta} (1 - \eps_k/10) p m^{k-1},
\end{align*} which implies $|\Y| \leq 3 \theta^{1/2} \binom{t}{k-1}$.
It follows that $\J$ is $(3 \theta^{1/2}, \mathcal{S})$-avoiding, as desired.

\subsection{Embedding lemma} \label{appendix:embedding}

Note that \cite[Theorem 2]{CooleyFountoulakisKuehnEtAl2009} is stronger than Lemma~\ref{lemma:embeddinglemma} in the sense that it allows embeddings of $k$-graphs with bounded maximum degree whose number of vertices is linear in~$m$, but we don't require that property here.

The main technical difference between Lemma~\ref{lemma:embeddinglemma} and Theorem 2 in \cite{CooleyFountoulakisKuehnEtAl2009} is that their lemma asks for the stronger condition that for all $e \in E(\h)$ intersecting the vertex classes $\{ X_{i_j} : 1 \leq j \leq k \}$, the $k$-graph $G$ should be $(d, \eps_k, r)$-regular with respect to the $k$-set of clusters $\{ V_{i_j} : 1 \leq j \leq k \}$, such that the value $d$ does not depend on $e$, and $1/d \in \mathbb{N}$; where as we allow $G$ to be $(d_e, \eps_k, r)$-regular for some $d_e \ge d$ depending on $e$ and not necessarily satisfying $1/d_e \in \mathbb{N}$.
By the discussion after Lemma~4.6 in \cite{KuehnMycroftOsthus2010}, we can reduce to that case by working with a sub-$k$-complex of $\J \cup G$ which is $(d, d_{k-1}, d_{k-2}, \dotsc, d_2, \eps_k, \eps, r)$-regular, whose existence is guaranteed by an application of the ``slicing lemma'' \cite[Lemma 8]{CooleyFountoulakisKuehnEtAl2009}.

\end{document}